\documentclass[11pt]{article}
 \usepackage{theorem}
 \newtheorem{theorem}{Theorem}
 \newtheorem{lemma}[theorem]{Lemma}

 \newenvironment{proof}{\begin{trivlist}
   \item[\hskip\labelsep{\it Proof.}]}{$\hfill\square$\end{trivlist}}
 \newenvironment{proofof}[1]{\begin{trivlist}
		\item[\hskip\labelsep{\bf Proof of {#1}.}]}{$\hfill\Box$\end{trivlist}}
 \usepackage[hidelinks]{hyperref}
 \usepackage[round]{natbib}
\usepackage{latexsym,amsfonts,amsmath,amssymb,url,mathrsfs}
\usepackage{graphicx}
\usepackage{subcaption}
\usepackage[usenames,dvipsnames,svgnames]{xcolor}
\usepackage{tikz, pgf, pgfplots}

\newcommand{\indx}{\mathcal{I}}
\newcommand{\satop}[2]{\stackrel{\scriptstyle{#1}}{\scriptstyle{#2}}}

\newcommand{\bsalpha}{{\boldsymbol{\alpha}}}
\newcommand{\bsDelta}{{\boldsymbol{\Delta}}}

\newcommand{\bse}{{\boldsymbol{e}}}
\newcommand{\bsgamma}{{\boldsymbol{\gamma}}}
\newcommand{\bsbeta}{{\boldsymbol{\beta}}}

\newcommand{\bsrho}{{\boldsymbol{\rho}}}
\newcommand{\bsnu}{{\boldsymbol{\nu}}}
\newcommand{\bsh}{{\boldsymbol{h}}}

\newcommand{\bsell}{{\boldsymbol{\ell}}}

\newcommand{\bsu}{{\boldsymbol{u}}}

\newcommand{\bsx}{{\boldsymbol{x}}}

\newcommand{\bsy}{{\boldsymbol{y}}}
\newcommand{\bsz}{{\boldsymbol{z}}}
\newcommand{\bsw}{{\boldsymbol{w}}}

\newcommand{\bsm}{{\boldsymbol{m}}}

\newcommand{\bszero}{{\boldsymbol{0}}}

\newcommand{\rd}{{\mathrm{d}}}
\newcommand{\ri}{{\mathrm{i}}}
\newcommand{\bbN}{{\mathbb{N}}}

\newcommand{\bbB}{{\mathbb{B}}}
\newcommand{\bbC}{{\mathbb{C}}}
\newcommand{\bbR}{{\mathbb{R}}}
\newcommand{\bbZ}{{\mathbb{Z}}}

\newcommand{\bbT}{{\mathbb{T}}}
\newcommand{\calA}{{\mathcal{A}}}

\newcommand{\calI}{{\mathcal{I}}}

\newcommand{\calE}{{\mathcal{E}}}

\newcommand{\calH}{{\mathcal{H}}}
\newcommand{\calO}{{\mathcal{O}}}

\newcommand{\calQ}{{\mathcal{Q}}}

\newcommand{\calS}{{\mathcal{S}}}

\newcommand{\mask}[1]{}
\DeclareMathOperator*{\argmin}{\mathrm{arg\,min}}
\newcommand{\supp}{{\mathrm{supp}}}
\newcommand{\setu}{{\mathfrak{u}}}

\definecolor{darkred}{RGB}{139,0,0}
\definecolor{darkgreen}{RGB}{0,100,0}
\definecolor{darkmagenta}{RGB}{180,0,180}
\definecolor{darkblue}{RGB}{0,0,190}
\definecolor{darkorange}{RGB}{180,60,0}
\definecolor{NVgreen}{rgb}{0.462745098, 0.725490196, 0.000000000}

\definecolor{HighlightColor}{rgb}{0.462745098, 0.725490196, 0.000000000}
\definecolor{Emerald}{rgb}{0.0, 0.521568627, 0.392156863}
\definecolor{Amethyst}{rgb}{0.364705882, 0.08627451, 0.509803922}
\definecolor{CPUBlue}{rgb}{0.0, 0.443137255, 0.77254902}
\definecolor{Garnet}{rgb}{0.537254902, 0.047058824, 0.345098039}
\definecolor{Fluorite}{rgb}{0.980392157, 0.760784314, 0.0}
\definecolor{LightGray}{rgb}{0.803921569, 0.803921569, 0.803921569}
\definecolor{MediumGray}{rgb}{0.549019608, 0.549019608, 0.549019608}
\definecolor{DarkGray}{rgb}{0.368627451, 0.368627451, 0.368627451}

\usepackage{xspace}
\newcommand{\FNO}{FNO\xspace}
\newcommand{\HCFNO}{FNO-HC\xspace}
\newcommand{\HCLATFNO}{FNO-HC-LAT\xspace}
\newcommand{\LATFNO}{FNO-LAT\xspace}

\definecolor{Chart1}{HTML}{0072B2}    
\definecolor{Chart2}{HTML}{D55E00}    
\definecolor{Chart3}{HTML}{E69F00}    
\definecolor{Chart4}{HTML}{56B4E9}    
\definecolor{Chart5}{HTML}{000000}    

\usepackage{pgfplots}
\usepgfplotslibrary{groupplots}
\usepackage{xcolor}
\pgfplotsset{compat=1.18}

\definecolor{FNOColor}{HTML}{FF0000}
\definecolor{LATFNOColor}{HTML}{0000FF}
\definecolor{HCFNOColor}{HTML}{952FC6}
\definecolor{HCLATFNOColor}{HTML}{227700}

\pgfplotsset{
  comparison plot/.style={
    width=0.46\textwidth,
    height=0.34\textwidth,
    tick align=outside,
    tick pos=left,
    line width=1pt,
    grid=both,
    major grid style={draw=gray!20},
    minor grid style={draw=gray!10},
    xlabel near ticks,
    ylabel near ticks,
    legend columns=4,
    legend style={
      draw=none,
      fill=none,
      font=\small,
      /tikz/every even column/.append style={column sep=0.8em},
    },
  },
  comparison error bars/.style={
    error bars/.cd,
      x dir=both,
      x fixed=0,
      y dir=both,
      y explicit,
      error bar style={solid, line width=0.8pt},
  },
  fno curve/.style={very thick, loosely dotted, color=FNOColor, mark=none},
  latfno curve/.style={very thick, dashed, color=LATFNOColor, mark=none},
  hcfno curve/.style={very thick, densely dashdotted, color=HCFNOColor, mark=none},
  hclatfno curve/.style={very thick, solid, color=HCLATFNOColor, mark=none},
}

\begin{document}

\title{Fourier Neural Operators with Rank-1 Lattice Points and Hyperbolic Cross}

 \author{Jakob Dilen, Alexander Keller, Frances Y. Kuo, Dirk Nuyens}
 \date{June 2026}

\maketitle

\begin{abstract}
The \emph{Fourier neural operator} (FNO) is a neural network architecture that
learns mappings between function spaces. Its efficient implementation is
based on the multi-dimensional Fourier transform. By deriving general regularity bounds for the FNO with respect to both the spatial and parametric variables, we prove that the generalization error of the FNO can be improved by replacing spatial tensor product grids with purpose-built rank-1 lattice points, and by using a second lattice carefully constructed as training points in the parametric space. We achieve more
accurate and efficient approximations from fewer network parameters, fewer spatial points, and fewer training samples. 
In addition, the architecture is simplified, because the high-dimensional
Fourier transform on rank-1 lattices requires only a \emph{one-dimensional fast Fourier transform}, and we can use a \emph{hyperbolic cross} frequency index set with lattice points.
We demonstrate the benefits of our \emph{lattice-based hyperbolic-cross FNOs} for an elliptic PDE on the torus.
\end{abstract}


\section{Introduction}

We consider the problem of solving the Data-to-Observables map
with \emph{operators} between two infinite-dimensional spaces:
\begin{equation} \label{eq:dto1}
\mbox{Given data $a\in A$, compute observable $G(a)\in Z$},
\end{equation}
where $A$ and $Z$ are separable Banach spaces of functions on a bounded
physical domain $D\subset\bbR^d$ taking values in $\bbR^{d_a}$ and $\bbR^{d_u}$,
respectively. As an example, consider the elliptic partial differential
equation (PDE)
\begin{align} \label{eq:elliptic}
 - \nabla \cdot(a \nabla u) = f \quad\mbox{in } D,
\end{align}
possibly together with some boundary condition on $\partial D$. Here 
$a$ is the input scalar permeability field, $f$ is the forcing term, and the PDE solution $u$ is the scalar pressure field. The observable may be $G(a) = u$ so that $d_a = d_u = 1$, and we may take 
$Z = H^1(D)$ to be the standard first-order Sobolev space on the physical domain $D$.
As another observable, $G(a) = \nabla u$ might be the velocity field, in which case $d_a$ = 1 and $d_u = d$. 

The aim is to learn an approximation $G_\theta : A\to Z$ to the operator $G$ with parameters~$\theta$ from some finite-dimensional parameter space~$\Theta$. Following \citet{LKALBSA21} and \citet{KLLABSA23}, we consider a \emph{Fourier neural operator} (\emph{FNO}) as our approximation $G_\theta$. The new contribution of this paper is that we deploy \emph{quasi-Monte Carlo} \citep{Nie92,DKS13} methods on two fronts: we use a \emph{rank-$1$ lattice} \citep{SJ94,DKP22} with respect to the physical variable $\bsx$ to generate efficient data points in the physical domain, and we use a \emph{second, different rank-$1$ lattice} with respect to the parametric variable~$\bsy$ to generate effective training points in the parametric domain \citep{KKNS26a,KKNS26b}. We also make use of a \emph{hyperbolic cross} in the frequency domain to accompany our lattice points in the physical domain \citep{CKN10,CKNS20,CKNS21,KMN25}. In more detail, our three contributions are:
\begin{enumerate}
\item Observations are taken at a discrete set of data points on the physical domain $D$. This is typically taken to be a regular grid on $D$ and can be computationally expensive for a fine grid in dimension $d=3$ or higher. We \textbf{replace a regular grid by a specifically selected rank-$1$ lattice in the physical domain~$D$}. This drastically decreases the number of required spatial points while maintaining similar accuracy.

\item The input field $a(\bsx,\bsy)$ for $\bsx\in D$ is typically generated by a series expansion in terms of random parameters $\bsy$ from some parametric space $Y$, e.g., an \emph{affine} random field model or a \emph{Chebyshev} model (see Section~\ref{sec:random} ahead). In this case, we can view $a(\bsx,\bsy)$ 
as a \emph{hard-coded initial layer} of the neural operator. We \textbf{replace random parametric sampling by a carefully constructed rank-$1$ lattice as training points in the parametric domain~$Y$}, to improve the convergence of the generalization error.

\item We assume that we are approximating functions with a certain prescribed smoothness, characterized by the decay of their Fourier coefficients. This justifies the truncation of the frequency domain $\bbZ^d$ to a finite index set $\calA$. Typically this is taken to be an axis-parallel box, which bounds all components in an index $\bsh\in\calA$ by some prescribed magnitude. For suitable classes of functions, \textbf{replacing such a box by a hyperbolic cross in the frequency domain~$\bbZ^d$} allows us to use much fewer indices while maintaining similar accuracy.
\end{enumerate}

Solving PDEs using model order reduction techniques to approximate the underlying operator has been the main focus of a lot of research \citep{SAP19,BHKS21,KLY21,PMRDJ23,NSS24}. One of its main branches of research has been the use of \emph{neural networks} \citep{RPK19,BHKS21,KLY21,PMRDJ23}. More recently some of this focus has shifted to address the fact that these operators act on function spaces instead of vectors or tensors, resulting in the development of the \emph{neural operator} \citep{LKAKBSA20a,LKALBSA20b,LKALBSA21,LJPZK21,PSL22,KLLABSA23,Lan23}. Some promising examples are the \emph{PCA-net} \citep{Lan23} that makes use of the principal component analysis, and the \emph{DeepONet} \citep{LJPZK21,PSL22} that uses a projecting-lifting structure by combining two models.

The architecture of choice in this paper is the \emph{Fourier neural operator} or \emph{FNO} \citep{LKALBSA21,KLLABSA23}, which uses the properties of convolutions in the Fourier frequency space to apply integral kernel operators in an efficient manner. Aside from its empirical performance, the theoretical aspects of the FNO have recently also been studied, including the \textit{universal approximation theorem} \citep{KLM21} for the FNO, as well as the role that non-locality plays herein \citep{LLS25}. While the universal approximation theorem proves that a neural network/operator can represent a function, \emph{expression rates} \citep{SZ19,SZ23,HSZ24} measure the efficiency of a network architecture and specify how big the network needs to be.
Other extensions and alternatives to the FNO include: the \emph{adaptive Fourier neural operator} or \emph{AFNO} \citep{JMZAAB22} that was initially developed for use in image analysis; the \emph{Laplace neural operator} or \emph{LNO} \citep{CGK24} which is an alternative to the FNO that specializes in handling transient behavior; and the \emph{spherical Fourier neural operator} or \emph{SFNO} \citep{BKHPBKA23,FourCastNet3} for learning operators on spherical geometries, well-suited for forecasting atmospheric dynamics.  

The practical performance of FNO is still a long way from the theory \citep{LST25,LMZ24}, just as observed for the neural network \citep{Cyb89,HSW89,Hor91,Voi23}. For related research on how to address this problem for neural networks, we refer to the review of \citet{ABDM24}. Addressing this problem for the FNO is part of the focus of the current paper. 

\emph{Quasi-Monte Carlo} points have been deployed as training points instead of random samples for deep neural networks \citep{LMR20,MR21,LMRS21,KKNS26a,KKNS26b}. They can reduce the \emph{generalization error} (for unseen data) via a proven, better convergence rate for the \emph{generalization gap}, which is the difference between the generalization error and the \emph{training error} observed in practice. The convergence theory requires knowledge of the \emph{parametric regularity} of the neural network, i.e., the mixed derivatives of the neural network with respect to the parametric variables~$\bsy$. Hence, \citet{KKNS26a} derived explicit parametric regularity bounds for general deep neural networks. The bounds depend on the network parameters $\theta$ and the choice of a smooth activation function, with explicit constants. These bounds are fully general and remain independent of both the target function and the training data. \citet{KKNS26a} also devised a \emph{tailored regularization} strategy during training to encourage the network parameters to match the regularity bounds of a given target function. \citet{KKNS26b} further extended the theory to a generalized sigmoid activation function which in the limit becomes ReLU. The philosophy of \citet{KKNS26a,KKNS26b} is very similar to the \emph{derivative informed neural operator} or \emph{DINO}  \citep{OCVG24} where training is done with gradient matching in mind.

Following the framework in \citet{KKNS26a,KKNS26b}, in this paper we prove general regularity bounds for the FNO, with respect to not only the parametric variables $\bsy$ but now also the physical variables $\bsx$. Instead of a regularization strategy, the regularity information of the target function is \emph{hard-coded as an initial layer} of our FNO. These regularity bounds together with known theory of \emph{lattice rules} for integration \citep{DKS13,DKP22} allow us to conclude a theoretical convergence rate for the \emph{generalization error}, see our main result--Theorem~\ref{thm:main} below. 
The task of establishing the \emph{universal approximation theorem} for our ``\emph{lattice-based hyperbolic-cross FNO}'' is the focus of a subsequent paper \citet{DKN26}, which will build on known theory of lattice-based algorithms for function approximation \citep{CKNS20,CKNS21,KMN25}. 

Lattice rules and other quasi-Monte Carlo methods have been successfully applied, with full theoretical justification, to problems in uncertainty quantification involving parametric PDEs. Some recent examples are wave propagation and scattering \citep{GKNSS26}, tumor growth and treatment \citep{GKNPSW26}, PDE-constrained optimal control \citep{GKKSS24}, PDEs with random domains \citep{HHKKS24}, probability density estimation \citep{GKS25}, kernel interpolation \citep{KKKNS22}; see~\cite{KN16} for a survey of earlier theoretical works. In all cases we have \emph{fast} construction \citep{NC06} of the points to achieve good convergence rates with errors bounded independently of the parametric dimension.

In the case of the FNO, lattice points are particularly beneficial for the physical space because they can \emph{replace computing the $d$-dimensional FFT} on a regular grid \emph{by a $1$-dimensional FFT} \citep{KKP10,KPV15,KPV21,KMNN21}. Further more, the frequency index set can be reduced to a \emph{hyperbolic cross} under favorable conditions \citep{CKNS20,CKNS21,KMN25}. 

The structure of this paper is as follows. In Section~\ref{sec:FNO} we define
the Fourier neural operator (FNO), explain how to apply the hyperbolic cross in the frequency domain,
and introduce the application of the Fourier transform on rank-1 lattices within the FNO.
In Section~\ref{sec:gen} we discuss the generalization and training errors for our FNO, and specify our random field models for the scalar field $a(\bsx,\bsy)$ and the theoretical settings for error analysis.
We then analyze the regularity with respect to the parametric
variables $\bsy$ in Section~\ref{sec:reg-y} and with respect to the spatial variables $\bsx$ in
Section~\ref{sec:reg-x}, ending with our combined generalization error bound in our main result--Theorem~\ref{thm:main}. We present our numerical results in
Section~\ref{sec:experiments} and finally draw the conclusion in
Section~\ref{sec:conclusion}. The technical proofs of our theorems are given in
Appendix~\ref{sec:proofs}.

\section{Fourier Neural Operator (FNO)} \label{sec:FNO}

For convenience, we fix the spatial/physical domain as
\[
  D := [0,1]^d.
\]
Now consider an FNO of the form
\begin{align} \label{eq:no1}
 G_\theta(a)(\bsx)
 := (Q \circ T_L \circ T_{L-1} \circ \cdots \circ  T_2 \circ T_1 \circ P(a))(\bsx),
 \quad \bsx\in D,
\end{align}
where $P\in\bbR^{d_1\times d_a}$ and $Q\in
\bbR^{d_u\times d_{L+1}}$ are local linear \emph{lifting} and \emph{projecting}
operators, respectively. 
Each \emph{hidden layer} $T_\ell$ for $\ell=1,\ldots, L$ is an operator of
the form
\begin{align} \label{eq:def-T}
  T_\ell(v)(\bsx) := \sigma \big(W_\ell\, v(\bsx) + (K_\ell\, v)(\bsx) + b_\ell\big)
  \quad\mbox{for}\quad v: D\to\bbC^{d_\ell},
 \quad \bsx\in D,
\end{align}
with a local \emph{weight} matrix $W_\ell \in \bbR^{d_{\ell+1} \times
d_\ell}$, a \emph{bias} vector $b_\ell \in \bbR^{d_{\ell+1}}$, a
scalar \emph{activation function} $\sigma :\bbC\to\bbC$ \citep{MB22} to be applied component-wise, and an \emph{integral kernel operator}
\begin{align} \label{eq:ker-int}
  (K_\ell\, v)(\bsx) := \int_D \kappa_\ell(\bsx-\bsx')\,v(\bsx')\,\rd\bsx',
 \qquad\bsx\in D,
\end{align}
where $\kappa_\ell : D \to \bbC^{d_{\ell+1}\times d_\ell}$ is a matrix-valued kernel function. 
The integration is carried out for each element in the vector. We note that the nonlinear operator $T_\ell$ maps functions $v : D\to \bbC^{d_\ell}$ to functions $T_\ell(v) : D \to \bbC^{d_{\ell+1}}$. 
In our notation we emphasize the application of nonlinear operators by using parentheses as in $T_\ell(v)$, in contrast to linear operators like $W_\ell\, v$ and $K_\ell\, v$; these can then be evaluated in a point $\bsx$ like $T_\ell(v)(\bsx)$, $W_\ell\, v(\bsx)$ and $(K_\ell\, v)(\bsx)$.
Popular activation functions include the rectified linear unit (ReLU) $\sigma(z) = \max(z,0)$, the sigmoid function $\sigma(z) = 1/(1+e^{-z})$, the hyperbolic tangent $\sigma(z) = \tanh(z)$, and the swish function $\sigma(z) = z/(1+e^{-z})$.
However, as we are working with complex-valued functions, the activation functions are applied on the real and imaginary parts separately, i.e., $\sigma(z) = \sigma(\text{Re}(z)) + \ri\,\sigma(\text{Im}(z))$.

In this FNO framework, all functions are represented by their Fourier series, including the kernel $\kappa_\ell$ in each hidden layer. For a function $v \in L_1(D)$, we define its Fourier coefficients by
\begin{align} \label{eq:FourierCoeff}
  \widehat{v}(\bsh) := \int_D v(\bsx) \, e^{-2\pi\ri\bsh\cdot\bsx} \, \rd\bsx,
  \qquad \bsh\in\bbZ^d.
\end{align}
When $v$ is matrix-valued or vector-valued, the integral is applied component-wise.

In practice, the frequency domain needs to be restricted to a finite frequency index set $\calA\subset\bbZ^d$, typically a box centered at the origin. We denote all the network parameters to be ``learned'' collectively as
\begin{align*}
  \theta := \Big(P, \big(W_\ell, 
  (\widehat{\kappa_{\ell}}(\bsh))_{\bsh\in\calA}, 
  b_{\ell}\big)_{\ell=1}^L, Q\Big),
\end{align*}
with Fourier coefficients $\widehat{\kappa_{\ell}}(\bsh)\in\bbC^{d_{\ell+1}\times d_\ell}$. Therefore, the total number of real-valued network parameters (multiplying all complex parameters by $2$) is 
\begin{align} \label{eq:N-FNO}
  N_{\rm FNO} := d_1\, d_a 
  + \sum_{\ell=1}^L (d_{\ell+1}\, d_\ell + 2\,d_{\ell+1}\, d_\ell\,|\calA|
  + d_{\ell+1} ) + d_u\, d_{L+1}, 
\end{align}
where $|\calA|$ denotes the cardinality of the frequency index set $\calA$.
The set of all permissible network parameters is denoted by $\Theta$. The \emph{depth} $L$, the \emph{widths} $d_\ell$, the frequency index set $\calA$, and other parameters used in the ``training'' phase (see below) are known as the \emph{hyperparameters} of the~FNO. 

The process of ``training'' is to identify suitable network parameters $\theta\in\Theta$ via the minimization problem
\[
  \theta^* = \argmin_{\theta\in\Theta}
  \bigg(\frac{1}{n} \sum_{i=1}^n \|G(a_i) - G_\theta(a_i)\|_{L_2(D)}^2\bigg),
\]
often with an additional regularization term. The spatial $L_2(D)$ norm would need to be approximated via a spatial discretization, typically a regular grid or finite elements in $d$ dimensions. 

In machine learning, each of the output dimensions $\{d_1,\ldots,d_{L+1}\}$ is called a \emph{channel}. The number of these intermediate output dimensions may vary, as $d_\ell = d_{\ell+1}$ does not have to hold. While in practice the number of output dimensions is kept the same across the hidden layers, the input dimension $d_a$ is often blown up or \textit{lifted} by the lifting operator $P \in \bbR^{d_1\times d_a}$, after which the output dimension $d_{L+1}$ of the hidden layers is \emph{projected} back to the desired output dimension $d_u$ by the projecting operator $Q \in \bbR^{d_u\times d_{L+1}}$. This is done under the assumption that some underlying property of the function can be better captured by the model in a higher dimensional representation.
The number of intermediate output dimensions $d_\ell$ that is used is one of the hyperparameters of the model that has to be selected in advance.

Instead of using \emph{linear} lifting and projecting operators $P$ and $Q$ as we have shown here, these operators are in practice often taken to be \emph{nonlinear} shallow neural networks. In Appendix~\ref{sec:nonlinearPQ} we will explain how the theory in this paper can be adapted for nonlinear $P$ and~$Q$.

\subsection{Evaluating the Integral Kernel}

Formulating the integral operator as a convolution in \eqref{eq:ker-int} allows us to take advantage of the fast Fourier transform:
Using the convolution theorem and the assumption that $\kappa_\ell: D\to\bbC^{d_{\ell+1}\times d_\ell}$ in \eqref{eq:ker-int} is a shift-invariant periodic matrix-valued function with absolutely converging Fourier series, and $v: D\to \bbC^{d_\ell}$ is $L_1$ measurable, we can rewrite \eqref{eq:ker-int} as
\begin{align} \label{eq:Kv1}
 (K_\ell\, v)(\bsx)
 &= \int_D \kappa_\ell(\bsx-\bsx')\,v(\bsx')\,\rd\bsx'
 = \bigg[\sum_{q=1}^{d_\ell} \int_D \kappa_{\ell,p,q}(\bsx-\bsx')\, v_q(\bsx')\,\rd\bsx'   
 \bigg]_{p=1}^{d_{\ell+1}} \nonumber\\
 &=  \bigg[\sum_{q=1}^{d_\ell} \sum_{\bsh\in\bbZ^d}(\widehat{\kappa_\ell})_{p,q}(\bsh) \, 
 e^{2\pi\ri\bsh\cdot\bsx} \int_D v_q(\bsx')\,e^{-2\pi\ri\bsh\cdot\bsx'} \,\rd\bsx' 
 \bigg]_{p=1}^{d_{\ell+1}} \nonumber\\
 &= \sum_{\bsh\in\bbZ^d} \bigg[ \sum_{q=1}^{d_\ell} (\widehat{\kappa_\ell})_{p,q}(\bsh)\,
 \widehat{v}_q(\bsh) \bigg]_{p=1}^{d_{\ell+1}} \, e^{2\pi\ri\bsh\cdot\bsx} 
 = \sum_{\bsh\in\bbZ^d} \widehat{\kappa_\ell}(\bsh) \, \widehat{v}(\bsh) \, e^{2\pi\ri\bsh\cdot\bsx}.
\end{align}
The coefficients $(\widehat{\kappa_\ell})_{p,q}(\bsh)$ are
``learned'', while the coefficients $\widehat{v}_q(\bsh)$ are evaluated from
the input function~$v$ entering each hidden layer.

The essence of the operator approach is that the expression
\eqref{eq:Kv1} holds for all $\bsx\in D$, and not only for the discretized
grid points. More related details are given in Section~\ref{sec:discr} below.

\subsection{Truncating the Frequency Index Set} \label{sec:trunc}

It is not possible to evaluate the infinite sum in \eqref{eq:Kv1}. Moreover, it is neither possible to learn infinitely many parameters $\widehat{\kappa_\ell}(\bsh)$ nor evaluate infinitely many coefficients~$\widehat{v}(\bsh)$.
From \citet{KLM21} 
we know that it suffices to take kernels $\kappa_\ell$ with finite support. Hence, we restrict the index set to 
\[ 
  \calA := \big\{\bsh\in\bbZ^d : r(\bsh) \le M \big\}, \qquad M> 0,
\]
where $r(\bsh)$ measures a certain kind of ``radius'' of an index $\bsh$ from the origin. 

The standard strategy is to truncate to a box
\[
 \calA^{\rm box} := \{\bsh\in\bbZ^d : |h_j| \le M \mbox{ for all } j=1,\ldots,d\},
\]
with $|\calA^{\rm box}| = (1+2M)^d$ when $M\in\bbZ^+$, which corresponds to $r^{\rm box}(\bsh) := \|\bsh\|_\infty$. In relation to our function space setting, in this paper we replace a box index set by a \emph{hyperbolic cross}
\[
 \calA^{\rm hc} := \Big\{\bsh\in\bbZ^d : \prod_{j=1}^d \max(1,|h_j|) \le M \Big\};
\]
a more general form of this is provided in Section~\ref{sec:func} (see \eqref{eq:norm-per-rD}). Figure~\ref{fig:index_sets} shows both the box index set and the hyperbolic cross index set for $M=10$.  

In summary, we approximate
\begin{align} \label{eq:Kv2}
  (K_\ell\, v)_p(\bsx) &\approx
  \sum_{\bsh\in\calA} \sum_{q=1}^{d_\ell} (\widehat{\kappa_\ell})_{p,q}(\bsh)\;\widehat{v}_q(\bsh)\, e^{2\pi\ri\bsh\cdot\bsx},  
 \qquad p=1,\ldots,d_{\ell+1}, 
 \quad \bsx\in D. 
\end{align}

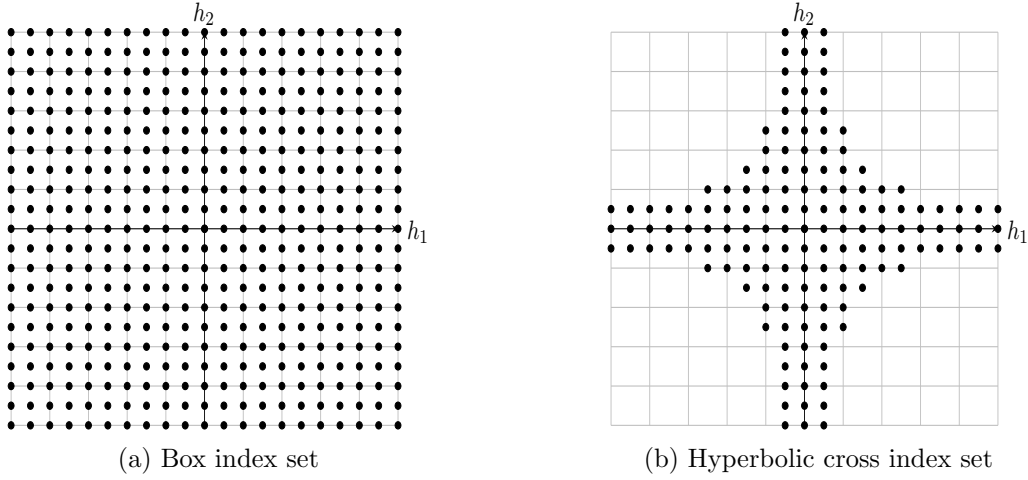
\begin{figure}[t]
    \centering
        \begin{subfigure}[b!]{0.4\textwidth}
        \centering
        \resizebox{0.9\textwidth}{0.9\textwidth}{
        \begin{tikzpicture}
        
        \begin{axis}[{grid=both,ymin=-2.5,ymax=2.5,xmax=2.5,xmin=-2.5,xticklabel=\empty,yticklabel=\empty,
               minor tick num=1,axis lines = middle,xlabel=$h_1$,ylabel=$h_2$,label style =
               {at={(ticklabel cs:1.1)}}}]
           
                \addplot[
                only marks,
                mark=*,
                color=black,
                mark size=1.5pt
            ] table[row sep=crcr] {
                x y \\
                0.0 0.0 \\
-2.5 -2.5 \\
-2.5 -2.25 \\
-2.5 -2.0 \\
-2.5 -1.75 \\
-2.5 -1.5 \\
-2.5 -1.25 \\
-2.5 -1.0 \\
-2.5 -0.75 \\
-2.5 -0.5 \\
-2.5 -0.25 \\
-2.5 0.0 \\
-2.5 0.25 \\
-2.5 0.5 \\
-2.5 0.75 \\
-2.5 1.0 \\
-2.5 1.25 \\
-2.5 1.5 \\
-2.5 1.75 \\
-2.5 2.0 \\
-2.5 2.25 \\
-2.5 2.5 \\
-2.25 -2.5 \\
-2.25 -2.25 \\
-2.25 -2.0 \\
-2.25 -1.75 \\
-2.25 -1.5 \\
-2.25 -1.25 \\
-2.25 -1.0 \\
-2.25 -0.75 \\
-2.25 -0.5 \\
-2.25 -0.25 \\
-2.25 0.0 \\
-2.25 0.25 \\
-2.25 0.5 \\
-2.25 0.75 \\
-2.25 1.0 \\
-2.25 1.25 \\
-2.25 1.5 \\
-2.25 1.75 \\
-2.25 2.0 \\
-2.25 2.25 \\
-2.25 2.5 \\
-2.0 -2.5 \\
-2.0 -2.25 \\
-2.0 -2.0 \\
-2.0 -1.75 \\
-2.0 -1.5 \\
-2.0 -1.25 \\
-2.0 -1.0 \\
-2.0 -0.75 \\
-2.0 -0.5 \\
-2.0 -0.25 \\
-2.0 0.0 \\
-2.0 0.25 \\
-2.0 0.5 \\
-2.0 0.75 \\
-2.0 1.0 \\
-2.0 1.25 \\
-2.0 1.5 \\
-2.0 1.75 \\
-2.0 2.0 \\
-2.0 2.25 \\
-2.0 2.5 \\
-1.75 -2.5 \\
-1.75 -2.25 \\
-1.75 -2.0 \\
-1.75 -1.75 \\
-1.75 -1.5 \\
-1.75 -1.25 \\
-1.75 -1.0 \\
-1.75 -0.75 \\
-1.75 -0.5 \\
-1.75 -0.25 \\
-1.75 0.0 \\
-1.75 0.25 \\
-1.75 0.5 \\
-1.75 0.75 \\
-1.75 1.0 \\
-1.75 1.25 \\
-1.75 1.5 \\
-1.75 1.75 \\
-1.75 2.0 \\
-1.75 2.25 \\
-1.75 2.5 \\
-1.5 -2.5 \\
-1.5 -2.25 \\
-1.5 -2.0 \\
-1.5 -1.75 \\
-1.5 -1.5 \\
-1.5 -1.25 \\
-1.5 -1.0 \\
-1.5 -0.75 \\
-1.5 -0.5 \\
-1.5 -0.25 \\
-1.5 0.0 \\
-1.5 0.25 \\
-1.5 0.5 \\
-1.5 0.75 \\
-1.5 1.0 \\
-1.5 1.25 \\
-1.5 1.5 \\
-1.5 1.75 \\
-1.5 2.0 \\
-1.5 2.25 \\
-1.5 2.5 \\
-1.25 -2.5 \\
-1.25 -2.25 \\
-1.25 -2.0 \\
-1.25 -1.75 \\
-1.25 -1.5 \\
-1.25 -1.25 \\
-1.25 -1.0 \\
-1.25 -0.75 \\
-1.25 -0.5 \\
-1.25 -0.25 \\
-1.25 0.0 \\
-1.25 0.25 \\
-1.25 0.5 \\
-1.25 0.75 \\
-1.25 1.0 \\
-1.25 1.25 \\
-1.25 1.5 \\
-1.25 1.75 \\
-1.25 2.0 \\
-1.25 2.25 \\
-1.25 2.5 \\
-1.0 -2.5 \\
-1.0 -2.25 \\
-1.0 -2.0 \\
-1.0 -1.75 \\
-1.0 -1.5 \\
-1.0 -1.25 \\
-1.0 -1.0 \\
-1.0 -0.75 \\
-1.0 -0.5 \\
-1.0 -0.25 \\
-1.0 0.0 \\
-1.0 0.25 \\
-1.0 0.5 \\
-1.0 0.75 \\
-1.0 1.0 \\
-1.0 1.25 \\
-1.0 1.5 \\
-1.0 1.75 \\
-1.0 2.0 \\
-1.0 2.25 \\
-1.0 2.5 \\
-0.75 -2.5 \\
-0.75 -2.25 \\
-0.75 -2.0 \\
-0.75 -1.75 \\
-0.75 -1.5 \\
-0.75 -1.25 \\
-0.75 -1.0 \\
-0.75 -0.75 \\
-0.75 -0.5 \\
-0.75 -0.25 \\
-0.75 0.0 \\
-0.75 0.25 \\
-0.75 0.5 \\
-0.75 0.75 \\
-0.75 1.0 \\
-0.75 1.25 \\
-0.75 1.5 \\
-0.75 1.75 \\
-0.75 2.0 \\
-0.75 2.25 \\
-0.75 2.5 \\
-0.5 -2.5 \\
-0.5 -2.25 \\
-0.5 -2.0 \\
-0.5 -1.75 \\
-0.5 -1.5 \\
-0.5 -1.25 \\
-0.5 -1.0 \\
-0.5 -0.75 \\
-0.5 -0.5 \\
-0.5 -0.25 \\
-0.5 0.0 \\
-0.5 0.25 \\
-0.5 0.5 \\
-0.5 0.75 \\
-0.5 1.0 \\
-0.5 1.25 \\
-0.5 1.5 \\
-0.5 1.75 \\
-0.5 2.0 \\
-0.5 2.25 \\
-0.5 2.5 \\
-0.25 -2.5 \\
-0.25 -2.25 \\
-0.25 -2.0 \\
-0.25 -1.75 \\
-0.25 -1.5 \\
-0.25 -1.25 \\
-0.25 -1.0 \\
-0.25 -0.75 \\
-0.25 -0.5 \\
-0.25 -0.25 \\
-0.25 0.0 \\
-0.25 0.25 \\
-0.25 0.5 \\
-0.25 0.75 \\
-0.25 1.0 \\
-0.25 1.25 \\
-0.25 1.5 \\
-0.25 1.75 \\
-0.25 2.0 \\
-0.25 2.25 \\
-0.25 2.5 \\
0.0 -2.5 \\
0.0 -2.25 \\
0.0 -2.0 \\
0.0 -1.75 \\
0.0 -1.5 \\
0.0 -1.25 \\
0.0 -1.0 \\
0.0 -0.75 \\
0.0 -0.5 \\
0.0 -0.25 \\
0.0 0.0 \\
0.0 0.25 \\
0.0 0.5 \\
0.0 0.75 \\
0.0 1.0 \\
0.0 1.25 \\
0.0 1.5 \\
0.0 1.75 \\
0.0 2.0 \\
0.0 2.25 \\
0.0 2.5 \\
0.25 -2.5 \\
0.25 -2.25 \\
0.25 -2.0 \\
0.25 -1.75 \\
0.25 -1.5 \\
0.25 -1.25 \\
0.25 -1.0 \\
0.25 -0.75 \\
0.25 -0.5 \\
0.25 -0.25 \\
0.25 0.0 \\
0.25 0.25 \\
0.25 0.5 \\
0.25 0.75 \\
0.25 1.0 \\
0.25 1.25 \\
0.25 1.5 \\
0.25 1.75 \\
0.25 2.0 \\
0.25 2.25 \\
0.25 2.5 \\
0.5 -2.5 \\
0.5 -2.25 \\
0.5 -2.0 \\
0.5 -1.75 \\
0.5 -1.5 \\
0.5 -1.25 \\
0.5 -1.0 \\
0.5 -0.75 \\
0.5 -0.5 \\
0.5 -0.25 \\
0.5 0.0 \\
0.5 0.25 \\
0.5 0.5 \\
0.5 0.75 \\
0.5 1.0 \\
0.5 1.25 \\
0.5 1.5 \\
0.5 1.75 \\
0.5 2.0 \\
0.5 2.25 \\
0.5 2.5 \\
0.75 -2.5 \\
0.75 -2.25 \\
0.75 -2.0 \\
0.75 -1.75 \\
0.75 -1.5 \\
0.75 -1.25 \\
0.75 -1.0 \\
0.75 -0.75 \\
0.75 -0.5 \\
0.75 -0.25 \\
0.75 0.0 \\
0.75 0.25 \\
0.75 0.5 \\
0.75 0.75 \\
0.75 1.0 \\
0.75 1.25 \\
0.75 1.5 \\
0.75 1.75 \\
0.75 2.0 \\
0.75 2.25 \\
0.75 2.5 \\
1.0 -2.5 \\
1.0 -2.25 \\
1.0 -2.0 \\
1.0 -1.75 \\
1.0 -1.5 \\
1.0 -1.25 \\
1.0 -1.0 \\
1.0 -0.75 \\
1.0 -0.5 \\
1.0 -0.25 \\
1.0 0.0 \\
1.0 0.25 \\
1.0 0.5 \\
1.0 0.75 \\
1.0 1.0 \\
1.0 1.25 \\
1.0 1.5 \\
1.0 1.75 \\
1.0 2.0 \\
1.0 2.25 \\
1.0 2.5 \\
1.25 -2.5 \\
1.25 -2.25 \\
1.25 -2.0 \\
1.25 -1.75 \\
1.25 -1.5 \\
1.25 -1.25 \\
1.25 -1.0 \\
1.25 -0.75 \\
1.25 -0.5 \\
1.25 -0.25 \\
1.25 0.0 \\
1.25 0.25 \\
1.25 0.5 \\
1.25 0.75 \\
1.25 1.0 \\
1.25 1.25 \\
1.25 1.5 \\
1.25 1.75 \\
1.25 2.0 \\
1.25 2.25 \\
1.25 2.5 \\
1.5 -2.5 \\
1.5 -2.25 \\
1.5 -2.0 \\
1.5 -1.75 \\
1.5 -1.5 \\
1.5 -1.25 \\
1.5 -1.0 \\
1.5 -0.75 \\
1.5 -0.5 \\
1.5 -0.25 \\
1.5 0.0 \\
1.5 0.25 \\
1.5 0.5 \\
1.5 0.75 \\
1.5 1.0 \\
1.5 1.25 \\
1.5 1.5 \\
1.5 1.75 \\
1.5 2.0 \\
1.5 2.25 \\
1.5 2.5 \\
1.75 -2.5 \\
1.75 -2.25 \\
1.75 -2.0 \\
1.75 -1.75 \\
1.75 -1.5 \\
1.75 -1.25 \\
1.75 -1.0 \\
1.75 -0.75 \\
1.75 -0.5 \\
1.75 -0.25 \\
1.75 0.0 \\
1.75 0.25 \\
1.75 0.5 \\
1.75 0.75 \\
1.75 1.0 \\
1.75 1.25 \\
1.75 1.5 \\
1.75 1.75 \\
1.75 2.0 \\
1.75 2.25 \\
1.75 2.5 \\
2.0 -2.5 \\
2.0 -2.25 \\
2.0 -2.0 \\
2.0 -1.75 \\
2.0 -1.5 \\
2.0 -1.25 \\
2.0 -1.0 \\
2.0 -0.75 \\
2.0 -0.5 \\
2.0 -0.25 \\
2.0 0.0 \\
2.0 0.25 \\
2.0 0.5 \\
2.0 0.75 \\
2.0 1.0 \\
2.0 1.25 \\
2.0 1.5 \\
2.0 1.75 \\
2.0 2.0 \\
2.0 2.25 \\
2.0 2.5 \\
2.25 -2.5 \\
2.25 -2.25 \\
2.25 -2.0 \\
2.25 -1.75 \\
2.25 -1.5 \\
2.25 -1.25 \\
2.25 -1.0 \\
2.25 -0.75 \\
2.25 -0.5 \\
2.25 -0.25 \\
2.25 0.0 \\
2.25 0.25 \\
2.25 0.5 \\
2.25 0.75 \\
2.25 1.0 \\
2.25 1.25 \\
2.25 1.5 \\
2.25 1.75 \\
2.25 2.0 \\
2.25 2.25 \\
2.25 2.5 \\
2.5 -2.5 \\
2.5 -2.25 \\
2.5 -2.0 \\
2.5 -1.75 \\
2.5 -1.5 \\
2.5 -1.25 \\
2.5 -1.0 \\
2.5 -0.75 \\
2.5 -0.5 \\
2.5 -0.25 \\
2.5 0.0 \\
2.5 0.25 \\
2.5 0.5 \\
2.5 0.75 \\
2.5 1.0 \\
2.5 1.25 \\
2.5 1.5 \\
2.5 1.75 \\
2.5 2.0 \\
2.5 2.25 \\
2.5 2.5 \\
            };
        \end{axis}
     \end{tikzpicture}}
        \caption{Box index set}
        \label{fig:hyperrectangle}
        \end{subfigure}
        \hspace{0.08\textwidth}
    \begin{subfigure}[b!]{0.4\textwidth}
    \centering
    \resizebox{0.9\textwidth}{0.9\textwidth}{
        \begin{tikzpicture}
        
        \begin{axis}[{grid=both,ymin=-2.5,ymax=2.5,xmax=2.5,xmin=-2.5,xticklabel=\empty,yticklabel=\empty,
               minor tick num=1,axis lines = middle,xlabel=$h_1$,ylabel=$h_2$,label style =
               {at={(ticklabel cs:1.1)}}}]
           
                \addplot[
                only marks,
                mark=*,
                color=black,
                mark size=1.5pt
            ] table[row sep=crcr] {
                x y \\
                0.0 0.0 \\
-2.5 -0.25 \\
-2.5 0.0 \\
-2.5 0.25 \\
-2.25 -0.25 \\
-2.25 0.0 \\
-2.25 0.25 \\
-2.0 -0.25 \\
-2.0 0.0 \\
-2.0 0.25 \\
-1.75 -0.25 \\
-1.75 0.0 \\
-1.75 0.25 \\
-1.5 -0.25 \\
-1.5 0.0 \\
-1.5 0.25 \\
-1.25 -0.5 \\
-1.25 -0.25 \\
-1.25 0.0 \\
-1.25 0.25 \\
-1.25 0.5 \\
-1.0 -0.5 \\
-1.0 -0.25 \\
-1.0 0.0 \\
-1.0 0.25 \\
-1.0 0.5 \\
-0.75 -0.75 \\
-0.75 -0.5 \\
-0.75 -0.25 \\
-0.75 0.0 \\
-0.75 0.25 \\
-0.75 0.5 \\
-0.75 0.75 \\
-0.5 -1.25 \\
-0.5 -1.0 \\
-0.5 -0.75 \\
-0.5 -0.5 \\
-0.5 -0.25 \\
-0.5 0.0 \\
-0.5 0.25 \\
-0.5 0.5 \\
-0.5 0.75 \\
-0.5 1.0 \\
-0.5 1.25 \\
-0.25 -2.5 \\
-0.25 -2.25 \\
-0.25 -2.0 \\
-0.25 -1.75 \\
-0.25 -1.5 \\
-0.25 -1.25 \\
-0.25 -1.0 \\
-0.25 -0.75 \\
-0.25 -0.5 \\
-0.25 -0.25 \\
-0.25 0.0 \\
-0.25 0.25 \\
-0.25 0.5 \\
-0.25 0.75 \\
-0.25 1.0 \\
-0.25 1.25 \\
-0.25 1.5 \\
-0.25 1.75 \\
-0.25 2.0 \\
-0.25 2.25 \\
-0.25 2.5 \\
0.0 -2.5 \\
0.0 -2.25 \\
0.0 -2.0 \\
0.0 -1.75 \\
0.0 -1.5 \\
0.0 -1.25 \\
0.0 -1.0 \\
0.0 -0.75 \\
0.0 -0.5 \\
0.0 -0.25 \\
0.0 0.0 \\
0.0 0.25 \\
0.0 0.5 \\
0.0 0.75 \\
0.0 1.0 \\
0.0 1.25 \\
0.0 1.5 \\
0.0 1.75 \\
0.0 2.0 \\
0.0 2.25 \\
0.0 2.5 \\
0.25 -2.5 \\
0.25 -2.25 \\
0.25 -2.0 \\
0.25 -1.75 \\
0.25 -1.5 \\
0.25 -1.25 \\
0.25 -1.0 \\
0.25 -0.75 \\
0.25 -0.5 \\
0.25 -0.25 \\
0.25 0.0 \\
0.25 0.25 \\
0.25 0.5 \\
0.25 0.75 \\
0.25 1.0 \\
0.25 1.25 \\
0.25 1.5 \\
0.25 1.75 \\
0.25 2.0 \\
0.25 2.25 \\
0.25 2.5 \\
0.5 -1.25 \\
0.5 -1.0 \\
0.5 -0.75 \\
0.5 -0.5 \\
0.5 -0.25 \\
0.5 0.0 \\
0.5 0.25 \\
0.5 0.5 \\
0.5 0.75 \\
0.5 1.0 \\
0.5 1.25 \\
0.75 -0.75 \\
0.75 -0.5 \\
0.75 -0.25 \\
0.75 0.0 \\
0.75 0.25 \\
0.75 0.5 \\
0.75 0.75 \\
1.0 -0.5 \\
1.0 -0.25 \\
1.0 0.0 \\
1.0 0.25 \\
1.0 0.5 \\
1.25 -0.5 \\
1.25 -0.25 \\
1.25 0.0 \\
1.25 0.25 \\
1.25 0.5 \\
1.5 -0.25 \\
1.5 0.0 \\
1.5 0.25 \\
1.75 -0.25 \\
1.75 0.0 \\
1.75 0.25 \\
2.0 -0.25 \\
2.0 0.0 \\
2.0 0.25 \\
2.25 -0.25 \\
2.25 0.0 \\
2.25 0.25 \\
2.5 -0.25 \\
2.5 0.0 \\
2.5 0.25 \\
            };
        \end{axis}
     \end{tikzpicture}}
        \caption{Hyperbolic cross index set}
        \label{fig:hyperbolic_cross}
    \end{subfigure}
    \caption{Modes $\bsh \in \mathcal{A}$ of the both the box index set (a) and the hyperbolic cross index set (b) for $M=10$.}
    \label{fig:index_sets}
\end{figure}

\subsection{Applying the FFT on a Regular Grid} \label{sec:discr}

We utilize a spatial discretization
$\{\bsx_1,\ldots,\bsx_m\}\subset D$ to evaluate the Fourier coefficients
\begin{align} \label{eq:hatv}
  \widehat{v}_q(\bsh) \approx 
  \widehat{v}_q^{\rm app}(\bsh) 
  := \frac{1}{m} \sum_{k=1}^m v_q(\bsx_k)\,e^{-2\pi\ri\bsh\cdot\bsx_k},
 \qquad q=1,\ldots,d_\ell, \quad \bsh\in\bbZ^d.
\end{align}
Using a tensor product grid with $m = m_1\times\cdots\times m_d$ equally
spaced points, we have
\begin{align} \label{eq:gridv}
  \widehat{v}_q^{\rm app}(\bsh)  
  = \frac{1}{m} \sum_{k_1=0}^{m_1-1} \cdots \sum_{k_d=0}^{m_d-1}
  v_q\Big(\frac{k_1}{m_1},\ldots,\frac{k_d}{m_d}\Big)\,
  e^{-2\pi\ri\bsh\cdot (\frac{k_1}{m_1},\ldots,\frac{k_d}{m_d})}.
\end{align}
The values of $\widehat{v}_q^{\rm app}(\bsh)$ for all $m$ vectors $\{\bsh\in\bbZ^d : 0\le h_j\le m_j-1 \mbox{ for } j=1,\ldots,d\}$ can be computed at once by the $d$-dimensional FFT at the cost of $\calO(m\log(m))$ operations. For all other vectors $\bsh$, the values of $\widehat{v}_q^{\rm app}(\bsh)$ are periodic and are obtained by $\widehat{v}_q^{\rm app}(h_1\bmod m_1,\ldots,h_d\bmod m_d)$. 
Intuitively, $\calA$ should be centered around the origin, and we should have $|\calA| \le m$ to ensure that we do not get repeated values of $\widehat{v}_q^{\rm app}(\bsh)$ (this bad effect is known as \emph{aliasing}).

Then, we can approximate \eqref{eq:Kv2} at an arbitrary point $\bsx\in D$ by
\begin{align*} 
  (K_\ell\, v)_{p}(\bsx) \approx
  \sum_{\bsh\in\calA} \sum_{q=1}^{d_\ell} (\widehat{\kappa_\ell})_{p,q}(\bsh)\;\widehat{v}_q^{\rm app}(\bsh)\, 
  e^{2\pi\ri\bsh\cdot\bsx},
 \quad p=1,\ldots,d_{\ell+1}.
\end{align*}
We can also evaluate this at a grid, which could be the same grid or a different grid from before. If we do this now at a finer grid (not necessarily nested), say, $\{(\frac{k_1}{m_1'},\ldots,\frac{k_d}{m_d'}) : 0\le k_j\le m_j'-1 \mbox{ for } j=1,\ldots,d\}$ with each $m_j'> m_j$, then to be able to use IFFT we need to work with $\{\bsh\in\bbZ^d : 0\le h_j\le m_j'-1 \mbox{ for } j=1,\ldots,d\}$. 

\begin{figure}[t]
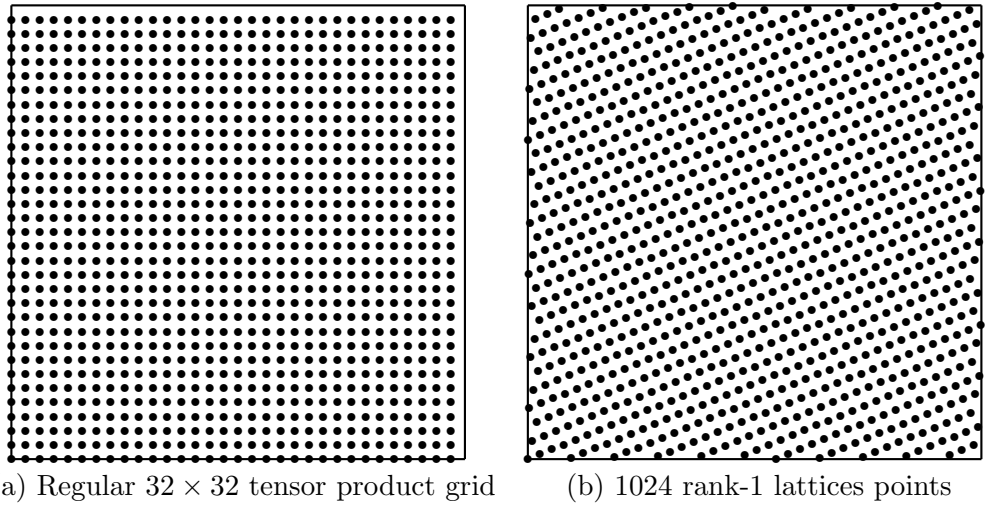

\centering
\setlength{\unitlength}{6cm} 
\begin{tabular}{cc}
\input{figures/Grid} & \input{figures/Rank-1Lattice_v2} \\
(a) Regular $32\times 32$ tensor product grid & (b) 1024 rank-1 lattices points
\end{tabular}
\caption{Spatial points $\{\bsx_1,\ldots,\bsx_m\}$ in the unit square $D=[0,1]^2$ as used by (a) the tensor product Fourier transform,
and (b) the Fourier transform on a rank-1 lattice with generating vector  $\bsz = (1, 721)$, for $m = 32 \times 32 = 1024$.}
\label{Fig:Points}
\end{figure}

\subsection{Applying the FFT on a Rank-1 Lattice} \label{sec:lat-x}

Instead of a regular tensor product  grid (see Figure~\ref{Fig:Points}(a)), we select the points $\{\bsx_1,\ldots,\bsx_m\}$ to be a
$d$-dimensional rank-$1$ lattice (see Figure~\ref{Fig:Points}(b))
\begin{align} \label{eq:lat-x}
  \bsx_k := \frac{k\,\bsz \bmod m}{m}, \qquad k=1,\ldots, m,
\end{align}
where $\bsz\in\bbZ^d$ is known as the \emph{generating vector} \citep{SJ94} that determines the quality of the quadrature rule to the original integral in the definition of the Fourier coefficients in equation~\eqref{eq:FourierCoeff}.
For the rank-1 lattice points, we have
\[
  \widehat{v}_q^{\rm app}(\bsh)  
  = \frac{1}{m} \sum_{k=1}^m
  v_q\Big(\frac{k\,\bsz \bmod m}{m}\Big)\,
  e^{-2\pi\ri k\bsh\cdot \bsz/m},
 \qquad q=1,\ldots,d_\ell, \quad \bsh\in\bbZ^d
\]
instead of \eqref{eq:gridv} for the tensor product grid.
The values of $\widehat{v}_q^{\rm app}(\bsh)$ for all $m$ possible dot products $\bsh\cdot\bsz\bmod m$ can be computed at once using a one-dimensional FFT at the combined cost of $\calO(m\log(m))$
operations. The inverse FFT (see \eqref{eq:Kv2}) is evaluated by a $1$-dimensional IFFT at the same lattice points $\{\bsx_1,\ldots,\bsx_m\}$ \citep{KKP10,KPV15,KPV21,KMNN21}.
As compared to the tensor product grid case, only one $1$-dimensional transformation is computed when using rank-$1$ lattice points. In addition, $m$ may be selected more freely. However, as already said, $m$ has to be chosen large enough ($|\calA| < m$) and a lattice generating vector with the \emph{reconstruction property} \citep{KKP10,KPV15,KPV21,KMNN21} can be used to avoid aliasing.

\section{Generalization and Training Errors} \label{sec:gen}

Suppose we have observations $(a_i,G(a_i))_{i=1}^{n}$ where $a_i
\sim\mu$ is an i.i.d.\ sequence from the probability measure $\mu$
supported on the data space $A$. 
We are interested in controlling the \emph{generalization error} $\calE_G$ defined by
\begin{align} \label{eq:gen1}
  \calE_G^2 := \int_A \|G(a) - G_\theta(a)\|_{L_2(D)}^2 \,\rd\mu(a).
\end{align}

As foreshadowed in the introduction, we assume precise knowledge on how the training data is generated. As we will see in Section~\ref{sec:random}, we assume that the random field $a(\bsx,\bsy)$ depends on parameters $\bsy$ in a parametric space $Y$. Given this framework in mind, problem \eqref{eq:dto1} can be reformulated as
\begin{align} \label{eq:dto2}
 \mbox{Given data $\bsy\in Y$, compute observable $G(\bsy)\in Z$},
\end{align}
with our knowledge on $a(\bsx,\bsy)$ built into the solution map for $G$. Thus the neural operator \eqref{eq:no1} can be reformulated as
\begin{align*}
 G_\theta(\bsy)(\bsx) 
 = (Q \circ T_L \circ T_{L-1} \circ \cdots \circ  T_2 \circ T_1 \circ P \circ a(\bsy))(\bsx),
\end{align*}
or more compactly as
\begin{align} \label{eq:no2}
 G_\theta(\bsx,\bsy) 
 = (Q \circ T_L \circ T_{L-1} \circ \cdots \circ  T_2 \circ T_1 \circ P \circ a)(\bsx,\bsy).
\end{align}
In other words, the random field modeling $a(\bsx,\bsy)$ can be interpreted as a hard-coded initial layer of the neural operator. It takes the role of an embedding.

Hence, the generalization error \eqref{eq:gen1} satisfies
\begin{align} \label{eq:gen2}
  \calE_G^2 = \int_Y \|G(\bsy)(\cdot) - G_\theta(\bsy)(\cdot)\|_{L_2(D)}^2\,\rd\bsy
  = \int_Y \int_D \|G(\bsx,\bsy) - G_\theta(\bsx,\bsy) \|_2^2 \,\rd\bsx\,\rd\bsy,
\end{align}
where we use the vector $2$-norm for vector-valued observables.

To generate a training set, we start with a set of points $\{\bsy_1,\ldots,\bsy_n\} \subset Y$ to obtain
the random fields $\{ a(\bsy_i) = a(\cdot,\bsy_i):\, i=1, \ldots,n\}$, and for each point $\bsy_i\in Y$ we compute the observable or target function at the spatial points $\{\bsx_1,\ldots,\bsx_m\}\subset D$, to obtain
\[
  \big\{ G(\bsy_i)(\bsx_k) = G(\bsx_k,\bsy_i):\; k = 1,\ldots,m,\; i=1, \ldots,n \big\}.
\]
For example, $G(\bsx_k,\bsy_i)$ could be the finite element solution $u(\bsx_k,\bsy_i)$ of a PDE problem at the spatial point $\bsx_k$, with the input random field generated by $\bsy_i$.

The \emph{training error} $\calE_T$ satisfies
\begin{align} \label{eq:train2}
  \calE_T^2 
  := \frac{1}{n} \sum_{i=1}^n \frac{1}{m} \sum_{k=1}^m 
 \|G(\bsx_k,\bsy_i) - G_\theta(\bsx_k,\bsy_i) \|_2^2,
\end{align}
and the \emph{generalization gap} satisfies $|\calE_G - \calE_T|
  \le \sqrt{ |\calE_G^2 - \calE_T^2| }$, where
\begin{align*}
  \calE_G^2 - \calE_T^2 
&= \int_Y \int_D \|G(\bsx,\bsy) - G_\theta(\bsx,\bsy) \|_2^2 \,\rd\bsx\,\rd\bsy
- \frac{1}{n} \sum_{i=1}^n \frac{1}{m} \sum_{k=1}^m 
  \|G(\bsx_k,\bsy_i) - G_\theta(\bsx_k,\bsy_i) \|_2^2.
\end{align*}
Abbreviating the integrals by $\calI_Y$ and $\calI_D$ and the quadratures by $\calQ_Y$ and $\calQ_D$, we can write more succinctly, using linearity,
\begin{align} \label{eq:gap}
  \calE_G^2 - \calE_T^2 
&= \calI_Y\calI_D\, \|G - G_\theta\|_2^2 - \calQ_Y\calQ_D\, \|G - G_\theta\|_2^2 \nonumber\\
&= (\calI_Y - \calQ_Y)\,\calI_D\, \|G - G_\theta\|_2^2  
  + \calQ_Y (\calI_D-\calQ_D)\, \|G - G_\theta\|_2^2 \nonumber\\
&= \sum_{p=1}^{d_u} \Big(
  (\calI_Y - \calQ_Y)\,\calI_D (G_p - G_{\theta,p})^2  
  + \calQ_Y (\calI_D-\calQ_D) (G_p - G_{\theta,p})^2 \Big),
\end{align}
where each component $p$ separates into a quadrature error in~$Y$ and a quadrature error in~$D$. In an appropriate function space setting, each error can be bounded by the \emph{worst case integration error} times the norm of the integrand:
\begin{align}
 |(\calI_Y - \calQ_Y)\,\calI_D (G_p - G_{\theta,p})^2| 
 &\le e^{\rm wor}_n(\{\bsy_i\}_{i=1}^n; \calH_Y)\,
 \big\|\textstyle\int_D (G_p(\bsx,\cdot) - G_{\theta,p}(\bsx,\cdot))^2\,\rd\bsx \big\|_{\calH_Y}, 
 \label{eq:norm-y} \\
 |\calQ_Y(\calI_D-\calQ_D) (G_p - G_{\theta,p})^2|
 &\le \frac{1}{n} \!\sum_{i=1}^n\!  e^{\rm wor}_m(\{\bsx_k\}_{k=1}^m;\! \calH_D) \,
 \big\|(G_p(\cdot,\bsy_i) \!-\! G_{\theta,p}(\cdot,\bsy_i))^2\big\|_{\calH_D}, \label{eq:norm-x}
\end{align}
where $\calH_Y$ is the function space on the parametric domain for $\bsy\in Y$ and $\calH_D$ is the function space on the spatial domain for $\bsx\in D$. We can quantify these errors if we know the regularity of $(G - G_\theta)^2$ with respect to $\bsy$ and $\bsx$, respectively. This is analyzed in Sections~\ref{sec:reg-y} and~\ref{sec:reg-x}.

As already explained in Section~\ref{sec:lat-x}, we choose the spatial points $\bsx_k$ to be a rank-$1$ lattice on the spatial domain~$D$ with a generating vector $\bsz\in\bbZ^d$, see~\eqref{eq:lat-x}. Similarly, we choose the training points $\bsy_i$ to be a rank-$1$ lattice or randomly-shifted lattice on the parametric domain~$Y$ with a generating vector $\bsw\in\bbZ^s$ (different from $\bsz\in\bbZ^d$), depending on the random field model and the function space setting. Thus we have
\begin{align} \label{eq:lat-y}
  \bsy_i := \Big(\frac{i\,\bsw}{n} + \bsDelta\Big) \bmod 1, \quad i = 1, \ldots, n,
\end{align}
where $\bsDelta\in [0,1]^s$ is the random shift ($\bsDelta = \bszero$ for the deterministic case without random shift) and ``mod $1$'' takes the fractional part of every component in the vector to ensure that the resulting points lies in the unit cube $Y = [0,1]^s$.
Figure~\ref{Fig:Points}(b) can also be interpreted as lattice points $\{\bsy_1,\ldots,\bsy_n\}$ on the parametric domain $Y = [0,1]^2$ with $n = 1024$.

\subsection{Modeling the Random Field \texorpdfstring{$a(\bsx,\bsy)$}{$a(x,y)$}} \label{sec:random}

In our error analysis we assume that the random field $a:Y\to A$ is generated by a series expansion depending on parameters $\bsy = (y_j)_{j=1}^s \in Y$ with a high parametric dimension $s$. Popular random field models include:
\begin{enumerate}
\item [(a)] \textbf{Affine model} (non-periodic), see e.g., \citet{CDS10}:
\begin{align} \label{eq:a-np}
  a(\bsy)(\bsx) = a(\bsx,\bsy) = \psi_0(\bsx) + \sum_{j=1}^s (y_j - \tfrac{1}{2})\,\psi_j(\bsx),
  \qquad \bsy\in Y := [0,1]^s.
\end{align}
\item [(b)] \textbf{Chebyshev model} (periodic), see \citet{KKS20}, shown in \citet{KKS24} 
to be equivalent to the affine model with a Chebyshev density \citep{ABW22} instead of a uniform density:
\begin{align} \label{eq:a-per}
  a(\bsy)(\bsx) = a(\bsx,\bsy) = \psi_0(\bsx) + \sum_{j=1}^s \sin(2\pi y_j)\,\psi_j(\bsx),
  \qquad \bsy\in Y := [0,1]^s.
\end{align}
\end{enumerate}
In both cases we assume that the parameters in $\bsy$ are independent and uniformly distributed. A precise choice of the basis functions $\psi_j$ is discussed in Section~\ref{sec:reg-x} below. To allow that the random field takes complex values over the spatial domain $D$, we will take an even value for~$s$ and set $\psi_{2j-1} \equiv \phi_j$ and $\psi_{2j} \equiv \ri\,\phi_j$ for complex basis functions $\{\phi_j\}$. This avoids the need to define complex-valued parameters $\bsy$.

\subsection{Function Space Settings} \label{sec:func}

We now briefly review three theoretical settings for lattice rules (noting that there are also other settings). It is known that good lattice generating vectors can be obtained from a \emph{fast component-by-component construction} \citep{NC06,CKN06} to achieve nearly the optimal convergence rates for the worst case errors in appropriately paired function space settings. 

Setting~(a) below yields first order convergence for non-periodic functions over the parametric space $Y = [0,1]^s$, see e.g.,~\citet[Theorem~5.8]{DKS13}. Setting~(b) yields higher order convergence for periodic functions over the parametric space $Y = [0,1]^s$, see e.g.,~\citet[Theorem~5.12]{DKS13}. Setting~(c) is the same as Setting~(b), but is now formulated for the physical space $D = [0,1]^d$ instead of the parametric space $Y = [0,1]^s$, with smoothness parameter $\omega$ instead of $\alpha$, and with weights $\rho_\setu$ instead of $\gamma_\setu$.

In the following, for $\bsy\in Y$ and a subset $\setu\subseteq\{1:s\} := \{1,2,\ldots,s\}$ we write $\bsy_\setu := (y_j)_{j\in\setu}$ and $\bsy_{-\setu} := (y_j)_{j\in\{1:s\}\setminus\setu}$, and denote the corresponding domains by $Y_\setu$ and $Y_{-\setu}$. We also have the mixed derivatives
$\partial^{1_\setu}_\bsy := \prod_{j\in\setu} \frac{\partial}{\partial y_j}$
and 
$\partial^{\alpha_\setu}_\bsy := \prod_{j\in\setu} (\frac{\partial}{\partial y_j})^\alpha$, and
$\zeta(z) := \sum_{h=1}^\infty h^{-z}$ is the Riemann zeta
function. Analogous notations are used for $\bsx\in D$.
Using this notation, the three settings are:
\begin{itemize}
\item [\textnormal{(a)}] For a ``randomly shifted'' lattice rule in a
    (non-periodic, ``unanchored'') weighted Sobolev space
    $\calH_Y^1$ with smoothness parameter $\alpha = 1$ and weights $\bsgamma:=(\gamma_\setu)_{\setu\subseteq\{1:s\}}$, we define
\begin{align} \label{eq:norm-np}
 \|v\|_{\calH^1_Y}^2
 := \sum_{\setu\subseteq\{1:s\}} \frac{1}{\gamma_\setu}
 \int_{Y_\setu} \bigg| \int_{Y_{-\setu}}
 \partial^{1_\setu}_\bsy v(\bsy)\,
 \rd\bsy_{-\setu}\bigg|^2\rd\bsy_\setu.
\end{align}
A good generating vector $\bsw\in\bbZ^s$ can be constructed such that the
root-mean-square worst case error satisfies
\begin{align} \label{eq:wce-np}
 {\rm r.m.s.}\; e^{\rm wor}_n(\bsw;\calH^1_Y)
  \le \bigg(
  \frac{2}{n} \sum_{\emptyset\ne \setu \subseteq \{1:s\}} \gamma_\setu^\lambda\,
  \bigg[
  \frac{2\zeta(2\lambda)\,2^\lambda}{(2\pi)^{2\lambda}} 
  \bigg]^{|\setu|}\bigg)^{\frac{1}{2\lambda}}\,
  \quad\forall\;\lambda\in (\tfrac{1}{2},1].
\end{align}

\item [\textnormal{(b)}] For a lattice rule in a (periodic)
    weighted Korobov space $\calH^\alpha_Y$ with integer smoothness parameter $\alpha\ge 1$ and weights $\bsgamma=(\gamma_\setu)_{\setu\subseteq\{1:s\}}$, we define
\begin{align} \label{eq:norm-per}
 \|v\|_{\calH^\alpha_Y}^2
 := \sum_{\setu\subseteq\{1:s\}} 
 \frac{1}{\gamma_\setu}
 \int_{Y_\setu} \bigg| \int_{Y_{-\setu}}
 \partial^{\alpha_\setu}_\bsy v(\bsy)\,
 \rd\bsy_{-\setu}\bigg|^2\rd\bsy_\setu.
\end{align}
A good generating vector $\bsw\in\bbZ^s$ can be constructed such that the worst
case error satisfies
\begin{align} \label{eq:wce-per}
  e^{\rm wor}_n(\bsw;\calH^\alpha_Y)
  \le \bigg(
  \frac{2}{n} \sum_{\emptyset\ne \setu\subseteq \{1:s\}}
  \gamma_\setu^\lambda\,
  \bigg[\frac{2\zeta(2\alpha\lambda)}{(2\pi)^{2\alpha\lambda}}
  \bigg]^{|\setu|}
  \bigg)^{\frac{1}{2\lambda}}\, \quad\forall\;\lambda\in
  (\tfrac{1}{2\alpha},1].
\end{align}

\item [\textnormal{(c)}] For a lattice rule in a (periodic)
    weighted Korobov space $\calH^\omega_D$ with integer smoothness parameter $\omega\ge 1$ and weights $\bsrho=(\rho_\setu)_{\setu\subseteq\{1:d\}}$, we define
\begin{align} \label{eq:norm-per-D}
 \|v\|_{\calH^\omega_D}^2
 := \sum_{\setu\subseteq\{1:d\}} 
 \frac{1}{\rho_\setu}
 \int_{D_\setu} \bigg| \int_{D_{-\setu}}
 \partial^{\omega_\setu}_\bsy v(\bsx)\,
 \rd\bsx_{-\setu}\bigg|^2\rd\bsx_\setu.
\end{align}
This norm can also be expressed in terms of Fourier coefficients as
\begin{align} \label{eq:norm-per-rD}
  \|v\|_{\calH^\omega_D}^2 := \sum_{\bsh\in\bbZ^d} |\widehat{v}(\bsh)|^2\,r(\bsh),
 \qquad
  r(\bsh) = r^{\rm hc}(\bsh) :=
  \frac{\prod_{j\in\supp(\bsh)} |h_j|^{2\omega}}{\rho_{\supp(\bsh)}},
\end{align}
with $\supp(\bsh) := \{1\le j\le d: h_j\ne 0\}$.
A good generating vector $\bsz\in\bbZ^d$ can be constructed such that the worst
case error satisfies
\begin{align} \label{eq:wce-per-D}
  e^{\rm wor}_m(\bsz;\calH^\omega_D)
  \le \bigg(
  \frac{2}{m} \sum_{\emptyset\ne \setu\subseteq \{1:d\}}
  \rho_\setu^\lambda\,
  \bigg[\frac{2\zeta(2\omega\lambda)}{(2\pi)^{2\omega\lambda}}
  \bigg]^{|\setu|}
  \bigg)^{\frac{1}{2\lambda}}\, \quad\forall\;\lambda\in
  (\tfrac{1}{2\omega},1].
\end{align}
\end{itemize}

Later, we consider PDE solutions in the standard Sobolev space on $D$ with norm 
\begin{align} \label{eq:Sob-norm}
  \|v\|_{H^\omega(D)}^2 := \sum_{|\bsnu|\le\omega} \| \partial^{\bsnu}_\bsx v \|_{L_2(D)}^2.
\end{align}
This is clearly different from the $\calH^\omega_D$ norm \eqref{eq:norm-per-D} which is more related to the ``mixed'' Sobolev norm
\begin{align} \label{eq:Sob-mix}
  \|v\|_{H^\omega_{\rm mix}(D)}^2 
  := \sum_{\bsnu\le (\omega,\ldots,\omega)} \| \partial^{\bsnu}_\bsx v \|_{L_2(D)}^2.
\end{align}
In the standard Sobolev norm \eqref{eq:Sob-norm} the combined number of derivatives in all coordinates is at most~$\omega$, while in the $\calH^\omega_D$ norm \eqref{eq:norm-per-D} and the mixed norm \eqref{eq:Sob-mix} every coordinate can be differentiated up to $\omega$ times.

We remark that in order to apply the upper bound \eqref{eq:norm-x} and make use of setting~(c) above, we need the target function to satisfy $G_p(\cdot,\bsy)\in \calH_D^\omega$ for every component $p$. In this periodic context, it means that the spatial domain $D = [0,1]^d$ should be viewed as the $d$-dimensional torus $\bbT^d$. For the specific example where the target function is related to the PDE \eqref{eq:elliptic}, this means that there is no boundary condition, and instead we require that $\int_{\bbT^d} f(\bsx)\,\rd\bsx = 0$. In this case it can be shown that there exists a unique solution $u_0$ satisfying $\int_{\bbT^d} u_0(\bsx)\,\rd\bsx = 0$, and a general solution is of the form $u = u_0 + c$ for $c\in\bbR$. 

For the most part of this paper, $G(\bsx,\bsy)$ is a generic target function which depends on an input field $a(\bsx,\bsy)$, but is not necessarily related to the PDE \eqref{eq:elliptic}.

\subsection{Regularity with respect to \texorpdfstring{$\bsy$}{$y$}} \label{sec:reg-y}

To control the norms in \eqref{eq:norm-y} and \eqref{eq:norm-x}  we need to know the regularity of $(G-G_\theta)^2$ with respect to $\bsy$ and $\bsx$, respectively. Therefore, we need to know the regularity of $G_\theta$ with respect to both $\bsy$ and $\bsx$. In this subsection we follow the strategy in \citet{KKNS26a} which obtained regularity bounds for deep neural networks with respect to the parametric variables $\bsy$. The regularity analysis with respect to the spatial variables $\bsx$ is new in this manuscript and is analyzed in the next subsection.

For convenience, we rewrite the FNO \eqref{eq:no2} in recursive form
\begin{align} \label{eq:recur}
 &G_\theta(\bsx,\bsy) := Q\, \sigma(g^{[L]}(\bsx,\bsy)), 
 \qquad\bsx\in D = [0,1]^d, \quad \bsy\in Y = [0,1]^s,\nonumber\\
 &\begin{cases}
 g^{[1]}(\bsx,\bsy)
 := W_1 P\, a(\bsx,\bsy) + (K_1 P\,a(\cdot,\bsy))(\bsx) + b_1, \\
 g^{[\ell]}(\bsx,\bsy)
 := W_\ell\, \sigma (g^{[\ell-1]}(\bsx,\bsy))
 + (K_\ell\, \sigma(g^{[\ell-1]}(\cdot,\bsy)))(\bsx) + b_\ell
 \quad\mbox{for } \ell\ge 2,
 \end{cases}
\end{align}
where for a finite frequency index set $\calA\subset\bbZ^d$, we have
\begin{align*}
  (K_\ell\, v)(\bsx) 
  = \sum_{\bsh\in\calA} \widehat{\kappa_\ell}(\bsh) \, \widehat{v}(\bsh) \, e^{2\pi\ri\bsh\cdot\bsx}.
\end{align*}
Note that we have truncated the frequency index set as discussed in Section~\ref{sec:trunc}, and we also implicitly approximate $\widehat{v}(\bsh)$ as discussed in Sections~\ref{sec:discr} and~\ref{sec:lat-x}, without using the explicit notation $\widehat{v}^{\rm app}(\bsh)$ to highlight them.

In order to present the regularity bounds we need a few notations. 
Let $\bbN_0 := \{0,1,2,\ldots\}$ be the set of nonnegative integers. For a
multiindex $\bsnu = (\nu_j)_{j\ge 1} \in \bbN_0^\bbN$ we define its
order to be $|\bsnu| := \sum_{j\ge 1} \nu_j$, and we consider the index
set $\indx := \{\bsnu \in \bbN_0^{\bbN} : |\bsnu| < \infty\}$. We write
$\bsm\le\bsnu$ when $m_j\le \nu_j$ for all $j\ge 1$, and
$\binom{\bsnu}{\bsm} := \prod_{j\ge 1} \binom{\nu_j}{m_j}$ using the
notation of binomial coefficients. For any sequence $\bsbeta = (\beta_j)_{j\ge
1}$ we write $\bsbeta^\bsnu := \prod_{j\ge 1} \beta_j^{\nu_j}$. We define
the mixed partial derivative operator with respect to the spatial
variables $\bsx\in D = [0,1]^d$ by $\partial^\bsnu_\bsx := \prod_{j=1}^d
(\frac{\partial}{\partial x_j})^{\nu_j}$, and the parametric
variables $\bsy\in Y = [0,1]^s$ by $\partial^\bsnu_\bsy := \prod_{j=1}^s
(\frac{\partial}{\partial y_j})^{\nu_j}$.

We need the \emph{Stirling number of the second kind}. They are defined
for integers $n$ and $k$ by $\calS(n,0):=\delta_{n,0}$ (i.e., it is 1 if $n=0$
and is $0$ otherwise), $\calS(n,k) := 0$ for $k> n$, and otherwise
\[
  \calS(n,k) := \frac{1}{k!}\sum_{i=0}^k (-1)^{k-i} \binom{k}{i} i^n, \qquad n\ge k.
\]
We introduce the shorthand notation
$\calS(\bsnu,\bsm) := \prod_{j\ge 1} \calS(\nu_j,m_j)$.

\subsubsection{Regularity of the FNO with respect to \texorpdfstring{$\bsy$}{$y$}}

For the affine random field \eqref{eq:a-np} and any $\bsnu\ne\bszero$ we have for all $\bsx\in D$ and $\bsy\in Y$ that
\begin{align} \label{eq:dy-a-np}
  \partial^\bsnu_\bsy a(\bsx,\bsy)
  = \begin{cases}
  \psi_j(\bsx) & \mbox{if } \bsnu=\bse_j, \\
  0 & \mbox{otherwise},
  \end{cases}
 \quad\mbox{and}\quad
  |\partial^\bsnu_\bsy a(\bsx,\bsy)|
  \le \begin{cases}
  \|\psi_j\|_{L_\infty(D)} & \mbox{if } \bsnu=\bse_j, \\
  0 & \mbox{otherwise}.
  \end{cases}
\end{align}
In this case the regularity bound depends on the magnitude of the basis functions $\psi_j$.

For the Chebyshev random field \eqref{eq:a-per} and any $\bsnu\ne\bszero$ we have 
\begin{align*}
  \partial^\bsnu_\bsy a(\bsx,\bsy)
  = \begin{cases}
  (2\pi)^k\,\sin(2\pi y_j + \frac{k\pi}{2})\,\psi_j(\bsx) & \mbox{if } \bsnu=k\bse_j,\, k\ge 1, \\
  0 & \mbox{otherwise},
  \end{cases}
\end{align*}
where we may differentiate with respect to a single variable $y_j$ multiple times but all cross derivatives vanish. Thus for all $\bsx\in D$ and $\bsy\in Y$ we have
\begin{align} \label{eq:dy-a-per}
  |\partial^\bsnu_\bsy a(\bsx,\bsy)|
  \le \begin{cases}
  (2\pi)^k\,\|\psi_j\|_{L_\infty(D)} & \mbox{if } \bsnu=k\bse_j,\, k\ge 1, \\
  0 & \mbox{otherwise}.
  \end{cases}
\end{align}

Building on these specific derivative bounds on $a(\bsx,\bsy)$, we can obtain regularity bounds with respect to $\bsy$ for the FNO for both random field models.

\begin{theorem}[Regularity bound for FNO with respect to $\bsy$] 
\label{thm:reg-y-An}
For $L\ge 1$, consider the FNO \eqref{eq:no2}, or equivalently \eqref{eq:recur}, where the random field $a(\bsx,\bsy)$ is either affine~\eqref{eq:a-np} or Chebyshev~\eqref{eq:a-per}, with basis functions $\psi_j$. Suppose there exist positive sequences $(\beta_j)_{j\ge 1}$, $(R_\ell)_{\ell\ge 1}$, $(A_n)_{n\ge 1}$ such that
\begin{align} \label{eq:ass-y}
 \left\{
 \begin{array}{rll}
 \|\psi_j\|_\infty &\le \beta_j, &j = 1,\ldots,s, \\[2mm]
 \|W_\ell\|_\infty
  + \displaystyle\sum_{\bsh\in\calA} \|\widehat{\kappa_\ell}(\bsh)\|_\infty &\le R_\ell, &
  \ell=1,\ldots,L, \\
  \|\sigma^{(n)}\|_\infty &\le A_n, &n\ge 1,
  \end{array}
  \right.
\end{align}
where $\|\cdot\|_\infty$ denotes the matrix $\infty$-norm or the sup-norm of a univariate function as appropriate. Suppose additionally 
\begin{align} \label{eq:common}
  A_n = \xi\,\tau^n\,n!
  \qquad\mbox{for some $\xi>0$ and $\tau>0$}.
\end{align}
Define 
\begin{align} \label{eq:CSP}
   C_L := \max\bigg\{
 \|G_\theta\|_\infty,\;
 \|Q\|_\infty\,\frac{\Pi_L}{S_L} \bigg\}, \quad 
  S_L := \frac{1}{\xi} \sum_{k=1}^L \Pi_k, \quad
  \Pi_k := \prod_{t=1}^k (\xi\,\tau\, R_t),
\end{align}
where $\|G_\theta\|_\infty$ denotes taking first the vector $\infty$-norm and then the sup-norms over $D$ and~$Y$.

For any multiindex $\bsnu\in\indx$ including $\bsnu=\bszero$, we have the regularity bounds with respect to $\bsy$:
\begin{enumerate}
\item [\textnormal{(a)}] 
The FNO \eqref{eq:no2} with the affine field \eqref{eq:a-np} satisfies
\begin{align} \label{eq:dy-np-An}
  \|\partial^\bsnu_\bsy G_\theta\|_\infty
  \le C_L\,|\bsnu|!\,(\|P\|_\infty\,S_L\,\bsbeta)^{\bsnu}.
  \qquad\qquad\qquad\qquad\quad\;\,
\end{align}

\item [\textnormal{(b)}] 
The FNO \eqref{eq:no2} with the Chebyshev field \eqref{eq:a-per} satisfies
\begin{align} \label{eq:dy-per-An}
  \|\partial^\bsnu_\bsy G_\theta\|_\infty
  \le C_L\,(2\pi)^{|\bsnu|}
  \sum_{\bsm\le\bsnu} |\bsm|!\,(\|P\|_\infty\,S_L\,\bsbeta)^{\bsm}\,\calS(\bsnu,\bsm).
\end{align}
\end{enumerate}
\end{theorem}

\begin{proof}
This proof is given in Appendix~\ref{sec:regularity_proofs_y}, which is adapted from the proofs of \citet[Theorems~2.2 and~2.3]{KKNS26a}.
\end{proof}

\citet{KKNS26b} showed that \eqref{eq:common} holds for common activation functions as follows:
\begin{align*} 
 {\rm sigmoid}_c(x) &:= \displaystyle\frac{1}{1+e^{-cx}},
 & A_n &= c^n\,n! && \mbox{(i.e., $\xi = 1$, $\tau = c$)};
 \vspace{0.2cm} \nonumber\\
 {\rm tanh}_c(x) &:= \displaystyle\frac{e^{cx}-e^{-cx}}{e^{cx}+e^{-cx}},
 & A_n &= (2c)^n\,n! && \mbox{(i.e., $\xi=1$, $\tau=2c$)};
 \vspace{0.2cm} \\
 {\rm swish}_c(x) &:= \displaystyle\frac{x}{1+e^{-cx}},
 & A_n &= \tfrac{1.1}{c}\,c^n\,n! && \mbox{(i.e., $\xi = \frac{1.1}{c}$, $\tau=c$)}; 
\end{align*}
while ${\rm ReLU}(x) := \max\{x,0\}$ is not smooth and so \eqref{eq:common} does not hold. However, we have
${\rm swish}_c(x) \to {\rm ReLU}(x)$ as $c \rightarrow \infty$.

\subsubsection{Norm of \texorpdfstring{$(G - G_\theta)^2$}{$(G - G_\theta)^2$} with respect to \texorpdfstring{$\bsy$}{$y$}}

Recall that we need a bound on the norm in \eqref{eq:norm-y}. From the definition \eqref{eq:norm-per}, we have
\begin{align} \label{eq:norm-y-2} 
 &\big\|\textstyle\int_D (G_p(\bsx,\cdot)- G_{\theta,p}(\bsx,\cdot))^2\,\rd\bsx \big\|_{\calH_Y^\alpha}^2 
 \nonumber\\
 &= \sum_{\setu\subseteq\{1:s\}} 
 \frac{1}{\gamma_\setu}
 \int_{Y_\setu} \bigg| \int_{Y_{-\setu}}
 \int_D \partial^{\alpha_\setu}_\bsy (G_p(\bsx,\bsy)- G_{\theta,p}(\bsx,\bsy))^2\,\rd\bsx
 \,\rd\bsy_{-\setu}\bigg|^2\rd\bsy_\setu \nonumber\\
 &\le \sum_{\setu\subseteq\{1:s\}} 
 \frac{1}{\gamma_\setu}
 \int_Y
 \bigg|\int_D \partial^{\alpha_\setu}_\bsy
 (G_p(\bsx,\bsy)- G_{\theta,p}(\bsx,\bsy))^2\,\rd\bsx\bigg|^2\,
 \rd\bsy.
\end{align}
Using the Leibniz product rule and the Cauchy--Schwarz inequality, we obtain for a general derivative (suppressing temporarily the dependence on $\bsy$ and the index $p$ for brevity)
\begin{align} \label{eq:four}
 &\bigg|\int_D \partial^\bsnu_\bsy (G(\bsx)- G_{\theta}(\bsx))^2\,\rd\bsx \bigg|\nonumber\\
 &= \bigg| \int_D\sum_{\bsm\le\bsnu} \binom{\bsnu}{\bsm}
  (\partial^\bsm_\bsy (G(\bsx)- G_{\theta}(\bsx)))\,
  (\partial^{\bsnu-\bsm}_\bsy (G(\bsx)- G_{\theta}(\bsx)))
  \,\rd\bsx \bigg| \nonumber\\
 &\le \sum_{\bsm\le\bsnu} \binom{\bsnu}{\bsm}
  \|\partial^\bsm_\bsy (G- G_{\theta})\|_{L_2(D)}\,
  \|\partial^{\bsnu-\bsm}_\bsy (G- G_{\theta})\|_{L_2(D)} \nonumber\\
 &\le \sum_{\bsm\le\bsnu} \binom{\bsnu}{\bsm}
  \big(\|\partial^\bsm_\bsy G\|_{L_2(D)} + \|\partial^\bsm_\bsy G_{\theta}\|_{L_2(D)}\big)\,
  \big(\|\partial^{\bsnu-\bsm}_\bsy G\|_{L_2(D)}
  + \|\partial^{\bsnu-\bsm}_\bsy G_{\theta}\|_{L_2(D)}\big).
\end{align}

We already obtained general bounds on $\|\partial^\bsnu_\bsy G_{\theta}\|_\infty$ for all $\bsnu\in\indx$ in \eqref{eq:dy-np-An} and \eqref{eq:dy-per-An}, independently of the target function $G$; these provide upper bounds on $\|\partial^\bsnu_\bsy G_{\theta,p}(\cdot,\bsy)\|_{L_2(D)}$ for $\bsy\in Y$, and all components $1\le p\le d_u$.

Now we need also the regularity bounds of the target function with respect to $\bsy$.
For many parametric PDE problems, including the specific elliptic example \eqref{eq:elliptic}, the regularity bounds with respect to $\bsy$ for the affine and Chebyshev models are known to take the same POD (``product and order dependent'') and SPOD (``smoothness-driven product and order dependent'') forms as in  \eqref{eq:dy-np-An} and \eqref{eq:dy-per-An}. For example, for the target function $G(\bsx,\bsy) = u(\bsx,\bsy)$ being the solution of~\eqref{eq:elliptic} with zero Dirichlet boundary condition, we know for the affine field \eqref{eq:a-np} that \citep[see e.g.,][]{CDS10}
\begin{align*}
  \|\nabla \partial_\bsy^\bsnu G(\cdot,\bsy)\|_{L_2(D)} 
  \le \tfrac{1}{a_{\min}}\,\|f\|_{(H^1(D))^*}\,|\bsnu|!\, (\tfrac{1}{a_{\min}}\,\bsbeta)^\bsnu,
\end{align*}
and for the Chebyshev field \eqref{eq:a-per} that \citep[see][]{KKKNS22,KKS20,KKS24} 
\begin{align*}
  \|\nabla \partial_\bsy^\bsnu G(\cdot,\bsy)\|_{L_2(D)} 
  \le \tfrac{1}{a_{\min}}\,\|f\|_{(H^1(D))^*}\,
  (2\pi)^{|\bsnu|}
  \sum_{\bsm\le\bsnu} |\bsm|!\,(\tfrac{1}{a_{\min}}\,\bsbeta)^{\bsm}\,\calS(\bsnu,\bsm),
\end{align*}
where $(H^1(D))^*$ is the dual space of $H^1(D)$, and
$a_{\min}>0$ denotes the lower bound on $a(\bsx,\bsy)$ and takes different values for the two random field models. Using the Poincar\'e inequality, these provide upper bounds on $\|\partial_\bsy^\bsnu G(\cdot,\bsy)\|_{L_2(D)}$ for all $\bsnu\in\indx$ and $\bsy\in Y$.

We state the result for generic target functions with the same form of parametric regularity bounds.

\begin{theorem}[Convergence of FNO with respect to $\bsy$] \label{thm:rate-y}
Under the setting of Theorem~\ref{thm:reg-y-An}, suppose that $C_L \le C$ and $\|P\|_\infty\,S_L\le T$ for some $C>0$ and $T>0$. Assume further that there is a ``summability exponent'' $p^* \in (0,1)$ such that
$\sum_{j\ge 1} \beta_j^{p^*} < \infty$.
\begin{itemize}
\item [\textnormal{(a)}] Suppose a target function
    $G(\bsx,\bsy)$ satisfies the regularity bound: for all multiindices
    $\bsnu\in\calI$, all $\bsy\in Y$, and all components $1\le p\le d_u$,
\begin{align} \label{eq:tar-np}
    \|\partial^\bsnu_\bsy G_p(\cdot,\bsy)\|_{L_2(D)} \le C\,|\bsnu|!\, (T\,\bsbeta)^{\bsnu}.
\end{align}
Then the FNO \eqref{eq:no2} with affine field \eqref{eq:a-np} satisfies for all components $1\le p\le d_u$,
\begin{align}
 \big\|\textstyle\int_D (G_p(\bsx,\cdot)- G_{\theta,p}(\bsx,\cdot))^2\,\rd\bsx \big\|_{\calH_Y^1}^2
 &\le 16\,C^4
 \sum_{\setu\subseteq\{1:s\}} \frac{1}{\gamma_\setu} \bigg(
 (|\setu|+1)!\, \prod_{j\in\setu} (T\,\beta_j)
 \bigg)^2. \label{eq:final-np}
\end{align}
Construct as in Section~\ref{sec:func}$(a)$ a good generating vector $\bsz$ for an $n$-point randomly-shifted lattice rule in weighted Sobolev space $\calH^1_Y$ with weights
\begin{align} \label{eq:weight-np}
 \gamma_\setu := \bigg(
 (|\setu|+1)!\,
 \prod_{j\in\setu} \bigg(
 \sqrt{\frac{(2\pi)^{2\lambda}}{2\zeta(2\lambda)\,2^\lambda}}\, T\,\beta_j\bigg)
 \bigg)^{\frac{2}{1+\lambda}},
 \;
 \lambda :=
 \begin{cases}
 \frac{1}{2-2\delta}, \delta\in (0,\frac{1}{2}),
 & \mbox{if } p^* \in (0, \frac{2}{3}], \\
 \frac{p^*}{2-p^*} & \mbox{if } p^* \in (\frac{2}{3}, 1).
 \end{cases}
\end{align}
Then the expression \eqref{eq:norm-y} is of order $\calO(n^{-r})$ with $r := \min(1-\delta,\frac{1}{p^*}-\frac{1}{2})$, where the implied constant is independent of the parametric dimension $s$.

\item [\textnormal{(b)}] Suppose a target function $G(\bsx,\bsy)$
    satisfies the regularity bound: for all multiindices
    $\bsnu\in\calI$, all $\bsy\in Y$, and all components $1\le p\le d_u$,
\begin{align} \label{eq:tar-per}
  \|\partial^\bsnu_\bsy G_p(\cdot,\bsy)\|_{L_2(D)}
  \le C\,(2\pi)^{|\bsnu|} \sum_{\bsm\le\bsnu} |\bsm|!\,(T\,\bsbeta)^{\bsm}\,\calS(\bsnu,\bsm).
\end{align}
Then the FNO \eqref{eq:no2} with Chebyshev field \eqref{eq:a-per} satisfies for all components $1\le p\le d_u$,
\begin{align}
 &\big\|\textstyle\int_D (G_p(\bsx,\cdot)- G_{\theta,p}(\bsx,\cdot))^2\,\rd\bsx \big\|_{\calH_Y^\alpha}^2 \nonumber\\
 &\le 16\,C^4\!\!
 \sum_{\setu\subseteq\{1:s\}} \!\!\!\!
 \frac{(2\pi)^{2\alpha|\setu|}}{\gamma_\setu} \bigg(
 \sum_{\bsm_\setu\le\bsalpha_\setu} (|\bsm_\setu|+1)!\,
 \calS(\bsalpha_\setu,\bsm_\setu)\, \prod_{j\in\setu} (T\,\beta_j)^{m_j}\!
 \bigg)^2. \label{eq:final-per}
\end{align}
Construct as in Section~\ref{sec:func}$(b)$ a good generating vector $\bsz$ for an $n$-point lattice rule in weighted Korobov space $\calH^\alpha_Y$ with weights
\begin{align} \label{eq:weight-per}
 \gamma_\setu &:=
 (2\pi)^{2\alpha|\setu|}
 \sum_{\bsm_\setu\le\bsalpha_\setu}\!\!
 \bigg(
 (|\bsm_\setu|+1)!\,
 \prod_{j\in\setu}
 \frac{(T\,\beta_j)^{m_j}\calS(\alpha,m_j)}{\sqrt{2\zeta(2\alpha\lambda)}}
 \bigg)^{\frac{2}{1+\lambda}},
 \\
 \alpha &:= \Big\lfloor \tfrac{1}{p^*} + \tfrac{1}{2} \Big\rfloor,
 \quad
 \lambda := \tfrac{p^*}{2-p^*}. \nonumber
\end{align}
Then the expression \eqref{eq:norm-y} is of order $\calO(n^{-r})$ with $r := \frac{1}{p^*}-\frac{1}{2}$, where the implied constant is independent of the parametric dimension $s$.
\end{itemize}
If instead we have $C_L > C$ or $\|P\|_\infty\,S_L > T$, then the bounds \eqref{eq:final-np}, \eqref{eq:final-per} will hold with $C$ and/or $T$ replaced by $C_L$ and/or $\|P\|_\infty\,S_L$, respectively.
In this case, the lattice training points can still be constructed with the weights \eqref{eq:weight-np}, \eqref{eq:weight-per}, and the big-$\calO$ bounds on the expression \eqref{eq:norm-y} still hold, but now with enlarged implied constants that depend on $C_L$ and/or $\|P\|_\infty\,S_L$, and is still bounded independently of $s$.
\end{theorem}

\begin{proof}
The result is proved by adapting the proof of \citet[Theorem~3.2]{KKNS26a} (in which there is no variable $\bsx$) to the present case where $\bsx$ appears but is integrated out. Then we conclude the appropriate choice of weights from \citet[Theorem~3.3]{KKNS26a}. The proof is given in Appendix~\ref{sec:regularity_proofs_y}.
\end{proof}

\subsection{Regularity with respect to \texorpdfstring{$\bsx$}{$x$}} \label{sec:reg-x}

In order to obtain a regularity bound for the FNO with respect to $\bsx$, we need to make a precise choice of the basis functions $\psi_j(\bsx)$ in the expansion of $a(\bsx,\bsy)$. This analysis is new and is built upon the assumption that we are approximating functions with a prescribed smoothness characterized by the decay of their Fourier coefficients. 

\subsubsection{Specifying the Random Field Basis Functions}

Recall from Section~\ref{sec:trunc} that we use a function $r(\bsh)$ to measure a kind of ``radius'' of a frequency index $\bsh\in\bbZ^d$. We enumerate $\{\bsh\in\bbZ^d\} = \{\bsh_j :j=1,2,\ldots\}$ so that
\[
  \frac{1}{r(\bsh_1)} \ge \frac{1}{r(\bsh_2)} \ge \cdots.
\]
Then in the series expansion \eqref{eq:a-np} or \eqref{eq:a-per} of the random field $a(\bsx,\bsy)$ we define 
\begin{align} \label{eq:psi}
 \psi_0(\bsx) = 1, \quad
 \psi_{2j-1}(\bsx) := \frac{\eta\, e^{2\pi\ri\bsh_j\cdot\bsx}}{r(\bsh_j)}
 \quad\mbox{and}\quad
 \psi_{2j}(\bsx) := \ri\,\frac{\eta\, e^{2\pi\ri\bsh_j\cdot\bsx}}{r(\bsh_j)}
 \quad\mbox{for } j\ge 1,
\end{align}
where $\eta$ is a scaling factor to make sure that
$\min_{\bsx \in D} a(\bsx, \bsy)\ge a_{\min} > 0$ for all $\bsy \in Y.$
Therefore the sequence $(\beta_j)_{j\ge 1}$ in \eqref{eq:ass-y} can be taken to be
\begin{align} \label{eq:beta-fix}
 \beta_{2j-1} = \|\psi_{2j-1}\|_{L_\infty(D)} 
 = \beta_{2j} = \|\psi_{2j}\|_{L_\infty(D)} = \frac{\eta}{r(\bsh_j)}
 \quad\mbox{for } j\ge 1.
\end{align}

For $M>0$ we consider the truncated index set 
\begin{align} \label{eq:index}
  \calA := \big\{\bsh\in\bbZ^d : r(\bsh) \le M \big\}. 
\end{align}
We then conclude that 
\[
  s = 2\,|\calA|\,,
\] 
that is, the parametric dimension $s$ is given explicitly by twice the cardinality of the truncated frequency index set $\calA$.

For the derivative with respect to $\bsx$, for all $\bsnu\ne\bszero$ satisfying $\|\bsnu\|_\infty \le \omega$ we have 
\begin{align*}
  |\partial_\bsx^\bsnu \psi_{2j-1}(\bsx)| = |\partial_\bsx^\bsnu \psi_{2j}(\bsx)|
  = \bigg|\frac{\eta \, e^{2\pi\ri\bsh_j\cdot\bsx}}{r(\bsh_j)}\,\prod_{p\in\supp(\bsh_j)} (2\pi\ri h_{j,p})^{\nu_p}\bigg|.
\end{align*}
Using $|y_j-\frac{1}{2}|\le \frac{1}{2} \le 1$ for the affine field \eqref{eq:a-np} and $|\sin(2\pi y_j)|\le 1$ for the Chebyshev field \eqref{eq:a-per}, we obtain with \eqref{eq:psi} for both fields
\begin{align} \label{eq:dx-a}
  |\partial_\bsx^\bsnu a(\bsx,\bsy)|
  &\le \sum_{j=1}^s |\partial_\bsx^\bsnu \psi_j(\bsx)| 
  = \sum_{j=1}^{s/2} 
  \big( |\partial_\bsx^\bsnu \psi_{2j-1}(\bsx)| + |\partial_\bsx^\bsnu \psi_{2j}(\bsx)| \big) \nonumber\\
  &\le 2\,(2\pi)^{|\bsnu|} \,\eta
  \sum_{j=1}^{s/2} \frac{\prod_{p\in\supp(\bsh_j)} |h_{j,p}|^{\nu_p}}{r(\bsh_j)} 
  = 2\, (2\pi)^{|\bsnu|} \,\eta\sum_{\bsh\in\calA} \frac{\prod_{p\in\supp(\bsh)} |h_p|^{\nu_p}}{r(\bsh)} \nonumber\\
  &\le (2\pi)^{|\bsnu|} J, \qquad 
  J := 2\eta \sum_{\bsh\in\calA} \frac{\prod_{p\in\supp(\bsh)} |h_p|^\omega}{r(\bsh)}.
\end{align}

For a kernel $K$ we have
\begin{align} \label{eq:dx-Kv}
  \partial_\bsx^\bsnu (K\, v)(\bsx) 
  = \sum_{\bsh\in\calA} \Big(\prod_{p\in\supp(\bsh)} (2\pi\ri h_p)^{\nu_p}\Big)\,
  \widehat{\kappa}(\bsh)\,\widehat{v}(\bsh)\, 
  e^{2\pi\ri\bsh\cdot\bsx}
  =  (K(\partial_\bsx^\bsnu v))(\bsx),
\end{align}
where the last equality allows us to move the derivative to either the function $v$ as we did, or to the kernel, whichever is more convenient.

\subsubsection{Specializing to the Hyperbolic Cross} \label{sec:hc-theory}

Later we focus on the weighted Korobov space setting in Section~\ref{sec:func}(c), where the function $r^{\rm hc}(\bsh)$ defined in \eqref{eq:norm-per-rD} gives a hyperbolic cross index set $\calA^{\rm hc}$, which has been considered in many papers \citep[see e.g.,][]{CKNS20,CKNS21,KMN25}.
Upper and lower bounds on $|\calA^{\rm hc}|$ can be found in e.g.,~\citet{CKNS20}, which 
depend on the spatial dimension~$d$, the smoothness parameter~$\omega$, and the weights $(\rho_\setu)_{u\subseteq\{1:d\}}$ of the weighted Korobov space~$\calH^\omega_D$. 

To ensure that $a_{\min} > 0$ for both the affine and Chebyshev fields, using the explicit definition of $r^{\rm hc}(\bsh)$ in \eqref{eq:norm-per-rD}, we can derive a sufficient condition
\begin{align*}
 1 > \sum_{j=1}^\infty \beta_j
  = 2 \sum_{j=1}^\infty \frac{\eta}{r(\bsh_j)}
 = 2 \eta \sum_{\setu\subseteq\{1:d\}} 
 \sum_{\satop{\bsh\in\bbZ^d}{\supp(\bsh)=\setu}}  
 \frac{\rho_\setu}{\prod_{p\in\setu} |h_p|^{2\omega}} 
 &= 2 \eta  \sum_{\setu\subseteq\{1:d\}} \rho_\setu
 \prod_{p\in\setu} \sum_{h_p\in\bbZ\setminus\{0\}} \frac{1}{|h_p|^{2\omega}} \\
 &= 2 \eta \sum_{\setu\subseteq\{1:d\}} \rho_\setu\, [2\zeta(2\omega)]^{|\setu|},
\end{align*}
which turns into a condition on the value of $\eta$ relative to the smoothness parameter $\omega$ and the weight parameters $\rho_\setu$.

Since the sequence $\beta_j$ is non-increasing, we have for all $\lambda > 1/(2\omega)$,
\begin{align*}
 \beta_{2j-1}^\lambda = \beta_{2j}^\lambda
 =  \frac{\eta^\lambda}{r(\bsh_j)^\lambda}
 \le \frac{\eta^\lambda}{j} \sum_{i=1}^j  \frac{1}{r(\bsh_i)^\lambda}
 \le \frac{\eta^\lambda}{j} \sum_{\bsh\in\bbZ^d}  \frac{1}{r(\bsh)^\lambda} 
 &= \frac{\eta^\lambda}{j} \sum_{\setu\subseteq\{1:d\}} \rho_\setu^\lambda\, [2\zeta(2\omega\lambda)]^{|\setu|}.
\end{align*}
Thus $\beta_j$ decays like $j^{-2\omega}$, and we have $\sum_{j\ge 1} \beta_j^{p^*} < \infty$ with the summability exponent $p^*$ in Theorem~\ref{thm:rate-y} arbitrarily close to 
\[
  p^* \approx \frac{1}{2\omega}.
\]
Assuming that $\omega\ge 1$ is an integer, Theorem~\ref{thm:rate-y}(a) gives for the affine field the convergence rate $r\approx 1$, and Theorem~\ref{thm:rate-y}(b) gives for the Chebyshev field $r \approx 2\omega - \frac{1}{2}$ by setting effectively $\alpha = 2\omega$.

The value of $J$ defined in \eqref{eq:dx-a} can be bounded by
\begin{align*} 
  J \le 2\eta \sum_{\bsh\in\bbZ^d} \frac{\prod_{p\in\supp(\bsh)} |h_p|^\omega}{r(\bsh)}\,
  \Big(\frac{M}{r(\bsh)}\Big)^{1/2}
  =  2\eta\sqrt{M}\,\sum_{\setu\subseteq\{1:d\}} \rho_\setu^{3/2}\,[2\zeta(2\omega)]^{|\setu|},
\end{align*}
where we multiplied the sum over $\bsh\in\calA$ by the extra factor $(\frac{M}{r(\bsh)})^{1/2}\ge 1$ and then extended the sum to all $\bsh\in\bbZ^d$. This estimate is only useful in very high dimensions $d$, since it can be bounded independently of $d$ if the weights $\rho_\setu$ decay sufficiently fast.

\subsubsection{Regularity of the FNO with respect to \texorpdfstring{$\bsx$}{$x$}}

With the above preparations, we can now present the regularity bound with respect to $\bsx$. Our result below holds for general index sets of the form \eqref{eq:index}, not just the hyperbolic cross. However, depending on the choice of the index set, 
the value of $J$ defined in \eqref{eq:dx-a} can be quite large.

A key observation is that, while the derivatives of $a(\bsx,\bsy)$ with respect to $\bsy$ vanish for many multiindices $\bsnu$, see \eqref{eq:dy-a-np} and \eqref{eq:dy-a-per}, the derivatives of $a(\bsx,\bsy)$ with respect to $\bsx$ do not vanish, see \eqref{eq:dx-a}. Moreover, here we do not benefit from the product $\bsbeta^\bsnu = \prod_{j=1}^s \beta_j^{\nu_j}$ as we did in the regularity bounds with respect to $\bsy$.

\begin{theorem}[Regularity bound for FNO in $\bsx$] \label{thm:reg-x-An}
Consider the FNO \eqref{eq:recur} where the random field $a(\bsx,\bsy)$ is either affine~\eqref{eq:a-np} or Chebyshev~\eqref{eq:a-per}, with basis functions $\psi_j$ defined by \eqref{eq:psi} and index set $\calA$ defined by \eqref{eq:index} for some $M>0$. Let $J$ be defined in \eqref{eq:dx-a} for some $\omega\ge 1$. Suppose there exist positive sequences $(R_\ell)_{\ell\ge 1}$ and $(A_n)_{n\ge 1}$ such that \eqref{eq:ass-y} hold.
Suppose additionally \eqref{eq:common} holds, and define
\begin{align} \label{eq:CSP-x}
  &\widetilde{C}_L := \max\bigg\{
 \|G_\theta\|_\infty,\;
 \|Q\|_\infty\,\frac{\xi\,\Pi_L}{\sum_{k=1}^L \Pi_k} \bigg\}, \\
  &\widetilde{S}_L := 1 + \frac{\|P\|_\infty J}{\xi} \sum_{k=1}^L \Pi_k, \quad
  \Pi_k := \prod_{t=1}^k (\xi\,\tau\, R_t). \nonumber
\end{align}

For $L\ge 1$ and any multiindex $\bsnu\in\indx$ with $0\le \|\bsnu\|_\infty \le \omega$, we have the following regularity bound for the FNO \eqref{eq:recur} with respect to $\bsx$:
\begin{align} \label{eq:dx-An}
 \|\partial^\bsnu_\bsx G_{\theta}\|_\infty
 \le \widetilde{C}_L\, (2\pi)^{|\bsnu|}\, |\bsnu|!\,\,\widetilde{S}_L^{|\bsnu|}.
\end{align}
\end{theorem}

\begin{proof}
The proof is given in Appendix~\ref{sec:regularity_proofs_x}.
\end{proof}

\subsubsection{Norm of \texorpdfstring{$(G - G_\theta)^2$}{$(G - G_\theta)^2$} with respect to \texorpdfstring{$\bsx$}{$x$}}

We are now ready to obtain a bound on the norm in \eqref{eq:norm-x}. From the definition of the norm \eqref{eq:norm-per-D}, we have
\begin{align} \label{eq:norm-x-2}
 &\big\|(G_p(\cdot,\bsy_i)- G_{\theta,p}(\cdot,\bsy_i))^2 \big\|_{\calH^\omega_D}^2 \nonumber\\
 &= \sum_{\setu\subseteq\{1:d\}} 
 \frac{1}{\rho_\setu}
 \int_{D_\setu} \bigg| \int_{D_{-\setu}}
 \partial^{\omega_\setu}_\bsx
 (G_p(\bsx,\bsy_i)- G_{\theta,p}(\bsx,\bsy_i))^2\,
 \rd\bsx_{-\setu}\bigg|^2\rd\bsx_\setu \nonumber\\
 &\le \sum_{\setu\subseteq\{1:d\}} 
 \frac{1}{\rho_\setu}
 \int_D \Big| \partial^{\omega_\setu}_\bsx 
 (G_p(\bsx,\bsy_i)- G_{\theta,p}(\bsx,\bsy_i))^2\Big|^2\,
 \rd\bsx \\
 &= \sum_{\setu\subseteq\{1:d\}} 
 \frac{1}{\rho_\setu}\,
 \big\| \partial^{\omega_\setu}_\bsx 
 (G_p(\cdot,\bsy_i)- G_{\theta,p}(\cdot,\bsy_i))^2\big\|^2_{L_2(D)}. \nonumber
\end{align}
The challenge in \eqref{eq:norm-x-2} compared to \eqref{eq:norm-y-2} is that we now have an extra square inside the integral with respect to $\bsx$. Consequently, we need bounds on the $L_4(D)$ norm of derivatives instead of $L_2(D)$ norm.

In analogy to \eqref{eq:four}, for a general derivative (suppressing temporarily the dependence on~$\bsy_i$ and the index $p$ for brevity), we obtain with the Cauchy--Schwarz inequality,
\begin{align} \label{eq:four-x}
 &\big\|\partial^\bsnu_\bsx (G- G_{\theta})^2\big\|_{L_2(D)} 
 = \bigg\|
 \sum_{\bsm\le\bsnu} \binom{\bsnu}{\bsm}
  (\partial^\bsm_\bsx (G- G_{\theta}))\,
  (\partial^{\bsnu-\bsm}_\bsx (G- G_{\theta})) \bigg\|_{L_2(D)} \nonumber\\
 &\le
 \sum_{\bsm\le\bsnu} \binom{\bsnu}{\bsm}
  \|\partial^\bsm_\bsx (G- G_{\theta})\|_{L_4(D)}\,
  \|\partial^{\bsnu-\bsm}_\bsx (G- G_{\theta})\|_{L_4(D)} \nonumber\\
 &\le \sum_{\bsm\le\bsnu} \binom{\bsnu}{\bsm}
  \big(\|\partial^\bsm_\bsx G\|_{L_4(D)} + \|\partial^\bsm_\bsx G_{\theta}\|_{L_4(D)}\big)\,
  \big(\|\partial^{\bsnu-\bsm}_\bsx G\|_{L_4(D)} + \|\partial^{\bsnu-\bsm}_\bsx G_{\theta}\|_{L_4(D)}\big). 
\end{align}

We already obtained a general bound on $\|\partial^\bsnu_\bsx G_\theta\|_\infty$ for all $\bsnu\in\indx$ with $0\le|\bsnu|\le\omega$ in \eqref{eq:dx-An}, independent of the target function $G$; these provide upper bounds on $\|\partial^\bsnu_\bsx G_{\theta,p}(\cdot,\bsy_i)\|_{L_4(D)}$ for all $1\le p\le d_u$ and $\bsy_i\in Y$.

Now we need to know the regularity of the target function with respect to $\bsx$. In the theorem below we make an assumption on the form of regularity bound for a generic function. We demonstrate this with an example in Appendix~\ref{sec:example}.

\begin{theorem}[Convergence of FNO with respect to $\bsx$] \label{thm:rate-x}
Under the setting of Theorem~\ref{thm:reg-x-An}, suppose that $\widetilde{C}_L\le \widetilde{C}$ and $\widetilde{S}_L \le \widetilde{T}$ for some $\widetilde{C}>0$ and $\widetilde{T}>0$. 
Suppose a target function $G(\bsx,\bsy)$ satisfies the regularity bound: for all multiindices $\bsnu\in\indx$ with $0\le \|\bsnu\|_\infty\le\omega$ and all $\bsy\in Y$, and all components $1\le p\le d_u$,
\begin{align} \label{eq:tar-x}
 \|\partial_{\bsx}^{\bsnu}G_p(\cdot,\bsy)\|_{L_4(D)} 
 \le \widetilde{C}\, (2\pi)^{|\bsnu|}\, |\bsnu|!\,\,\widetilde{T}^{|\bsnu|}.
\end{align}
Then the FNO \eqref{eq:no2} with either affine \eqref{eq:a-np} or Chebyshev \eqref{eq:a-per} field satisfies for all components $1\le p\le d_u$,
\begin{align} \label{eq:final-x}
 \big\|(G_p(\cdot,\bsy)- G_{\theta,p}(\cdot,\bsy))^2 \big\|_{\calH_D^\omega}^2 
 \le 16\,\widetilde{C}^4 \sum_{\setu\subseteq\{1:d\}} 
 \frac{(2\pi\,\widetilde{T})^{2\omega|\setu|}\, [\,(\omega|\setu|+1)!\,]^2}{\rho_\setu}.
\end{align}
Construct as in Section~\ref{sec:func}$(c)$ a good generating vector $\bsw$ for an $m$-point lattice rule in weighted Korobov space $\calH^\omega_D$ with weights
\begin{align} \label{eq:weight-x}
 \rho_\setu := 
 (2\pi)^{2\omega|\setu|}
 \bigg(
 (\omega|\setu|+1)!\,
 \bigg(\frac{\widetilde{T}^\omega}{\sqrt{2\zeta(2\omega\lambda)}}\bigg)^{|\setu|}\bigg)^{\frac{2}{1+\lambda}}, \quad
 \lambda := \tfrac{1}{2(\omega-\epsilon)}, \quad\epsilon\in (0,\tfrac{1}{2}).
\end{align}
Then the expression \eqref{eq:norm-x} is of order $\calO(m^{- (\omega -\epsilon)})$, where the implied constant depends on the physical dimension $d$, which is expected to be low.
\end{theorem}

\begin{proof}
The proof is given in Appendix~\ref{sec:regularity_proofs_x}.
\end{proof}

\subsection{Combined Generalization Error Bound}

Combining all of the above results, we conclude the following generalization error bounds.

\begin{theorem}[Combined generalization error bound] \label{thm:main}
Under the setting of Theorem~\ref{thm:rate-x}, and for the
respective settings \textnormal{(a)} and \textnormal{(b)} of Theorem~\ref{thm:rate-y}:

\begin{itemize}
\item [\textnormal{(a)}] If a FNO \eqref{eq:no2} with affine field \eqref{eq:a-np} 
is trained using carefully constructed randomly-shifted lattice points $\{\bsy_i\}_{i=1}^n\subset Y$ and spatial lattice points $\{\bsx_k\}_{k=1}^m\subset D$, and the training error reaches the threshold ${\tt tol} \in
(0,1)$, then the generalization error \eqref{eq:gen2} is bounded by
\begin{align} \label{eq:main-np}
  \calE_G 
  \le {\tt tol} + \calO\big(n^{-r/2} + m^{-(\omega-\epsilon)/2}\big),
  \quad
  r := \min\big(1-\delta,\;\tfrac{1}{p^*}-\tfrac{1}{2}\big),
  \quad \delta,\epsilon\in (0,\tfrac{1}{2}). 
\end{align}

\item [\textnormal{(b)}] If a FNO \eqref{eq:no2} with Chebyshev field \eqref{eq:a-per} 
is trained using carefully constructed lattice points $\{\bsy_i\}_{i=1}^n \subset Y$ and spatial lattice points $\{\bsx_k\}_{k=1}^m \subset D$, and the training error reaches the threshold ${\tt tol} \in
(0,1)$, then the generalization error \eqref{eq:gen2} is bounded by
\begin{align} \label{eq:main-per}
  \calE_G 
  \le {\tt tol} + \calO\big(n^{-r/2} + m^{-(\omega-\epsilon)/2}\big),
  \quad
  r := \tfrac{1}{p^*}-\tfrac{1}{2},
  \quad \epsilon\in (0,\tfrac{1}{2}). 
  \hspace{2.5cm}
\end{align}
\end{itemize}
\end{theorem}

Recall that for the theoretical setting in Section~\ref{sec:hc-theory} with hyperbolic cross and integer $\omega\ge 1$, we have summability exponent $p^* \approx 1/(2\omega)$. Thus when $\omega\approx 1$ we expect close to the theoretical rate $\calO(n^{-1/2} + m^{-1/2})$.

Appendix~\ref{sec:nonlinearPQ} explains how our theory can be adapted from linear $P$ and $Q$ to cover the nonlinear case where $P$ and $Q$ are shallow neural networks.

\section{Numerical Experiments} \label{sec:experiments}

The experiments assess the performance of the rank-1 lattice approach in comparison to the standard tensor product grid and evaluate the impact of employing a hyperbolic cross for the index set $\calA$. The results show that, under appropriate conditions, incorporating both the rank-1 lattice structure and the hyperbolic cross does not degrade the performance of the FNO, while leading to a reduction in the number of model parameters.
Our code can be found on GitHub \url{https://github.com/JakobDilen/Lattice_FNO}.

\subsection{Problem description}

The problem tackled in the experiments is the elliptic PDE on the torus 
\begin{align}\label{eq:problem_equation}
  - \nabla \cdot \bigl(a(\bsx,\bsy) \, \nabla u(\bsx, \bsy)\bigr) &= f(\bsx), & \bsx \in D, \, \bsy \in Y, \\
  \int_D u(\bsx,\bsy) \,\rd\bsx &= \kappa, & \bsy \in Y, \nonumber
\end{align}
where $a(\bsx,\bsy)$ is the affine random field~\eqref{eq:a-np} using the basis functions~\eqref{eq:psi} with $\eta = 0.2$ and the hyperbolic cross ``radius" $r^{\rm hc}(\bsh)$ defined in \eqref{eq:norm-per-rD} with $\rho_\setu = 1$ and $\omega=1$.
 
The domains are $D = \bbT^d = [0, 1]^d$ and $Y = [0, 1]^s$ with $d=2$ and $s = 2\,|\mathcal{A}|$. Further, we take $\kappa = 0$ and the forcing term is chosen to be 
\begin{equation*}
 f(\bsx) = -\phi\bigl(\sin(2\pi \, x_1)\bigr) - \phi\bigl(\sin(2\pi \, x_2)\bigr) + \ri\, \phi\bigl(\sin(2\pi \, x_1)\bigr) + \ri\, \phi\bigl(\sin(2\pi \, x_2)\bigr),
\end{equation*}
with $\phi(x) = \frac{500}{1 + \exp(-3x)}$ and such that $\int_D f(\bsx) \,\rd\bsx = 0$. Figure \ref{fig:forcing_term} shows the real and imaginary parts of $f(\bsx)$. 

\begin{figure}[t]
    \centering
    \begin{subfigure}[t]{0.42\textwidth}
        \centering
        \includegraphics[width=0.99\textwidth]{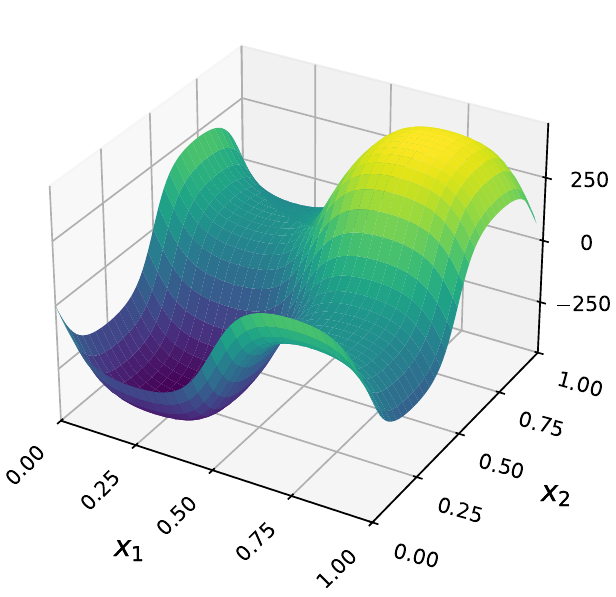}
        \caption{Real part of $f(\bsx)$}
    \end{subfigure}
    ~
    \begin{subfigure}[t]{0.42\textwidth}
        \centering
        \includegraphics[width=0.99\textwidth]{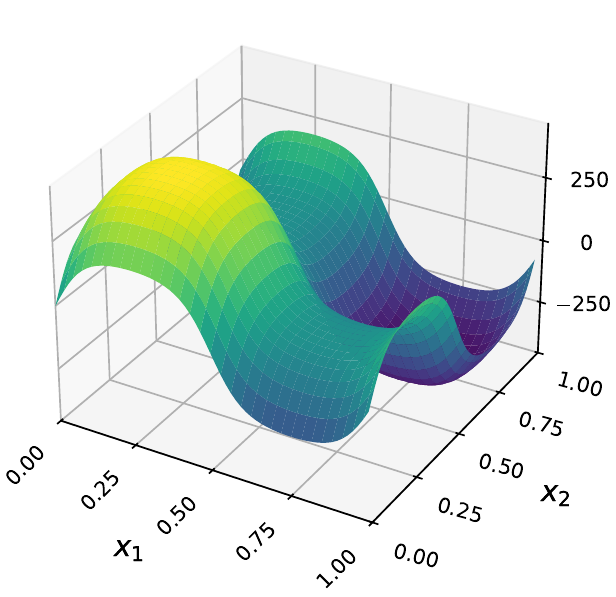}
        \caption{Imaginary part of $f(\bsx)$}
    \end{subfigure}
    \caption{Real and imaginary parts of the forcing term $f(\bsx)$.}
    \label{fig:forcing_term}
\end{figure}

To calculate the corresponding solution $u(\bsx,\bsy)$ for a given random field $a(\bsx,\bsy)$ we use the fact that $a(\bsx,\bsy)$, $f(\bsx)$ and $u(\bsx,\bsy)$ can be represented by their Fourier series. Substituting these into \eqref{eq:problem_equation} results in the equation 
\begin{equation*}
  4\pi^2 \sum_{\bsh\in\mathbb{Z}^d\setminus\{\bszero\}} 
   (\bsh \cdot \bsell) \, \widehat{a}(\bsell - \bsh) \, \widehat{u}(\bsh) = \widehat{f}(\bsell)
   \quad \mbox{for all } \bsell\in\bbZ^d,
\end{equation*}
where $\widehat{a}(\bsh)$, $\widehat{u}(\bsh)$ and $\widehat{f}(\bsh)$ are the Fourier coefficients of  $a(\bsx,\bsy)$, $f(\bsx)$ and $u(\bsx,\bsy)$, respectively. We excluded the $\bsh = \bszero$ term from the sum since it is zero. By taking $\bsell = \bszero$, we see that necessarily $\widehat{f}(\bszero) = \int_D f(\bsx) \,\rd\bsx = 0$.
This allows us to define an infinite-dimensional system of equations
\begin{equation*}
  A \, \widehat{u} = \widehat{f} \quad \text{with} \quad A := [4\pi^2 \, (\bsh \cdot \bsell) \, \widehat{a}(\bsell - \bsh)]_{\bsell, \bsh \in \mathbb{Z}^d\setminus\{\bszero\}},
\end{equation*}
from which we can solve to obtain the coefficients $\widehat{u}(\bsh)$ for $\bsh\ne\bszero$. The constant coefficient $\widehat{u}(\bszero) = \kappa$ is given by the PDE.
While it is not possible to solve this infinite system exactly, it is possible to solve it approximately by truncating all series to a finite index set. This allows us to calculate the solution $u(\bsx,\bsy)$ up to an arbitrary accuracy by taking these index sets large enough.
For the numerical experiments we take a box index set with $M=25$ to create the truncated system.
Figure~\ref{fig:input_output} shows a single realization $\bsy^* \in Y$ of an affine input field $a(\bsx,\bsy^*)$ and the corresponding output field $u(\bsx,\bsy^*)$ computed in this way.

\begin{figure}[t]
    \centering 
    \begin{subfigure}[t]{0.42\textwidth}
        \centering
        \includegraphics[width=0.99\textwidth]{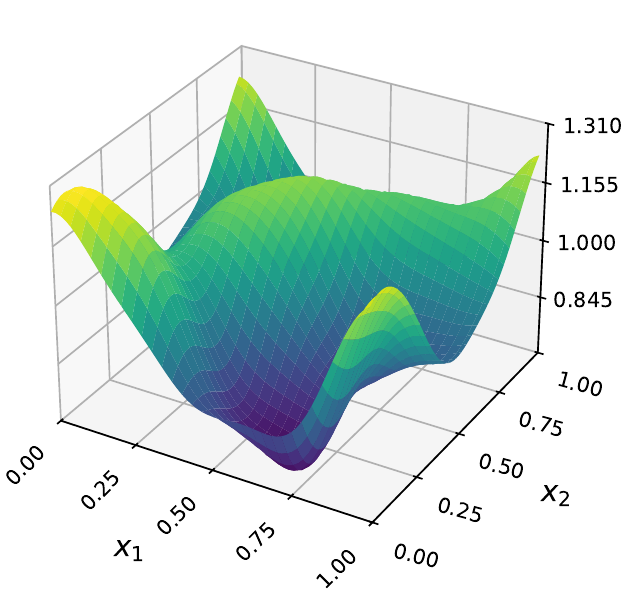}
        \caption{Real part of the input field}
    \end{subfigure}
    ~
    \begin{subfigure}[t]{0.42\textwidth}
        \centering
        \includegraphics[width=0.99\textwidth]{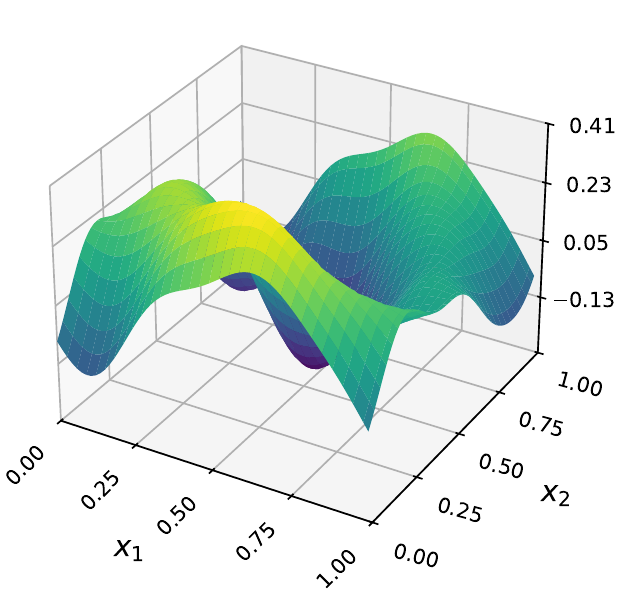}
        \caption{Imaginary part of the input field}
    \end{subfigure}
    ~
    \begin{subfigure}[t]{0.42\textwidth}
        \centering
        \includegraphics[width=0.99\textwidth]{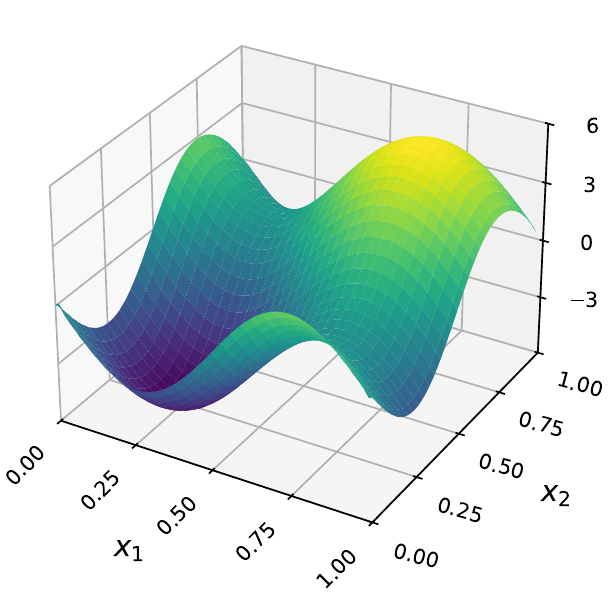}
        \caption{Real part of the solution}
    \end{subfigure}
    ~
    \begin{subfigure}[t]{0.42\textwidth}
        \centering
        \includegraphics[width=0.99\textwidth]{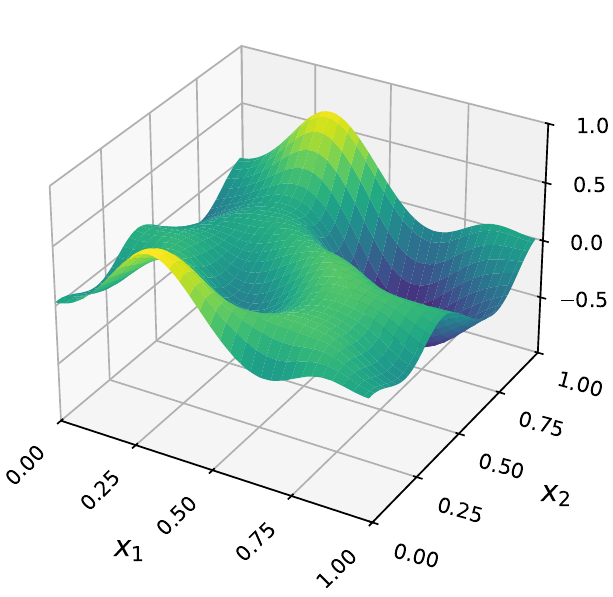}
        \caption{Imaginary part of the solution}
    \end{subfigure}
    \caption{Real and imaginary parts of a single realization $\bsy^*\in Y$ of an affine random input field $a(\bsx, \bsy^*)$ and the corresponding output field $u(\bsx, \bsy^*)$.}
    \label{fig:input_output}
\end{figure}

\subsection{Data generation}

In order to generate the data, it is necessary to choose a spatial discretization $\{\bsx_k\}_{k=1}^m$ in the physical domain $D = [0,1]^2$, i.e., $d=2$. We consider a regular tensor product grid with resolution of $m=32 \times 32$ points (see Figure~\ref{Fig:Points}(a)) for training and one with resolution of $m^*=64 \times 64$ points to estimate the generalization error (see below). We also consider an embedded lattice generating vector $\bsz = (1,721)$: we use $m = 1024$ points for training (see Figure~\ref{Fig:Points}(b)) and $m^{*}=4096$ points to estimate the generalization error.

We generate realizations of the affine random field \eqref{eq:a-np} from parameters $\{\bsy_i\}_{i=1}^n$ using an \emph{off-the-shelf} embedded lattice generating vector constructed by \citet{CKN06}, with $n = 1024$ and $s = 2\,|\calA|$, and we will consider different index sets $\calA$ to be described below.

\subsection{Training Setup}\label{subsec:training_param}

The models are trained with the mini-batch Adam optimization algorithm~\citep{ADAM} with batch size 16, using an adaptive \textit{“reduce learning rate on plateau”} scheduler with reduction factor $0.8$, patience $50$, and an initial learning rate of $10^{-2}$. For the experiments in Section~\ref{subsec:general_error_analysis} and Section~\ref{subsec:equal_parameter_analysis},
the networks are trained for 2000 epochs on a NVIDIA GeForce P100 GPU, while for the timing study in Section~\ref{subsec:training_time_exp} training is carried out on the CPU to better expose the influence of the parameter count on the training time. All architectures employ \emph{linear} lifting and projecting operators $P$ and $Q$ and use $L=2$ hidden layers with $d_\ell=16$ channels each. We have a scalar input (affine model) field and consider the quantity of interest $G(\bsx,\bsy) = u(\bsx,\bsy)$, so $d_a = d_u = 1$. However, as is standard, we use positional embedding, resulting in $d_a=3$.
The standard FNO is implemented using the NeuralOperator package~\citep{neuralop}, and the code for the lattice-based variants is made available on \href{https://github.com/JakobDilen/Lattice_FNO}{GitHub}.

To assess accuracy, we compute an estimate of the ``relative'' generalization error (compare with $\mathcal{E}_G$ in~\eqref{eq:gen2}) given by
\begin{equation*}
  \widetilde{\mathcal{E}}_G^2 
  := 
  \frac{
  \sum_{i=1}^{n^*} \sum_{k=1}^{m^*} \|G(\bsx_k,\bsy_i) - G_{\theta}(\bsx_k,\bsy_i)\|_2^2}
  {\sum_{i=1}^{n^*} \sum_{k=1}^{m^*} \|G(\bsx_k,\bsy_i)\|_2^2} . 
\end{equation*}
In contrast to the $n=1024$ samples and $m=1024$ physical points used during training, the approximate generalization error is evaluated on $n^*=4096$ samples (with the same lattice construction) and $m^*=4096$ physical points, and is reported relative to the average $\ell_2$-norm of $G(\cdot, \bsy_i)$ over this generalization dataset.

\subsection{General Error Analysis}\label{subsec:general_error_analysis}

To assess the effect of the rank-1 lattice discretization and the hyperbolic cross index set, we train four FNO architectures. The first is regular \FNO which uses a box frequency index set with $M = 9$, yielding $|\calA^{\rm box}| = 361$, together with a standard tensor product grid in the physical domain and serves as the reference model. The second is \LATFNO which keeps the box index set but replaces the regular grid by a rank-1 lattice in the physical domain. The third is \HCFNO which uses a (unweighted, $\rho_\setu= 1$) hyperbolic cross frequency index set with $M = 9$ and $\omega=1$, yielding $|\calA^{\rm hc}| = 33$, and a standard tensor product grid. The fourth
is \HCLATFNO which combines a rank-1 lattice discretization with a (unweighted, $\rho_\setu= 1$) hyperbolic cross frequency index set with $M = 9$ and $\omega=1$, yielding $|\calA^{\rm hc}| = 33$. Each architecture is trained in 5 independent runs using the training configuration from Section~\ref{subsec:training_param}. Table~\ref{tab:general_error_analysis} reports the corresponding average relative generalization errors $\widetilde\calE_G^2$ as defined above. 

The results show that both lattice-based variants attain similar errors to regular \FNO. However, \HCFNO and \HCLATFNO both use considerably fewer trainable parameters $N_{\rm FNO}$ (see \eqref{eq:N-FNO}) than regular \FNO, as also indicated in Table~\ref{tab:general_error_analysis}. This reduction in the parameter count for \HCFNO and \HCLATFNO is due to the use of the hyperbolic cross index set. Figure~\ref{fig:model_output} shows the output fields for \HCLATFNO with the same realization of the affine
input field $a(\bsx,\bsy^*)$ from Figure~\ref{fig:input_output}. The figure shows that the model output matches the output in Figure~\ref{fig:input_output}(c)--(d). The other variants achieve similar accuracy and figures would be visually indistinguishable. The relative $\ell_2$ distances are 2.85e--2, 2.62e--2, 2.78e--2 and 2.73e--2 for the \FNO, \LATFNO, \HCFNO and \HCLATFNO respectively.

\bgroup
\def\arraystretch{1.3}
\begin{table} [t]
    \centering
    \begin{tabular}{ |l|cccc|cc|cc| }
    \hline
    \multicolumn{9}{|c|}{Generalization error analysis results} \\
    \hline
      Model variant & $M$ & $|\calA| = \frac{s}{2}$ & $N_{\rm FNO}$ & $\%$ 
      & mean $\widetilde{\mathcal{E}}_G$ & std $\widetilde{\mathcal{E}}_G$ & Time (s) & \% \\   
    \hline
    \FNO & $9$ & 361 & 370850 & 100 & 2.74e--2  & 5.82e--4 & 7092 & 100\\
    \LATFNO & $9$ & 361  & 370850 & 100 & 2.78e--2  & 3.43e--4 & 6669 & 94\\
    \HCFNO & 9 & 33 & 35010 & 9.44 & 2.99e--2 & 1.13e--3 & 4039 & 57 \\
    \HCLATFNO  & $9$ & 33 & 35010 & 9.44 & 2.93e--2 & 3.46e--4& 3838 & 54 \\
    \hline
    \end{tabular}
    \caption{Approximate generalization error with total and relative (to the regular \FNO) number of parameters for the models described in Section~\ref{subsec:general_error_analysis}.}
    \label{tab:general_error_analysis}
\end{table}  
\egroup 

\begin{figure}[]
    \centering 
    \begin{subfigure}[t]{0.42\textwidth}
        \centering
        \includegraphics[width=0.99\textwidth]{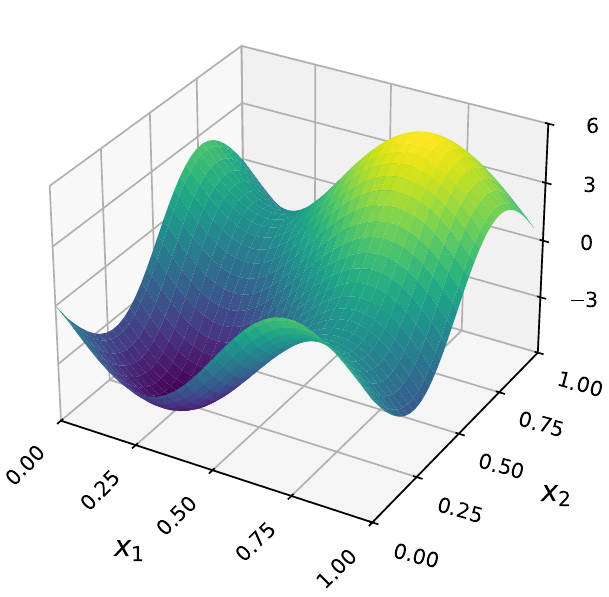}
        \caption{Real part \HCLATFNO solution}
    \end{subfigure}
    ~
    \begin{subfigure}[t]{0.42\textwidth}
        \centering
        \includegraphics[width=0.99\textwidth]{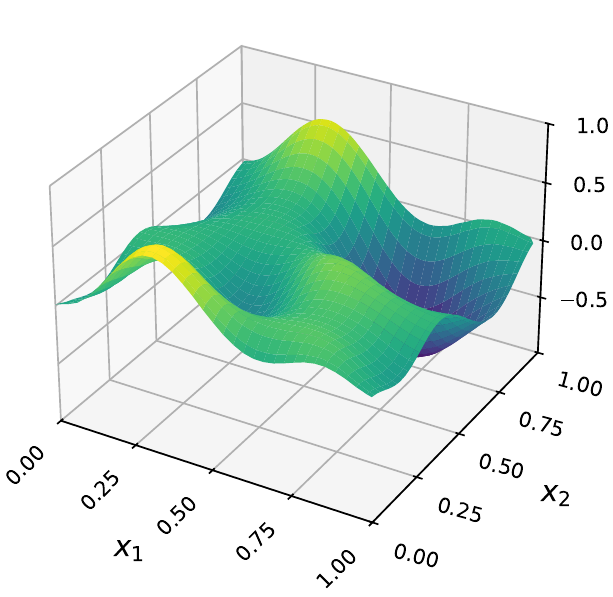}
        \caption{Imaginary part \HCLATFNO solution}
    \end{subfigure}
    \caption{Real and imaginary parts of output fields generated by
    \HCLATFNO for the realization of the affine input field $a(\bsx, \bsy^*)$ from Figure~\ref{fig:input_output}.}
    \label{fig:model_output}
\end{figure}

\subsection{Accuracy at Matched Parameter Counts}\label{subsec:equal_parameter_analysis}

The empirical errors in Section~\ref{subsec:general_error_analysis} indicate that \HCLATFNO can achieve comparable accuracy with fewer parameters, which raises the question of whether it can also outperform regular \FNO when the parameter counts are matched. To investigate this, we compare \FNO with a box index set at $M = 3$ to \HCLATFNO with $M = 9$ and $\omega=1$, chosen so that both architectures have nearly the same number of trainable parameters. Each model is trained $5$ times using the training configuration from Section~\ref{subsec:training_param}, and Table~\ref{tab:equal_parameter_analysis} lists both the parameter counts and the average relative approximate generalization errors over these runs. The results show that, for this test problem, \HCLATFNO attains a similar average error to regular \FNO with a comparable number of parameters.

\bgroup
\def\arraystretch{1.3}
\begin{table} [t]
    \centering
    \begin{tabular}{ |l|cccc|cc| }
    \hline
    \multicolumn{7}{|c|}{Accuracy at Matched Parameter Counts} \\
    \hline
      Model variant & $M$ & $|\calA| = \frac{s}{2}$ & $N_{\rm FNO}$ & $\%$ 
      & mean $\widetilde{\mathcal{E}}_G$ & std $\widetilde{\mathcal{E}}_G$ 
      \\
    \hline
    \FNO & 3 & 49 & 51362 &100& 2.88e--2  & 4.05e--4 \\
    \HCLATFNO & 9 & 33 & 35010 & 68.2 &  2.85e--2 & 4.02e--4 \\
    \hline
    \end{tabular}
    \caption{Approximate generalization error with total and relative (to the regular \FNO) number of parameters for the models described in Section~\ref{subsec:equal_parameter_analysis}.}
    \label{tab:equal_parameter_analysis}
\end{table}  
\egroup

\subsection{Training Time Analysis}\label{subsec:training_time_exp}

Figure~\ref{fig:timings}(a) reports the average training time for four model variants (\FNO, \LATFNO, and both models with hyperbolic cross index set) for 5 runs for different
values of $M$, and confirms that \HCLATFNO is consistently faster than the other variants. The gap widens with increasing $M$ between the hyperbolic cross and box index set variants. This behavior is consistent with the fact that the savings due to the hyperbolic cross index set become more pronounced for increasing $M$. Figure~\ref{fig:timings}(b) shows the average approximate relative generalization error for the same data, in terms of the parameter count. If we take the parameter count as a rough indication of training time, then it is clear that the hyperbolic cross variants win. Combining this with Figure~\ref{fig:timings}(a) which shows the actual training time, we conclude that \HCLATFNO is faster to train than \HCFNO. The results indicate that \HCLATFNO outperforms all variants in terms of accuracy and training time.

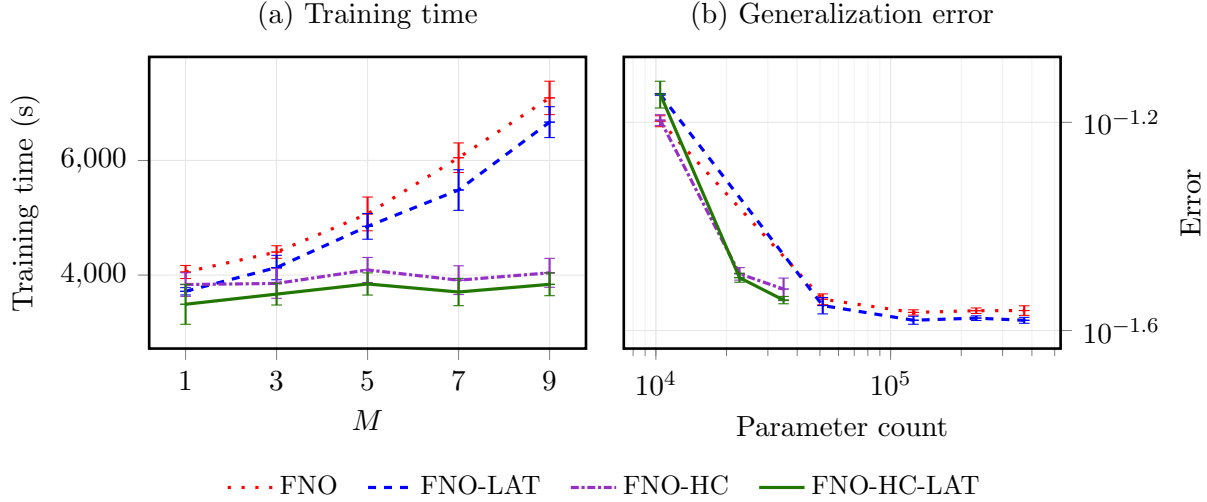
\begin{figure}[t]
    \centering
    \begin{tikzpicture}
        \begin{groupplot}[
            group style={
                group size=2 by 1,
                horizontal sep=0.5cm,
                xlabels at=edge bottom,
                ylabels at=edge left
            },
            comparison plot
        ]

        \nextgroupplot[
            title={(a) Training time},
            xlabel={$M$},
            ylabel={Training time (s)},
            xtick={1,3,5,7,9,11},
            legend to name=ComparisonLegend
        ]
            \addplot+[fno curve, comparison error bars]
                table [x=M, y=mean, y error=std] {data_files/normal_timings.txt};
            \addlegendentry{\FNO}

            \addplot+[latfno curve, comparison error bars]
                table [x=M, y=mean, y error=std] {data_files/lattice_timings.txt};
            \addlegendentry{\LATFNO}

            \addplot+[hcfno curve, comparison error bars]
                table [x=M, y=mean, y error=std] {data_files/hyperbolic_normal_timings.txt};
            \addlegendentry{\HCFNO}

            \addplot+[hclatfno curve, comparison error bars]
                table [x=M, y=mean, y error=std] {data_files/hyperbolic_timings.txt};
            \addlegendentry{\HCLATFNO}

        \nextgroupplot[
            title={(b) Generalization error},
            xlabel={Parameter count},
            ylabel={Error},
            xmode=log,
            ymode=log,
            ytick={0.025110,0.0630957},
           yticklabel pos=right,
           ytick pos=right,
           ylabel near ticks
        ]
            \addplot+[fno curve, comparison error bars]
                table [x=params, y=mean, y error=std] {data_files/normal_errors_parameter_count.txt};

            \addplot+[latfno curve, comparison error bars]
                table [x=params, y=mean, y error=std] {data_files/lattice_errors_parameter_count.txt};

            \addplot+[hcfno curve, comparison error bars]
                table [x=params, y=mean, y error=std] {data_files/hyperbolic_normal_errors_parameter_count.txt};

            \addplot+[hclatfno curve, comparison error bars]
                table [x=params, y=mean, y error=std] {data_files/hyperbolic_errors_parameter_count.txt};

        \end{groupplot}
    \end{tikzpicture}

    \vspace{0.6em}
    \pgfplotslegendfromname{ComparisonLegend}

    \caption{Training time and approximate generalization error comparison between the regular \FNO and the \HCLATFNO when training on CPU.}
    \label{fig:timings}
\end{figure}

\section{Conclusion} \label{sec:conclusion}

The paper shows that Fourier neural operators (FNOs) can be combined with
rank-1 lattices in both the spatial and parametric domains and with hyperbolic-cross
frequency truncation to obtain provable theoretical generalization error bounds.
For a benchmark elliptic PDE, we demonstrate lower parameter counts
and training cost, while maintaining or slightly improving predictive accuracy.
The rigorous theoretical foundation provides avenues of future research,
especially the practical exploration of higher-dimensional domains and
the application of rank-1 lattice based FNOs to other practical problems.
The universal approximation theory for our lattice-based hyperbolic cross FNOs is the focus of a subsequent paper \citep{DKN26}.

\section*{Acknowledgments}

The authors are grateful to Boris Bonev, Nikola Kovachki, Nicolas Rousell and Jakob Zech for insightful conversations and valuable remarks.
Frances Kuo acknowledges financial support from the Australian Research Council Discovery
Project (DP240100769).

\appendix

\section{Appendix} \label{sec:proofs}

\subsection{Regularity Proofs with Respect to \texorpdfstring{$\bsy$}{$y$}}
\label{sec:regularity_proofs_y}
We first prove a general regularity bound for the FNO with respect to $\bsy$ without using the specific form \eqref{eq:common} for the derivative bounds on the activation function. We adapt from the proof of \citet[Theorems~2.1 and~2.2]{KKNS26a}. 

Note that $g^{[1]}(\bsx,\bsy)$ in our FNO \eqref{eq:recur} corresponds to the $0$th layer of the deep neural network in~\citet{KKNS26a}, both without any activation function (there is no activation function between our first hidden layer $T_1$ and the lifting operator $P$). Consequently, we need to adjust the base step for the induction proof and realign the indexing of sequences. In the sum in \eqref{eq:G-def} below, $R_\ell$ now appears instead of $R_{\ell-1}$ as defined in \citet[Formula~(2.8)]{KKNS26a}.

\begin{theorem} \label{thm:reg-y}
For the FNO \eqref{eq:recur} with either affine~\eqref{eq:a-np} or Chebyshev~\eqref{eq:a-per} field, suppose there exist positive sequences $(\beta_j)_{j\ge 1}$, $(R_\ell)_{\ell\ge 1}$, $(A_n)_{n\ge 1}$ satisfying \eqref{eq:ass-y}, but do not assume~\eqref{eq:common}.
Define the sequence $\Gamma_n^{[\ell]}$ recursively by
\begin{align}
  &\begin{cases}
  \Gamma_n^{[1]} := A_n & \mbox{for $n\ge 1$}, \\
  \Gamma_n^{[\ell]} := 
  \displaystyle\sum_{\lambda=1}^n A_\lambda\,R_\ell^\lambda\, \bbB_{n,\lambda}^{[\ell-1]}
  \hspace{1.8cm}
  & \mbox{for $\ell\ge 2$ and $n\ge1$,} 
  \end{cases} \label{eq:G-def}
  \\
  &\begin{cases}
  \bbB_{n,1}^{[\ell]} := \Gamma_n^{[\ell]} &\mbox{for $\ell\ge 1$ and $n\ge 1$}, \\
  \bbB_{n,\lambda}^{[\ell]} :=
  \displaystyle\sum_{i=\lambda-1}^{n-1} \binom{n-1}{i}\, \Gamma_{n-i}^{[\ell]}\,\bbB_{i,\lambda-1}^{[\ell]}
  & \mbox{for $\ell\ge 1$ and $n\ge\lambda\ge 2$}. 
  \end{cases} \label{eq:B-def}
\end{align}
For $L\ge 1$ and $\bsnu\in\indx$ with $\bsnu\ne\bszero$, we have the following regularity bounds with respect to $\bsy$:
\begin{enumerate}
\item [\textnormal{(a)}] The FNO \eqref{eq:recur} with the affine field \eqref{eq:a-np} satisfies
\begin{align} \label{eq:dy-np}
  \|\partial^\bsnu_\bsy G_{\theta}\|_\infty
  \le \|Q\|_\infty\,(R_1\,\|P\|_\infty\,\bsbeta)^{\bsnu}\,\Gamma_{|\bsnu|}^{[L]}.
  \qquad\qquad\qquad\qquad\quad\;
\end{align}
\item [\textnormal{(b)}] The FNO \eqref{eq:recur} with the Chebyshev field \eqref{eq:a-per} satisfies
\begin{align} \label{eq:dy-per}
  \|\partial^\bsnu_\bsy G_{\theta}\|_\infty
  \le \|Q\|_\infty\,(2\pi)^{|\bsnu|}
  \sum_{\bsm\le\bsnu} (R_1\,\|P\|_\infty\,\bsbeta)^{\bsm}\,\Gamma_{|\bsm|}^{[L]}\,\calS(\bsnu,\bsm).
\end{align}
\end{enumerate}
\end{theorem}

\begin{proofof}{Theorem~\ref{thm:reg-y}(a)}
We begin by considering any component $1\le p\le d_2$ of the function $g^{[1]}(\bsx,\bsy)$ in the recursive definition \eqref{eq:recur} of the FNO. Using \eqref{eq:dy-a-np} for the affine field, we have for any $j=1,\ldots,s$,
\begin{align*}
  \tfrac{\partial}{\partial_{y_j}} g^{[1]}_p(\bsx,\bsy)
  &= \sum_{q=1}^{d_1} W_{1,p,q}\, P_{q,1}\,\psi_j(\bsx)
   + \int_D \sum_{q=1}^{d_1} (\kappa_1)_{p,q}(\bsx-\bsx') \,P_{q,1}\,\psi_j(\bsx')\,\rd\bsx',
\end{align*}
which is independent of $\bsy$, and thus from the assumption \eqref{eq:ass-y} we obtain
\begin{align*}
  |\tfrac{\partial}{\partial_{y_j}} g^{[1]}_p(\bsx,\bsy)| 
  &\le \sum_{q=1}^{d_1} |W_{1,p,q}|\, \|P\|_\infty \,\beta_j 
   + \sum_{q=1}^{d_1} \sum_{\bsh\in\calA} |(\widehat{\kappa_1})_{p,q}(\bsh)|\,
   \|P\|_\infty \,\beta_j 
 \le R_1\,\|P\|_\infty \,\beta_j := \widetilde{\beta}_j,
\end{align*}
Since $\tfrac{\partial}{\partial_{y_j}} g^{[1]}_p(\bsx,\bsy)$ is independent of~$\bsy$, we deduce from the chain rule that for any $\bsnu\ne\bszero$,
\begin{align*}
  | \partial^\bsnu_\bsy\big(\sigma(g^{[1]}_p(\bsx,\bsy))\big) |
  &= \bigg|\sigma^{(|\bsnu|)}(g^{[1]}_p(\bsx,\bsy))\,
  \prod_{j=1}^s \big(\tfrac{\partial}{\partial_{y_j}} g^{[1]}_q(\bsx,\bsy)\big)^{\nu_j} \bigg| 
 \le A_{|\bsnu|}\,\widetilde\bsbeta^\bsnu = \widetilde\bsbeta^\bsnu\,\Gamma_{|\bsnu|}^{[1]},
\end{align*}
where we used $\|\sigma^{(n)}\|_\infty \le A_n = \Gamma_n^{[1]}$ for the sequence $\Gamma^{[\ell]}_n$ defined in \eqref{eq:G-def}.

Thus for any component $1\le p\le d_3$ of the function $g^{[2]}(\bsx,\bsy)$ we have
\begin{align*}
 |\partial^\bsnu_\bsy g^{[2]}_p(\bsx,\bsy)|
 &= \bigg|\sum_{q=1}^{d_2} W_{2,p,q}\,\partial^\bsnu_\bsy\big(\sigma(g^{[1]}_q(\bsx,\bsy))\big) 
 + \int_D \sum_{q=1}^{d_2} (\kappa_2)_{p,q}(\bsx-\bsx')\,\partial^\bsnu_\bsy\big(\sigma(g^{[1]}_q(\bsx',\bsy))\big)\,\rd\bsx' \bigg| \\
 &\le \sum_{q=1}^{d_2} |W_{2,p,q}|\,
 \widetilde\bsbeta^\bsnu\,\Gamma_{|\bsnu|}^{[1]}
 + \sum_{q=1}^{d_2} \sum_{\bsh\in\calA} |(\widehat\kappa_2)_{p,q}(\bsh)| \,
 \widetilde\bsbeta^\bsnu\,\Gamma_{|\bsnu|}^{[1]}
 \le R_2\, \widetilde\bsbeta^\bsnu\,\Gamma_{|\bsnu|}^{[1]}.
\end{align*}

Using the $\ell=2$ case as the base step for induction, we follow the argument in \citet[Proof of Theorem~2.1(a)]{KKNS26a} to obtain for all $\ell\ge 3$, 
\begin{align*}
 |\partial^\bsnu_\bsy \big(\sigma(g^{[\ell-1]}_q(\bsx,\bsy))\big)|
 \le\; & \widetilde\bsbeta^\bsnu\,\Gamma_{|\bsnu|}^{[\ell-1]}
 \qquad\mbox{for all } 1\le q\le d_\ell, 
 \\
 |\partial^\bsnu_\bsy g^{[\ell]}_p(\bsx,\bsy)|
 \le R_\ell\, & \widetilde\bsbeta^\bsnu\,\Gamma_{|\bsnu|}^{[\ell-1]}
 \qquad\mbox{for all } 1\le p\le d_{\ell+1},
\end{align*}
where we note that the index of the sequence $R_\ell$ is ``one ahead'' of $\Gamma_{|\bsnu|}^{[\ell-1]}$.

Finally we conclude that for every $1\le p\le d_u$ we have
\begin{align*}
  |\partial^\bsnu_\bsy G_{\theta,p}(\bsx,\bsy)|
  = \bigg|\sum_{q=1}^{d_{L+1}} Q_{p,q}\, \partial^\bsnu_\bsy \big(\sigma(g^{[L]}_q(\bsx,\bsy))\big)\bigg|
 \le \|Q\|_\infty\, 
 \widetilde\bsbeta^\bsnu\,\Gamma_{|\bsnu|}^{[L]}
 = \|Q\|_\infty\, (R_1\,\|P\|_\infty\,\bsbeta)^\bsnu\,\Gamma_{|\bsnu|}^{[L]}.
\end{align*}
This completes the proof for \eqref{eq:dy-np}.
\end{proofof}

\begin{proofof}{Theorem~\ref{thm:reg-y}(b)}
Consider now the FNO \eqref{eq:recur} with the Chebyshev field \eqref{eq:a-per}. 
For any $\bsnu\ne\bszero$ and any index $1\le p\le d_2$, we use \eqref{eq:dy-a-per} to obtain
\begin{align*}
  |\partial^\bsnu_\bsy g^{[1]}_p(\bsx,\bsy)|
  &= \bigg|
   \sum_{q=1}^{d_1} W_{1,p,q}\, P_{q,1}\,\partial^\bsnu_\bsy a(\bsx,\bsy) 
   + \int_D \sum_{q=1}^{d_1} (\kappa_1)_{p,q}(\bsx-\bsx') \,P_{q,1}\,\partial^\bsnu_\bsy a(\bsx',\bsy)\,\rd\bsx'
  \bigg| \\
  &\le \begin{cases}
  R_1\,\|P\|_\infty\,\beta_j\,(2\pi)^k = (2\pi)^k\,\widetilde\beta_j
  & \mbox{if } \bsnu=k\bse_j,\, k\ge 1, \\
  0 & \mbox{otherwise}.
  \end{cases}
\end{align*}
Following the argument in \citet[Proof of Theorem~2.1(b)]{KKNS26a}, we obtain
\begin{align*}
  |\partial^\bsnu_\bsy\big(\sigma(g^{[1]}_p(\bsx,\bsy))\big)|
  \le \sum_{\lambda=1}^{|\bsnu|} A_\lambda\,(2\pi)^{|\bsnu|} 
  \sum_{\satop{\bsm\le\bsnu}{|\bsm|=\lambda}} \widetilde\bsbeta^\bsm\,\calS(\bsnu,\bsm)
 = (2\pi)^{|\bsnu|} \sum_{\bsm\le\bsnu} \widetilde\bsbeta^\bsm\,
 A_{|\bsm|}\,\calS(\bsnu,\bsm),
\end{align*}
where $A_{|\bsm|} = \Gamma_{|\bsm|}^{[1]}$.

Thus for any component $1\le p\le d_3$ we have
\begin{align*}
 |\partial^\bsnu_\bsy g^{[2]}_p(\bsx,\bsy)|
 &= \bigg| \sum_{q=1}^{d_2} W_{2,p,q}\,\partial^\bsnu_\bsy\big(\sigma(g^{[1]}_q(\bsx,\bsy))\big) 
 + \int_D \sum_{q=1}^{d_2} (\kappa_2)_{p,q}(\bsx-\bsx') \,
 \partial^\bsnu_\bsy\big(\sigma(g^{[1]}_q(\bsx',\bsy))\big)\,\rd\bsx'
 \bigg| \\
 &\le R_2\,(2\pi)^{|\bsnu|} \sum_{\bsm\le\bsnu} \widetilde\bsbeta^\bsm\,
 \Gamma_{|\bsm|}^{[1]}\,\calS(\bsnu,\bsm).
\end{align*}

Using the $\ell=2$ case as the base step for induction, we follow \citet[Proof of Theorem~2.1(b)]{KKNS26a} to obtain for all $\ell\ge 3$, 
\begin{align*}
 |\partial^\bsnu_\bsy \big(\sigma(g^{[\ell-1]}_q(\bsx,\bsy))\big)|
 \le\, &(2\pi)^{|\bsnu|} \sum_{\bsm\le\bsnu} \widetilde\bsbeta^\bsm\,
 \Gamma_{|\bsm|}^{[\ell-1]}\,\calS(\bsnu,\bsm)
 \qquad\mbox{for all } 1\le q\le d_\ell,
 \\
 |\partial^\bsnu_\bsy g^{[\ell]}_p(\bsx,\bsy)|
 \le R_\ell\,&(2\pi)^{|\bsnu|} \sum_{\bsm\le\bsnu} \widetilde\bsbeta^\bsm\,
 \Gamma_{|\bsm|}^{[\ell-1]}\,\calS(\bsnu,\bsm)
 \qquad\mbox{for all } 1\le p\le d_{\ell+1},
\end{align*}
where we note that the index of the sequence $R_\ell$ is again ``one ahead'' of $\Gamma_{|\bsm|}^{[\ell-1]}$.

Finally, for every $1\le p\le d_u$ we have
\begin{align*}
  |\partial^\bsnu_\bsy G_{\theta,p}(\bsx,\bsy)|
  = \bigg|\sum_{q=1}^{d_{L+1}} Q_{p,q}\, \partial^\bsnu_\bsy \big(\sigma(g^{[L]}_q(\bsx,\bsy))\big)\bigg|
 \le \|Q\|_\infty\, (2\pi)^{|\bsnu|} \sum_{\bsm\le\bsnu} \widetilde\bsbeta^\bsm\,
 \Gamma_{|\bsm|}^{[L]}\,\calS(\bsnu,\bsm).
\end{align*}
This completes the proof for \eqref{eq:dy-per}.
\end{proofof}

\begin{proofof}{Theorem~\ref{thm:reg-y-An}}
We now use the special form \eqref{eq:common} of the sequence $A_n$ to obtain an explicit formula for $\Gamma_n^{[\ell]}$ defined in \eqref{eq:G-def}--\eqref{eq:B-def}. We adapt \citet[Lemma~6.3]{KKNS26a} by replacing all occurrences of $R_{\ell-1}$ there with $R_{\ell}$ (one head), to obtain
\begin{align*}
  \Gamma_n^{[\ell]}
 = \Phi_{\ell-1}\,\bigg(\sum_{k=0}^{\ell-1} \Phi_k\bigg)^{n-1}\,\xi\,\tau^n\,n!\,, 
 \quad\mbox{with}\quad
 \Phi_0 &:= 1, \quad
  \Phi_k := \prod_{t=1}^k (\xi\,\tau\, R_{t+1}) \quad\mbox{for $k\ge 1$}.
\end{align*}
Using a more convenient product notation defined in \eqref{eq:CSP}, we obtain
\begin{align*}
  R_1^n\,\Gamma_n^{[\ell]}
 &= (\xi\,\tau\,R_1)\Phi_{\ell-1}\,
 \bigg(\frac{1}{\xi} \sum_{k=0}^{\ell-1} (\xi\,\tau\,R_1)\Phi_k\bigg)^{n-1} n! \\
 &= \Pi_{\ell}\,
 \bigg(\frac{1}{\xi}\sum_{k=1}^{\ell} \Pi_k\bigg)^{n-1} n!
 = \frac{\Pi_\ell}{S_\ell}\,S_\ell^n\,n!
 \le C_\ell\,S_\ell^n\,n!\,.
\end{align*}
Substituting this into \eqref{eq:dy-np} and \eqref{eq:dy-per} yields 
\eqref{eq:dy-np-An} and \eqref{eq:dy-per-An} for the case $\bsnu\ne\bszero$. The case $\bsnu = \bszero$ holds trivially since $\|G_\theta\|_\infty \le C_L$. This completes the proof of Theorem~\ref{thm:reg-y-An}.
\end{proofof}

\begin{proofof}{Theorem~\ref{thm:rate-y}}
We continue the proof from the expression \eqref{eq:four}. 
For the affine field, we use \eqref{eq:tar-np} and \eqref{eq:dy-np-An} to further bound \eqref{eq:four} by
\begin{align*} 
 &\sum_{\bsm\le\bsnu} \binom{\bsnu}{\bsm}
 \big(2\,C\, |\bsm|!\, (T\,\bsbeta)^\bsm\big)
 \big(2\,C\, |\bsnu-\bsm|!\, (T\,\bsbeta)^{\bsnu-\bsm} \big)
 = 4\,C^2\,(|\bsnu|+1)!\, (T\,\bsbeta)^\bsnu.
\end{align*}
For the Chebyshev field, we use \eqref{eq:tar-per} and \eqref{eq:dy-per-An} to further bound \eqref{eq:four} by
\begin{align*}
 &\sum_{\bsm\le\bsnu} \binom{\bsnu}{\bsm}
 \bigg(2\,C\,(2\pi)^{|\bsm|} \sum_{\bsw\le\bsm} |\bsw|!\,(T\,\bsbeta)^\bsw\,\calS(\bsm,\bsw)\!\bigg) \\
 &\qquad\qquad\qquad\times
 \bigg(2\,C\,(2\pi)^{|\bsnu-\bsm|} \sum_{\bsu\le\bsnu-\bsm}
 |\bsu|!\,(T\,\bsbeta)^\bsu\, \calS(\bsnu-\bsm,\bsu)\!\bigg) \\
 &= 4\,C^2\,(2\pi)^{|\bsnu|} \sum_{\bsm\le\bsnu}
 \bigg(\sum_{\bsw\le\bsm} \binom{\bsm}{\bsw}\,|\bsw|!\,|\bsm-\bsw|!\bigg)
 (T\,\bsbeta)^{\bsm}\, \calS(\bsnu,\bsm) \\
 &= 4\,C^2\,(2\pi)^{|\bsnu|} \sum_{\bsm\le\bsnu}
 (|\bsm|+1)!\, (T\,\bsbeta)^{\bsm}\, \calS(\bsnu,\bsm).
\end{align*}
In the above we made use of two identities from \citet[Formula~(9.4)]{KN16} and
\citet[Lemma~A.3]{HHKKS24} as in the proof of \citet[Theorem~3.2]{KKNS26a}. These bounds yield the norm bounds \eqref{eq:final-np} and \eqref{eq:final-per} as required. 

Since the norm bounds \eqref{eq:final-np} and \eqref{eq:final-per} take the same form as those in \citet[Theorem~3.2]{KKNS26a}, we conclude the appropriate choice of weights from \citet[Theorem~3.3]{KKNS26a}.
\end{proofof}

\subsection{Regularity Proofs with Respect to \texorpdfstring{$\bsx$}{$x$}}
\label{sec:regularity_proofs_x}
We now switch to consider regularity with respect to $\bsx$. As in the previous section, we begin without using the specific form \eqref{eq:common} for the activation function.

Note that the sequence $\widetilde\Gamma_n^{[\ell]}$ defined in \eqref{eq:G-def-x} below differs from \eqref{eq:G-def} due to fundamental differences between the derivatives of $a(\bsx,\bsy)$ with respect to $\bsy$ and $\bsx$. Specifically, the sequence now starts from index $\ell=0$ rather than $1$, with $\widetilde\Gamma_n^{[0]}$ being a constant $\|P\|_\infty J$ independent of $n$.

\begin{theorem}[Regularity bound for FNO in $\bsx$] \label{thm:reg-x}
Consider the FNO \eqref{eq:recur} with either the affine \eqref{eq:a-np} or Chebyshev~\eqref{eq:a-per} field, with basis functions \eqref{eq:psi} and index set \eqref{eq:index} for some $M>0$. Let $J$ be defined in \eqref{eq:dx-a} for some $\omega\ge 1$ and define $c := \|P\|_\infty J$. Suppose there exist positive sequences $(R_\ell)_{\ell\ge 1}$ and $(A_n)_{n\ge 1}$ such that \eqref{eq:ass-y} holds, but do not assume \eqref{eq:common}.
Define the sequence $\widetilde\Gamma_n^{[\ell]}$ recursively by
\begin{align}
  &\begin{cases}
  \widetilde\Gamma_n^{[0]} := c & \mbox{for $n\ge 1$}, \\
  \widetilde\Gamma_n^{[\ell]} := 
  \displaystyle\sum_{\lambda=1}^n A_\lambda\,R_\ell^\lambda\, 
  \widetilde\bbB_{n,\lambda}^{[\ell-1]}
  & \mbox{for $\ell\ge 1$ and $n\ge1$,}
  \end{cases} \label{eq:G-def-x}
  \\
  &\begin{cases}
  \widetilde\bbB_{n,1}^{[\ell]} := \widetilde\Gamma_n^{[\ell]} &\mbox{for $\ell\ge 0$ and $n\ge 1$}, \\
  \widetilde\bbB_{n,\lambda}^{[\ell]} :=
  \displaystyle\sum_{i=\lambda-1}^{n-1} \binom{n-1}{i}\, 
  \widetilde\Gamma_{n-i}^{[\ell]}\,\widetilde\bbB_{i,\lambda-1}^{[\ell]}
  & \mbox{for $\ell\ge 0$ and $n\ge\lambda\ge 2$}. 
  \end{cases} \label{eq:B-def-x}
\end{align}
For $L\ge 1$ and $\bsnu\in\indx$ with $\bsnu\ne\bszero$ and $\|\bsnu\|_\infty \le \omega$, we have the following regularity bound with respect to $\bsx$:
\begin{align} \label{eq:dx}
 \|\partial^\bsnu_\bsx G_{\theta}\|_\infty
 \le \|Q\|_\infty\,(2\pi)^{|\bsnu|}\, \widetilde\Gamma_{|\bsnu|}^{[L]}.
\end{align}
\end{theorem}

\begin{proof}
We return to the function $g^{[1]}(\bsx,\bsy)$ in \eqref{eq:recur}. For any $\bsnu\ne\bszero$ with $\|\bsnu\|_\infty \le \omega$ and any index $1\le p\le d_2$, we use \eqref{eq:dx-a} and the assumption \eqref{eq:ass-y} to obtain
\begin{align*}
  |\partial^\bsnu_\bsx g^{[1]}_p(\bsx,\bsy)|
  &= \bigg|
   \sum_{q=1}^{d_1} W_{1,p,q}\, P_{q,1}\,\partial^\bsnu_\bsx a(\bsx,\bsy) 
   + 
   \int_D \sum_{q=1}^{d_1} \kappa_1(\bsx-\bsx')_{p,q} \,P_{q,1}\,
   \partial^\bsnu_{\bsx'} a(\bsx',\bsy)\,\rd\bsx'
  \bigg| \\
  &\le \sum_{q=1}^{d_1} |W_{1,p,q}|\, \|P\|_\infty\,(2\pi)^{|\bsnu|} J
   + \sum_{q=1}^{d_1} \sum_{\bsh\in\calA} |(\widehat{\kappa_1})_{p,q}(\bsh)|\,\|P\|_\infty\,
   (2\pi)^{|\bsnu|} J
  \\
  &\le  (2\pi)^{|\bsnu|}\, 
  \bigg(\|W_1\|_\infty
  + \sum_{\bsh\in\calA} \|\widehat{\kappa_1}(\bsh)\|_\infty
  \bigg)
  \|P\|_\infty J 
  \;\le (2\pi)^{|\bsnu|}\, R_1\, \widetilde\Gamma_{|\bsnu|}^{[0]},
\end{align*}
where $\|P\|_\infty J =c = \widetilde\Gamma_n^{[0]}$ for any $n\ge 1$ from \eqref{eq:G-def-x}. Since our definition of $\widetilde\Gamma_n^{[0]}$ is independent of~$n$, this choice of $n = |\bsnu|$ is arbitrary, and is chosen for the benefit of the later induction.

Suppose as induction hypothesis that for some $\ell\ge 2$ and $1\le q\le d_\ell$ we have
\begin{align*} 
|\partial^\bsnu_\bsx g^{[\ell-1]}_q(\bsx,\bsy)|
  \le  (2\pi)^{|\bsnu|}\, R_{\ell-1}\,\widetilde\Gamma_{|\bsnu|}^{[\ell-2]}.
\end{align*}
The recursive form of Fa\`a di Bruno formula \citep[Formulas~(3.1) and~(3.5)]{Savits06} gives
\begin{align} \label{eq:dx-sigma-g}
  \partial^\bsnu_\bsx \big(\sigma(g^{[\ell-1]}_q(\bsx,\bsy))\big)
  = \sum_{\lambda=1}^{|\bsnu|} \sigma^{(\lambda)}(g^{[\ell-1]}_q(\bsx,\bsy))\,
  \alpha_{\bsnu,\lambda}(\bsx,\bsy),
\end{align}
where the auxiliary functions $\alpha_{\bsnu,\lambda}(\bsx,\bsy)$ are defined recursively by
$\alpha_{\bsnu,0}(\bsx,\bsy):= \delta_{\bsnu,\bszero}$, $\alpha_{\bsnu,\lambda}(\bsx,\bsy):=0$ for $\lambda>|\bsnu|$, and otherwise 
\begin{align*}
 \alpha_{\bsnu+\bse_j,\lambda}(\bsx,\bsy)
 := \sum_{\bsm\le\bsnu} \binom{\bsnu}{\bsm} \partial^{\bsnu-\bsm+\bse_j}_\bsx g^{[\ell-1]}_q(\bsx,\bsy)\,
  \alpha_{\bsm,\lambda-1}(\bsx,\bsy).
\end{align*}
Thus
\begin{align*} 
 |\alpha_{\bsnu+\bse_j,\lambda}(\bsx,\bsy)|
 &\le \sum_{\bsm\le\bsnu} \binom{\bsnu}{\bsm}\, 
  \Big( (2\pi)^{|\bsnu-\bsm|+1}\, R_{\ell-1}\,\widetilde\Gamma_{|\bsnu-\bsm|+1}^{[\ell-2]}\Big)\,
  |\alpha_{\bsm,\lambda-1}(\bsx,\bsy)|.
\end{align*}
Using \citet[Lemma~6.1]{KKNS26a}, we obtain
\begin{align} \label{eq:alpha}
 |\alpha_{\bsnu,\lambda}(\bsx,\bsy)|
 &\le R_{\ell-1}^\lambda\,(2\pi)^{|\bsnu|}\,\widetilde\bbB_{|\bsnu|,\lambda}^{[\ell-2]},
\end{align}
with the sequence $\bbB_{|\bsnu|,\lambda}^{[\ell-2]}$ defined as in \eqref{eq:B-def-x}.
Combining \eqref{eq:alpha} with \eqref{eq:dx-sigma-g} and using the bound $|\sigma^{(\lambda)}(z)|\le A_\lambda$, we obtain
\begin{align*}
  |\partial^\bsnu_\bsx \big(\sigma(g^{[\ell-1]}_q(\bsx,\bsy))\big)|
  \le \sum_{\lambda=1}^{|\bsnu|} A_\lambda\,
  R_{\ell-1}^\lambda\,(2\pi)^{|\bsnu|}\,\widetilde\bbB_{|\bsnu|,\lambda}^{[\ell-2]} 
  = (2\pi)^{|\bsnu|}\, \widetilde\Gamma_{|\bsnu|}^{[\ell-1]},
\end{align*}
where we used the definition of $\widetilde\Gamma_{|\bsnu|}^{[\ell-1]}$ from \eqref{eq:G-def-x}.

Thus for any $1\le p\le d_{\ell+1}$ we have
\begin{align*} 
 |\partial^\bsnu_\bsx g^{[\ell]}_p(\bsx,\bsy)| 
 &= \bigg| \sum_{q=1}^{d_\ell} W_{\ell,p,q}\,\partial^\bsnu_\bsx\big(\sigma(g^{[\ell-1]}_q(\bsx,\bsy))\big) 
 + \int_D \sum_{q=1}^{d_\ell} (\kappa_\ell)_{p,q}(\bsx-\bsx') \,\partial^\bsnu_{\bsx'}
 \big(\sigma(g^{[\ell-1]}_q(\bsx',\bsy))\big)\,\rd\bsx'
 \bigg| \\
 &\le \sum_{q=1}^{d_\ell} |W_{\ell,p,q}|\, (2\pi)^{|\bsnu|}\,\widetilde\Gamma_{|\bsnu|}^{[\ell-1]}
 + \sum_{q=1}^{d_\ell} \sum_{\bsh\in\calA} |(\widehat{\kappa_\ell})_{p,q}(\bsh)| \, 
 (2\pi)^{|\bsnu|}\,\widetilde\Gamma_{|\bsnu|}^{[\ell-1]} 
  \;\le  (2\pi)^{|\bsnu|}\, R_\ell \, \widetilde\Gamma_{|\bsnu|}^{[\ell-1]}.
\end{align*}
In turn we can then show that
\begin{align*}
  |\partial^\bsnu_\bsx \big(\sigma(g^{[\ell]}_p(\bsx,\bsy))\big)|
  \le (2\pi)^{|\bsnu|}\, \widetilde\Gamma_{|\bsnu|}^{[\ell]}.
\end{align*}

Finally, for every $1\le p\le d_u$ we have
\begin{align*}
 |\partial^\bsnu_\bsx G_{\theta,p}(\bsx,\bsy)|
 = \bigg| \sum_{q=1}^{d_{L+1}} Q_{p,q}\,\partial^\bsnu_\bsx\big(\sigma(g^{[L]}_q(\bsx,\bsy))\big) 
 \bigg|
 \le \|Q\|_\infty\,(2\pi)^{|\bsnu|}\, \widetilde\Gamma_{|\bsnu|}^{[L]}.
\end{align*}
This completes the proof of Theorem~\ref{thm:reg-x}.
\end{proof}

\begin{lemma} \label{lem:complex}
Let $\xi>0$ and $\tau>0$. If $A_n = \xi\,\tau^n\,n!$ for
$n\ge 1$, $R_\ell > 0$ for $\ell\ge 1$, and $\widetilde\Gamma_n^{[0]} = c >0$, then the sequences defined
recursively in \eqref{eq:G-def-x}--\eqref{eq:B-def-x} are given explicitly by
\begin{align}
  \widetilde\Gamma_n^{[\ell]}
  = \widetilde\bbB_{n,1}^{[\ell]}
  &= \mathfrak{a}_\ell \sum_{k=1}^n k!\, \calS(n,k)\, \mathfrak{b}_\ell^{k-1} 
 \quad\mbox{for $\ell\ge 0$ and $n\ge 1$}, \label{eq:G-exp-x} \\
  \widetilde\bbB_{n,\lambda}^{[\ell]}
  &= \frac{\mathfrak{a}_{\ell}^\lambda}{\lambda!}
  \sum_{k=\lambda}^n \binom{k-1}{\lambda-1}\, k!\, \calS(n,k)\,\mathfrak{b}_\ell^{k-\lambda}
  \quad\mbox{for $\ell\ge 0$ and $n\ge\lambda\ge 1$}, \label{eq:B-exp-x}
\end{align}
with
\begin{align} \label{eq:P-def-x}
 \mathfrak{a}_\ell := c \prod_{t=1}^{\ell} (\xi\,\tau R_t) 
  \quad\mbox{and}\quad 
  \mathfrak{b}_\ell := \frac{c}{\xi} \sum_{k=1}^{\ell} \prod_{t=1}^k (\xi\,\tau R_t)
  = \frac{1}{\xi} \sum_{k=1}^\ell \mathfrak{a}_k
  \quad\mbox{for $\ell\ge 0$},
\end{align}
where an empty product is $1$, an empty sum is $0$, and $0^0 :=1$, so that $\mathfrak{a}_0 = c$, $\mathfrak{b}_0 = 0$, and $\widetilde\bbB_{n,\lambda}^{[0]} = c^\lambda\,\calS(n,\lambda)$ for $n\ge\lambda\ge 1$.
A simple upper bound is given by
\begin{align} \label{eq:G-bound-x}
  \widetilde\Gamma_n^{[\ell]} 
  =  \frac{\mathfrak{a}_\ell}{\mathfrak{b}_\ell} F_n(\mathfrak{b}_\ell)  
  \le 
  \frac{\mathfrak{a}_\ell}{\mathfrak{b}_\ell} \,(1+\mathfrak{b}_\ell)^n\,n! 
  \quad\mbox{for $\ell\ge 1$},
\end{align}
where $F_n(z) := \sum_{k=0}^n k!\, \calS(n,k)\,z^k$ is the Fubini polynomial (a.k.a.\ the ordered Bell polynomial); here the $k=0$ term is zero since $n\ge 1$ and $\calS(n,0)=0$.
\end{lemma}

\begin{proof}
We define the exponential generating functions
\begin{align*}
 \mathscr{G}^{[\ell]}(z) := \sum_{n=1}^\infty \widetilde\Gamma^{[\ell]}_n\, \frac{z^n}{n!}
 \quad\mbox{and}\quad
 \mathscr{B}^{[\ell]}_\lambda(z) := \sum_{n=\lambda}^\infty \widetilde\bbB^{[\ell]}_{n,\lambda}\, \frac{z^n}{n!}
 \quad\mbox{for $\ell\ge 0$ and $\lambda\ge 1$}.
\end{align*}
From the recursion \eqref{eq:B-def-x} with $\lambda=1$ we see that $\mathscr{B}^{[\ell]}_1(z)=\mathscr{G}^{[\ell]}(z)$, and for $\lambda\ge 2$ we have
\begin{align*}
 (\mathscr{G}^{[\ell]})'(z)\,\mathscr{B}^{[\ell]}_{\lambda-1}(z)
 &= \bigg(\sum_{m=1}^\infty \widetilde\Gamma^{[\ell]}_m\,\frac{z^{m-1}}{(m-1)!}\bigg)
 \bigg(\sum_{i=\lambda-1}^\infty \widetilde\bbB^{[\ell]}_{i,\lambda-1}\,\frac{z^i}{i!}\bigg) \\
 &= \sum_{i=\lambda-1}^\infty 
 \sum_{n=i+1}^\infty \widetilde\Gamma^{[\ell]}_{n-i}\,\frac{z^{n-i-1}}{(n-i-1)!}\,
 \widetilde\bbB^{[\ell]}_{i,\lambda-1}\,\frac{z^i}{i!} \\
 &= \sum_{n=\lambda}^\infty \underbrace{\sum_{i=\lambda-1}^{n-1}
 \binom{n-1}{i}\,\widetilde\Gamma^{[\ell]}_{n-i}\,\widetilde\bbB^{[\ell]}_{i,\lambda-1}}_{=\,\widetilde\bbB^{[\ell]}_{n,\lambda}}
 \frac{z^{n-1}}{(n-1)!} 
 = (\mathscr{B}^{[\ell]}_{\lambda})'(z),
\end{align*}
where we swapped the sums over $m$ and $i$, substituted $n = m+i$, and then swapped the sums over $i$ and $n$. Integrating recursively yields
\begin{align} \label{eq:GB-rel}
 \mathscr{B}^{[\ell]}_{\lambda}(z) = \frac{(\mathscr{G}^{[\ell]}(z))^\lambda}{\lambda!}
 \quad\mbox{for $\ell\ge 0$ and $\lambda\ge 1$}.
\end{align}

Using the definition of $\widetilde\Gamma^{[\ell]}_n$ from \eqref{eq:G-def-x}, we obtain for $\ell\ge 1$,
\begin{align} \label{eq:mobius1}
 \mathscr{G}^{[\ell]}(z) 
 &= \sum_{n=1}^\infty \underbrace{\sum_{\lambda=1}^n A_\lambda\, R_\ell^\lambda\,
 \widetilde\bbB_{n,\lambda}^{[\ell-1]}}_{=\,\widetilde\Gamma_n^{[\ell]}}
 \, \frac{z^n}{n!} 
 = \sum_{\lambda=1}^\infty A_\lambda\, R_\ell^\lambda\, \underbrace{\sum_{n=\lambda}^\infty 
 \widetilde\bbB_{n,\lambda}^{[\ell-1]}\, \frac{z^n}{n!}}_{=\,\mathscr{B}_\lambda^{[\ell-1]}(z)} \nonumber\\
 &= \sum_{\lambda=1}^\infty \underbrace{\xi\,\tau^\lambda\,\lambda!}_{=\,A_\lambda}\,R_\ell^\lambda\,
 \frac{(\mathscr{G}^{[\ell-1]}(z))^{\lambda}}{\lambda!} 
 = \xi \sum_{\lambda=1}^\infty (\tau R_\ell\,\mathscr{G}^{[\ell-1]}(z))^{\lambda}
 = \frac{\xi\,\tau R_\ell\,\mathscr{G}^{[\ell-1]}(z)}{1-\tau R_\ell\,\mathscr{G}^{[\ell-1]}(z)},
\end{align}
where we swapped the sums over $n$ and $\lambda$, substituted $A_\lambda = \xi\,\tau^\lambda\,\lambda!$, applied \eqref{eq:GB-rel}, and then used the geometric series formula as a formal identity. The last step is valid because $\mathscr{G}^{[\ell-1]}(0) = 0$ and so there is always a neighborhood around $z=0$ for which $|\tau\,R_\ell\,\mathscr{G}^{[\ell-1]}(z)| < 1$. 

Next we recognize that the recursion \eqref{eq:mobius1} can be viewed as applying a M\"obius transform $f(z) = \frac{p\,z+q}{r\,z+s}$ in each step, where the action can be identified with the matrix $(\begin{smallmatrix} p & q \\ r & s \end{smallmatrix})$. Specifically, we have
\[
  M_\ell = \begin{pmatrix} \xi\,\tau R_\ell & 0 \\ -\tau R_\ell & 1 \end{pmatrix}
  \quad\mbox{for $\ell\ge 1$},
\]
and the total transformation from $\mathscr{G}^{[0]}(z)$ to $\mathscr{G}^{[\ell]}(z)$ is equivalent to the matrix product 
\[
 \begin{pmatrix} \Pi_\ell & 0 \\ - S_\ell & 1 \end{pmatrix}
 := M_\ell\,M_{\ell-1}\cdots M_1 
 = \begin{pmatrix} \xi\,\tau R_\ell & 0 \\ -\tau R_\ell & 1 \end{pmatrix}
 \begin{pmatrix} \Pi_{\ell-1} & 0 \\ - S_{\ell-1} & 1 \end{pmatrix}.
\]
We conclude that $\Pi_\ell = \xi\,\tau R_\ell\,\Pi_{\ell-1}$, $S_\ell = \tau R_\ell\,\Pi_{\ell-1} + S_{\ell-1}$, with $\Pi_0 = 1$, $S_0 = 0$, and we have
\begin{align} \label{eq:mobius2}
 \mathscr{G}^{[\ell]}(z) 
 = \frac{\Pi_\ell\,\mathscr{G}^{[0]}(z)}{1-S_\ell\,\mathscr{G}^{[0]}(z)},
\end{align}

Now we make use of the base step $\widetilde\Gamma_n^{[0]} = c$ to obtain $\mathscr{G}^{[0]}(z) = c \sum_{n=1}^\infty \frac{z^n}{n!} = c\,(e^z-1)$. Substituting this into \eqref{eq:mobius2} gives
\begin{align} \label{eq:mobius3}
 \mathscr{G}^{[\ell]}(z) 
 = \frac{\Pi_\ell\,c\,(e^z-1)}{1 - S_\ell\,c\,(e^z-1)}
 = \frac{\mathfrak{a}_\ell\,(e^z-1)}{1 - \mathfrak{b}_\ell\,(e^z-1)}
 \quad\mbox{for $\ell\ge 0$},
\end{align}
where $\mathfrak{a}_\ell := \Pi_\ell\,c$ and $\mathfrak{b}_\ell := S_\ell\,c$. This yields the definitions in \eqref{eq:P-def-x}.


Substituting \eqref{eq:mobius3} into \eqref{eq:GB-rel}, we obtain for $\ell\ge 0$ and $\lambda\ge 1$,
\begin{align*} 
 \mathscr{B}^{[\ell]}_{\lambda}(z) 
 = \frac{1}{\lambda!}
 \bigg(\frac{\mathfrak{a}_\ell\,(e^z-1)}{1 - \mathfrak{b}_\ell\,(e^z-1)}\bigg)^{\lambda}
 &= \frac{\mathfrak{a}_\ell^\lambda}{\lambda!} \sum_{j=0}^\infty \binom{\lambda+j-1}{\lambda-1}\,
  \mathfrak{b}_\ell^j\,(e^z-1)^{\lambda+j} \\
 &= \frac{\mathfrak{a}_\ell^\lambda}{\lambda!} \sum_{k=\lambda}^\infty \binom{k-1}{\lambda-1}\,
  \mathfrak{b}_\ell^{k-\lambda}\,
  \underbrace{k! \sum_{n=k}^\infty \calS(n,k)\,\frac{z^n}{n!}}_{=\,(e^z-1)^k} \\
 &= \sum_{n=\lambda}^\infty \bigg(
  \underbrace{\frac{\mathfrak{a}_\ell^\lambda}{\lambda!} \sum_{k=\lambda}^n \binom{k-1}{\lambda-1}\,
  \mathfrak{b}_\ell^{k-\lambda}\,k!\, \calS(n,k)}_{=\,\widetilde\bbB_{n,\lambda}^{[\ell]}} 
  \bigg)\frac{z^n}{n!},
\end{align*}
where we used the negative binomial series $\frac{1}{(1-u)^\lambda} = \sum_{j=0}^\infty \binom{\lambda+j-1}{\lambda-1} \,u^j$ for $|u| = |\mathfrak{b}_\ell\,(e^z-1)| <1$ in a neighborhood of $z=0$  \citep[see e.g.,][Table~335]{GKP89}, substituted $k = \lambda+j$, used the identity $(e^z-1)^k = k! \sum_{n=k}^\infty \calS(n,k)\,\frac{z^n}{n!}$ \citep[see e.g.,][Formula~(7.49)]{GKP89}, and swapped the sums over $k$ and~$n$.
This proves the formula for $\widetilde\bbB_{n,\lambda}^{[\ell]}$ in~\eqref{eq:B-exp-x}, including the formula for $\widetilde\Gamma_n^{[\ell]} = \widetilde\bbB_{n,1}^{[\ell]}$ in~\eqref{eq:G-exp-x}.

Finally the simple upper bound on $\widetilde\Gamma_n^{[\ell]}$ is obtained by using 
$k!\,\calS(n,k) \le n!\,\binom{n-1}{k-1} \le n!\,\binom{n}{k}$ for $n\ge k\ge 1$ so that
$F_n(z) \le \sum_{k=1}^n n!\,\binom{n}{k} z^k \le (1+z)^n\,n!$
for all $z>0$.

Alternatively, we can show that $F_n(z) \le F_n(1) < n!/(\ln 2)^{n+1}$ for $z\in (0,1]$, and $F_n(z) \le z^n\, F_n(1) \le z^n\,n!/(\ln 2)^{n+1}$ for $z\ge 1$, and thus $F_n(z) \le \max(1,z^n)\,n!/(\ln 2)^{n+1}$ for all $z>1$.
\end{proof}

\begin{proofof}{Theorem~\ref{thm:reg-x-An}}
From the upper bound \eqref{eq:G-bound-x} in Lemma~\ref{lem:complex} and using the definitions in \eqref{eq:CSP-x}, we have
\begin{align*}
 \|Q\|_\infty\,\widetilde\Gamma_n^{[L]}
 \le \|Q\|_\infty\,
 \frac{\xi\, \Pi_L}{\sum_{k=1}^L \Pi_k}\,\bigg(1 + \frac{\|P\|_\infty J}{\xi} \sum_{k=1}^L \Pi_k\bigg)^n\,n!
 \le \widetilde{C}_L\,\widetilde{S}_L^n\,n!\,,
\end{align*}
where we used the definitions in \eqref{eq:CSP-x}. Substituting this into \eqref{eq:dx} yields \eqref{eq:dx-An} for the case $\bsnu\ne\bszero$. The case $\bsnu = \bszero$ holds trivially since $\|G_\theta\|_\infty \le \widetilde{C}_L$. This completes the proof of Theorem~\ref{thm:reg-x-An}.
\end{proofof}

\begin{proofof}{Theorem~\ref{thm:rate-x}}
Continuing from \eqref{eq:four-x}, we use \eqref{eq:dx-An} and \eqref{eq:tar-x} to obtain
\begin{align*}
 \|\partial^\bsnu_\bsx (G-G_\theta)^2\|_{L_2(D)}
 &\le \sum_{\bsm\le\bsnu} \binom{\bsnu}{\bsm} 
 \Big( 2\,\widetilde{C}\,(2\pi\,\widetilde{T})^{|\bsm|}\,|\bsm|! \Big)\,
 \Big(2\, \widetilde{C}\,(2\pi\,\widetilde{T})^{|\bsnu-\bsm|}\,|\bsnu-\bsm|! \Big) \\
 &= 4\,\widetilde{C}^2\, (2\pi\,\widetilde{T})^{|\bsnu|} 
 \sum_{\bsm\le\bsnu} \binom{\bsnu}{\bsm}\, |\bsm|!\,|\bsnu-\bsm|! 
 = 4\,\widetilde{C}^2\, (2\pi\,\widetilde{T})^{|\bsnu|} (|\bsnu|+1)!\,,
\end{align*}
which yields the norm bound \eqref{eq:final-x}.

Since this norm bound again takes the same form as those in \citet[Theorem~3.2]{KKNS26a}, we can choose weights again following \citet[Theorem~3.3]{KKNS26a}. This is equivalent to choosing $\rho_\setu$ to equate the summand over $\setu$ in \eqref{eq:wce-per-D} and \eqref{eq:final-x}, 
\[
  \rho_\setu^\lambda\,
  \bigg[\frac{2\zeta(2\omega\lambda)}{(2\pi)^{2\omega\lambda}}
  \bigg]^{|\setu|}
  = \frac{(2\pi\,\widetilde{T})^{2\omega|\setu|}\, [\,(\omega|\setu|+1)!\,]^2}{\rho_\setu}.
\]
Unlike for the regularity bound with respect to $\bsy$ where we chose parameters to ensure that the constant is independent of the potentially high dimensionality $s$, here the spatial dimension is likely to be relatively low, such as $d=2$ or $3$. Hence we just choose $\lambda$ to be arbitrarily close to $1/(2\omega)$, giving the convergence rate close to $m^{-\omega}$ and ignoring that the implied constant can depend on $d$. This completes the proof of Theorem~\ref{thm:rate-x}.
\end{proofof}

\subsection{Example of target function} \label{sec:example}

Algebraic equation $G(\bsx,\bsy)\,a(\bsx,\bsy) = 1$.
Consider the particular target function
\begin{equation*}
  G(\bsx,\bsy) := \frac{1}{a(\bsx,\bsy)} = g(a(\bsx,\bsy)), \quad g(z) := \frac{1}{z},
  \qquad \bsx \in D, \quad \bsy \in Y,
\end{equation*}
with $a(\bsx,\bsy)\ge a_{\min} > 0$. 
The recursive form of Fa\`a di Bruno formula \citep[see][formula (3.1) and (3.5)]{Savits06} gives
\begin{equation}\label{eq:d_x_recipr}
    \partial_\bsx^\bsnu G(\bsx, \bsy)
    = \sum_{\lambda=1}^{|\boldsymbol{\nu}|} g^{(\lambda)}(a(\bsx,\bsy)) \, \alpha_{\bsnu,\lambda}(\boldsymbol{x},\boldsymbol{y}),
\end{equation}
with $\alpha_{\bsnu,0}(\bsx,\bsy) := \delta_{\bsnu,\bszero}$, and $\alpha_{\boldsymbol{\nu},\lambda}(\bsx,\bsy) := 0$ for all $\lambda > |\boldsymbol{\nu}|$, and
\begin{equation*}
    \alpha_{\boldsymbol{\nu}+\boldsymbol{e}_j,\lambda}(\boldsymbol{x},\boldsymbol{y}):= \sum_{\boldsymbol{m} \leq \boldsymbol{\nu}} \begin{pmatrix}\boldsymbol{\nu} \\ \boldsymbol{m}\end{pmatrix} \Bigl(\partial_{\boldsymbol{x}}^{\boldsymbol{\nu}-\boldsymbol{m} + \boldsymbol{e}_j}a(\boldsymbol{x}, \boldsymbol{y})\Bigr) \, \alpha_{\boldsymbol{m},\lambda-1}(\boldsymbol{x},\boldsymbol{y}).
\end{equation*} 
Using \eqref{eq:dx-a}, this can be bounded as 
\begin{equation*}
    |\alpha_{\boldsymbol{\nu}+\boldsymbol{e}_j,\lambda}(\boldsymbol{x},\boldsymbol{y})| 
    \le \sum_{\boldsymbol{m} \leq \boldsymbol{\nu}} \begin{pmatrix}\boldsymbol{\nu} \\ \boldsymbol{m}\end{pmatrix} (2\pi)^{|\boldsymbol{\nu}-\boldsymbol{m}+\boldsymbol{e}_j|}\,
    J\, |\alpha_{\boldsymbol{m},\lambda-1}(\boldsymbol{x},\boldsymbol{y})|.
\end{equation*}
Using \citet[Lemma 6.1]{KKNS26a} with $R = J$, $\Gamma_n = 1$, and  
\begin{equation*}
    \mathbb{B}_{n,1} = \Gamma_n \quad for \, n \geq 1, \quad \mathbb{B}_{n,\lambda} := \sum_{i=\lambda-1}^{n-1}\begin{pmatrix} n-1 \\ i \end{pmatrix} \Gamma_{n-i}\, \mathbb{B}_{i,\lambda-1} \quad for \, n \geq \lambda \geq 2,
\end{equation*}
one can find 
\begin{equation*}
 |\alpha_{\boldsymbol{\nu}, \lambda}(\boldsymbol{x}, \boldsymbol{y})| 
 \leq J^{\lambda}\, (2\pi)^{|\boldsymbol{\nu}|} \, \mathbb{B}_{|\boldsymbol{\nu}|,\lambda}.
\end{equation*}
Lemma~\ref{lem:complex} with $c = 1$ and $\ell=0$ gives $\mathbb{B}_{|\boldsymbol{\nu}|,\lambda} = \calS(|\bsnu|,\lambda)$.
Using this one can bound \eqref{eq:d_x_recipr} as 
\begin{align*}
    |\partial_\bsx^\bsnu G(\bsx,\bsy)|
    &\le \sum_{\lambda=1}^{|\bsnu|} | g^{(\lambda)}(a(\bsx,\bsy))|\, 
    J^{\lambda}\, (2\pi)^{|\bsnu|} \, \calS(|\bsnu|,\lambda)  
    = (2\pi)^{|\bsnu|} \sum_{\lambda=1}^{|\bsnu|} 
    \bigg| \frac{(-1)^\lambda\,\lambda!}{(a(\bsx,\bsy))^{\lambda+1}}\bigg|\, 
    J^{\lambda}\, \calS(|\bsnu|,\lambda) \\ 
    &\le \tfrac{(2\pi)^{|\bsnu|}}{a_{\min}} \sum_{\lambda=1}^{|\bsnu|} 
    \lambda!\, \calS(|\bsnu|,\lambda)\, \big(\tfrac{J}{a_{\min}}\big)^\lambda
    = \tfrac{(2\pi)^{|\bsnu|}}{a_{\min}} F_{|\bsnu|}\big(\tfrac{J}{a_{\min}}\big) 
    \le \tfrac{(2\pi)^{|\bsnu|}}{a_{\min}} \big(1+\tfrac{J}{a_{\min}}\big)^{|\bsnu|}\,|\bsnu|!\,, 
\end{align*}
where we used the Fubini polynomial and its upper bound as we did in Lemma~\ref{lem:complex}.

\subsection{Extension to Nonlinear \texorpdfstring{$P$}{$P$} and \texorpdfstring{$Q$}{$Q$}} \label{sec:nonlinearPQ}

Instead of using \emph{linear} lifting and projecting operators $P$ and $Q$, these operators are in practice often taken to be \emph{nonlinear} shallow neural networks with $p$ and $q$ hidden layers, respectively. Thus, instead of \eqref{eq:no2} we have
\begin{align*}
 &G_\theta(\bsx,\bsy) \\
 &:= (\underbrace{\widetilde{T}_{p+L+q} \,\circ\, 
 T_{p+L+q}\,\circ\,\cdots\,\circ\, T_{p+L+1}}_{\rm projecting} 
 \,\circ\, 
 T_{p+L}\,\circ\,\cdots\,\circ\, \widetilde{T}_{p+1} \, \circ \, 
 \underbrace{\widetilde{T}_p \, \circ\, T_p\,\circ\,\cdots\,\circ\, 
 T_1}_{\rm lifting} \,\circ\,a)(\bsx,\bsy), 
\end{align*}
where the layers marked with lifting and projecting have no kernel term (i.e., $K_\ell \equiv 0$), and additionally $\widetilde{T}_{p+L+q}$ and $\widetilde{T}_p$ have no activation function.

Since there is no activation function between $\widetilde{T}_{p+1}$ and $\widetilde{T}_p$, for the purpose of regularity analysis with respect to $\bsx$ and $\bsy$, we can interpret this as one combined layer as follows:
\begin{align*}
 \widetilde{T}_{p+1} \,\circ\, \widetilde{T}_p (v) 
 &= \sigma \Big( \widetilde{W}_{p+1}( \widetilde{W}_p\,v + \widetilde{b}_p ) 
 + \widetilde{K}_{p+1} ( \widetilde{W}_p\,v + \widetilde{b}_p ) + \widetilde{b}_{p+1} \Big) 
 =: T_{p+1} (v),
\end{align*}
with $W_{p+1} := \widetilde{W}_{p+1} \widetilde{W}_p$, $K_{p+1}\,v := \widetilde{K}_{p+1} ( \widetilde{W}_p\,v)$, 
 $b_{p+1} := \widetilde{W}_{p+1}\,\widetilde{b}_p + \widetilde{K}_{p+1}\,\widetilde{b}_p + \widetilde{b}_{p+1}$. 
Hence we can adapt our regularity bounds by taking $P = I$ and $Q = \widetilde{T}_{p+L+q}$ and $p+L+q$ hidden layers~$T_\ell$. This operator $Q$ has an extra bias term compared with our original setting, but since the bias is constant with respect to $\bsx$ and $\bsy$ it does not change the regularity bounds. 



\bibliography{references}

\begin{thebibliography}{66}
\providecommand{\natexlab}[1]{#1}
\providecommand{\url}[1]{\texttt{#1}}
\expandafter\ifx\csname urlstyle\endcsname\relax
  \providecommand{\doi}[1]{doi: #1}\else
  \providecommand{\doi}{doi: \begingroup \urlstyle{rm}\Url}\fi

\bibitem[Adcock et~al.(2022)Adcock, Brugiapaglia, and Webster]{ABW22}
B.~Adcock, S.~Brugiapaglia, and C.~G. Webster.
\newblock \emph{Sparse {Polynomial} {Approximation} of {High-Dimensional
  Functions}}.
\newblock Society for Industrial and Applied Mathematics, Philadelphia, PA,
  2022.

\bibitem[Adcock et~al.(2024)Adcock, Brugiapaglia, Dexter, and Moraga]{ABDM24}
B.~Adcock, S.~Brugiapaglia, N.~Dexter, and S.~Moraga.
\newblock Chapter 1 - learning smooth functions in high dimensions: From sparse
  polynomials to deep neural networks.
\newblock In S.~Mishra and A.~Townsend, editors, \emph{Numerical Analysis Meets
  Machine Learning}, volume~25 of \emph{Handbook of Numerical Analysis}, pages
  1--52. Elsevier, 2024.

\bibitem[Bhattacharya et~al.(2021)Bhattacharya, Hosseini, Kovachki, and
  Stuart]{BHKS21}
K.~Bhattacharya, B.~Hosseini, N.~B. Kovachki, and A.~M. Stuart.
\newblock Model reduction and neural networks for parametric {PDEs}.
\newblock \emph{The SIAM Journal of computational mathematics}, 7:\penalty0
  121--157, 2021.

\bibitem[Bonev et~al.(2023)Bonev, Kurth, Hundt, Pathak, Baust, Kashinath, and
  Anandkumar]{BKHPBKA23}
B.~Bonev, T.~Kurth, C.~Hundt, J.~Pathak, M.~Baust, K.~Kashinath, and
  A.~Anandkumar.
\newblock Spherical {F}ourier neural operators: Learning stable dynamics on the
  sphere.
\newblock \emph{Proceedings of the 40th International Conference on Machine
  Learning}, 202:\penalty0 2806--2823, 2023.

\bibitem[Bonev et~al.(2025)Bonev, Kurth, Mahesh, Bisson, Kossaifi, Kashinath,
  Anandkumar, Collins, Pritchard, and Keller]{FourCastNet3}
B.~Bonev, T.~Kurth, A.~Mahesh, M.~Bisson, J.~Kossaifi, K.~Kashinath,
  A.~Anandkumar, W.~D. Collins, M.~S. Pritchard, and A.~Keller.
\newblock Four{C}ast{N}et 3: A geometric approach to probabilistic
  machine-learning weather forecasting at scale, 2025.
\newblock URL \url{https://arxiv.org/abs/2507.12144}.

\bibitem[Cao et~al.(2024)Cao, Goswami, and Karniadakis]{CGK24}
Q.~Cao, S.~Goswami, and G.~E. Karniadakis.
\newblock Laplace neural operator for solving differential equations.
\newblock \emph{Nature Machine Intelligence}, 6\penalty0 (6):\penalty0
  631--640, 2024.

\bibitem[Cohen et~al.(2010)Cohen, DeVore, and Schwab]{CDS10}
A.~Cohen, R.~DeVore, and C.~Schwab.
\newblock Convergence rates of best ${N}$-term {Galerkin} approximations for a
  class of elliptic {SPDE}s.
\newblock \emph{Foundations of Computational Mathematics}, 10:\penalty0
  615--646, 2010.

\bibitem[Cools et~al.(2006)Cools, Kuo, and Nuyens]{CKN06}
R.~Cools, F.~Y. Kuo, and D.~Nuyens.
\newblock Constructing embedded lattice rules for multivariate integration.
\newblock \emph{SIAM Journal on Scientific Computing}, 28\penalty0
  (6):\penalty0 2162--2188, 2006.

\bibitem[Cools et~al.(2010)Cools, Kuo, and Nuyens]{CKN10}
R.~Cools, F.~Y. Kuo, and D.~Nuyens.
\newblock Constructing lattice rules based on weighted degree of exactness and
  worst case error.
\newblock \emph{Computing}, 87\penalty0 (1):\penalty0 63--89, 2010.

\bibitem[Cools et~al.(2020)Cools, Kuo, Nuyens, and Sloan]{CKNS20}
R.~Cools, F.~Y. Kuo, D.~Nuyens, and I.~H. Sloan.
\newblock Lattice algorithms for multivariate approximation in periodic spaces
  with general weight parameters.
\newblock In S.~C. Brenner, I.~Shparlinski, C.-W. Shu, and D.~Szyld, editors,
  \emph{75 Years of Mathematics of Computation}, volume 754 of
  \emph{Contemporary Mathematics}, pages 93--113. American Mathematical
  Society, 2020.

\bibitem[Cools et~al.(2021)Cools, Kuo, Nuyens, and Sloan]{CKNS21}
R.~Cools, F.~Y. Kuo, D.~Nuyens, and I.~H. Sloan.
\newblock Fast component-by-component construction of lattice algorithms for
  multivariate approximation with {POD} and {SPOD} weights.
\newblock \emph{Mathematics of Computation}, 90:\penalty0 787--812, 2021.

\bibitem[Cybenko(1989)]{Cyb89}
G.~Cybenko.
\newblock Approximation by superpositions of a sigmoidal function.
\newblock \emph{Mathematics of Control, Signals and Systems}, 2:\penalty0
  303--314, 1989.

\bibitem[Dick et~al.(2013)Dick, Kuo, and Sloan]{DKS13}
J.~Dick, F.~Y. Kuo, and I.~H. Sloan.
\newblock High-dimensional integration: The quasi-{Monte Carlo} way.
\newblock \emph{Acta Numerica}, 22:\penalty0 133--288, 2013.

\bibitem[Dick et~al.(2022)Dick, Kritzer, and Pillichshammer]{DKP22}
J.~Dick, P.~Kritzer, and F.~Pillichshammer.
\newblock \emph{Lattice Rules: Numerical Integration, Approximation, and
  Discrepancy}, volume~58 of \emph{Springer Series in Computational
  Mathematics}.
\newblock Springer Cham, 2022.

\bibitem[Dilen et~al.(2026)Dilen, Kuo, and Nuyens]{DKN26}
J.~Dilen, F.~Y. Kuo, and D.~Nuyens.
\newblock On universal approximation for lattice-based hyperbolic-cross
  {F}ourier neural operators, in preparation, 2026.

\bibitem[Gilbert et~al.(2025)Gilbert, Kuo, and Srikumar]{GKS25}
A.~D. Gilbert, F.~Y. Kuo, and A.~Srikumar.
\newblock Density estimation for elliptic {PDE} with random input by
  preintegration and quasi-{M}onte {C}arlo methods.
\newblock \emph{SIAM Journal on Numerical Analysis}, 63\penalty0 (2):\penalty0
  1025--1054, 2025.

\bibitem[Gilbert et~al.(2026)Gilbert, Kuo, Nuyens, Pash, Sloan, and
  Willcox]{GKNPSW26}
A.~D. Gilbert, F.~Y. Kuo, D.~Nuyens, G.~Pash, I.~H. Sloan, and K.~E. Willcox.
\newblock Quasi-{M}onte {C}arlo methods for uncertainty quantification of tumor
  growth modeled by a parametric semi-linear parabolic reaction-diffusion
  equation.
\newblock \emph{arXiv:2509.25753}, 2026.

\bibitem[Graham et~al.(2026)Graham, Kuo, Nuyens, Sloan, and Spence]{GKNSS26}
I.~G. Graham, F.~Y. Kuo, D.~Nuyens, I.~H. Sloan, and E.~A. Spence.
\newblock Quasi-{M}onte {C}arlo methods for uncertainty quantification of wave
  propagation and scattering problems modelled by the {H}elmholtz equation.
\newblock \emph{IMA Journal of Numerical Analysis}, 2026.
\newblock \doi{10.1093/imanum/draf101}.
\newblock Advance article available online.

\bibitem[Graham et~al.(1994)Graham, Knuth, and Oren]{GKP89}
R.~L. Graham, D.~E. Knuth, and P.~Oren.
\newblock \emph{Concrete Mathematics}.
\newblock Addison-Wesley, 1994.

\bibitem[Guibas et~al.(2022)Guibas, Mardani, Li, Tao, Anandkumar, and
  Catanzaro]{JMZAAB22}
J.~Guibas, M.~Mardani, Z.~Li, A.~Tao, A.~Anandkumar, and B.~Catanzaro.
\newblock Efficient token mixing for transformers via adaptive {F}ourier neural
  operators.
\newblock In \emph{International Conference on Learning Representations}, 2022.
\newblock URL \url{https://openreview.net/forum?id=EXHG-A3jlM}.

\bibitem[Guth et~al.(2024)Guth, Kaarnioja, Kuo, Schillings, and Sloan]{GKKSS24}
P.~A. Guth, V.~Kaarnioja, F.~Y. Kuo, C.~Schillings, and I.~H. Sloan.
\newblock Parabolic {PDE}-constrained optimal control under uncertainty with
  entropic risk measure using quasi-{Monte Carlo} integration.
\newblock \emph{Numerische Mathematik}, 156:\penalty0 565--608, 2024.

\bibitem[Hakula et~al.(2024)Hakula, Harbrecht, Kaarnioja, Kuo, and
  Sloan]{HHKKS24}
H.~Hakula, H.~Harbrecht, V.~Kaarnioja, F.~Y. Kuo, and I.~H. Sloan.
\newblock Uncertainty quantification for random domains using periodic random
  variables.
\newblock \emph{Numerische Mathematik}, 156:\penalty0 273--317, 2024.

\bibitem[Hardt and Recht(2022)]{MB22}
M.~Hardt and B.~Recht.
\newblock \emph{Patterns, predictions, and actions: Foundations of machine
  learning}.
\newblock Princeton University Press, 2022.

\bibitem[Herrmann et~al.(2024)Herrmann, Schwab, and Zech]{HSZ24}
L.~Herrmann, C.~Schwab, and J.~Zech.
\newblock Neural network expression rates for solvers of parametric {O}rdinary
  {D}ifferential {E}quations.
\newblock \emph{Advances in Computational Mathematics}, 50\penalty0
  (5):\penalty0 91, 2024.

\bibitem[Hornik(1991)]{Hor91}
K.~Hornik.
\newblock Approximation capabilities of multilayer feedforward networks.
\newblock \emph{Neural Networks}, 4\penalty0 (2):\penalty0 251--257, 1991.

\bibitem[Hornik et~al.(1989)Hornik, Stinchcombe, and White]{HSW89}
K.~Hornik, M.~Stinchcombe, and H.~White.
\newblock Multilayer feedworward networks are universal approximators.
\newblock \emph{Neural Networks}, 2, 1989.

\bibitem[Jin et~al.(2022)Jin, Meng, and Lu]{PSL22}
P.~Jin, S.~Meng, and L.~Lu.
\newblock Mionet: Learning multiple-input operators via tensor product.
\newblock \emph{SIAM Journal on Scientific Computing}, 44\penalty0
  (6):\penalty0 A3490--A3514, 2022.

\bibitem[Kaarnioja et~al.(2020)Kaarnioja, Kuo, and Sloan]{KKS20}
V.~Kaarnioja, F.~Y. Kuo, and I.~H. Sloan.
\newblock Uncertainty quantification using periodic random variables.
\newblock \emph{SIAM Journal on Numerical Analysis}, 58\penalty0 (2):\penalty0
  1068--1091, 2020.

\bibitem[Kaarnioja et~al.(2022)Kaarnioja, Kazashi, Kuo, Nobile, and
  Sloan]{KKKNS22}
V.~Kaarnioja, Y.~Kazashi, F.~Y. Kuo, F.~Nobile, and I.~H. Sloan.
\newblock Fast approximation by periodic kernel-based lattice-point
  interpolation with application in uncertainty quantification.
\newblock \emph{Numerische Mathematik}, 150:\penalty0 33--77, 2022.

\bibitem[Kaarnioja et~al.(2024)Kaarnioja, Kuo, and Sloan]{KKS24}
V.~Kaarnioja, F.~Y. Kuo, and I.~H. Sloan.
\newblock Lattice-based kernel approximation and serendipitous weights for
  parametric {PDE}s in very high dimensions.
\newblock In A.~Hinrichs, P.~Kritzer, and F.~Pillichshammer, editors,
  \emph{{M}onte {C}arlo and {Q}uasi-{M}onte {C}arlo Methods 2022}, pages
  81--103. Springer-Verlag, 2024.

\bibitem[K{\"a}mmerer et~al.(2015)K{\"a}mmerer, Potts, and Volkmer]{KPV15}
L.~K{\"a}mmerer, D.~Potts, and T.~Volkmer.
\newblock Approximation of multivariate periodic functions by trigonometric
  polynomials based on rank-1 lattice sampling.
\newblock \emph{Journal of Complexity}, 31\penalty0 (4):\penalty0 543--576,
  2015.

\bibitem[K{\"a}mmerer et~al.(2021)K{\"a}mmerer, Potts, and Volkmer]{KPV21}
L.~K{\"a}mmerer, D.~Potts, and T.~Volkmer.
\newblock High-dimensional sparse {FFT} based on sampling along multiple
  rank-$1$ lattices.
\newblock \emph{Applied and Computational Harmonic Analysis}, 51:\penalty0
  225--257, 2021.

\bibitem[Keiner et~al.(2010)Keiner, Kunis, and Potts]{KKP10}
J.~Keiner, S.~Kunis, and D.~Potts.
\newblock Using {NFFT} 3: A software library for various nonequispaced fast
  fourier transforms.
\newblock \emph{ACM transactions on mathematical software}, 36\penalty0
  (4):\penalty0 46--75, 2010.

\bibitem[Keller et~al.(2026{\natexlab{a}})Keller, Kuo, Nuyens, and
  Sloan]{KKNS26a}
A.~Keller, F.~Y. Kuo, D.~Nuyens, and I.~H. Sloan.
\newblock Regularity and tailored regularization of deep neural networks, with
  application to parametric {PDEs} in uncertainty quantification.
\newblock \emph{Mathematics of Computation}, to appear, 2026{\natexlab{a}}.
\newblock URL \url{https://arxiv.org/abs/2502.12496}.

\bibitem[Keller et~al.(2026{\natexlab{b}})Keller, Kuo, Nuyens, and
  Sloan]{KKNS26b}
A.~Keller, F.~Y. Kuo, D.~Nuyens, and I.~H. Sloan.
\newblock Lattice-based deep neural networks: regularity and tailored
  regularization.
\newblock In Lemieux and B.~Feng, editors, \emph{{M}onte {C}arlo and
  {Q}uasi-{M}onte {C}arlo Methods 2024, to appear}. Springer-Verlag,
  2026{\natexlab{b}}.

\bibitem[Khoo et~al.(2021)Khoo, Lu, and Ying]{KLY21}
Y.~Khoo, J.~Lu, and L.~Ying.
\newblock Solving parametric {PDE} problems with artificial neural networks.
\newblock \emph{European Journal of Applied Mathematics}, 32\penalty0
  (3):\penalty0 421--435, 2021.

\bibitem[Kingma and Ba(2015)]{ADAM}
D.~P. Kingma and J.~Ba.
\newblock Adam: A method for stochastic optimization. {In} proceedings of the
  3rd international conference on learning representations.
\newblock \emph{ICLR}, 2015.

\bibitem[Kossaifi et~al.(2026)Kossaifi, Kovachki, Li, Pitt, Liu-Schiaffini,
  George, Bonev, Azizzadenesheli, Berner, Duruisseaux, and
  Anandkumar]{neuralop}
J.~Kossaifi, N.~Kovachki, Z.~Li, D.~Pitt, M.~Liu-Schiaffini, R.~J. George,
  B.~Bonev, K.~Azizzadenesheli, J.~Berner, V.~Duruisseaux, and A.~Anandkumar.
\newblock A library for learning neural operators, 2026.
\newblock URL \url{https://arxiv.org/abs/2412.10354}.

\bibitem[Kovachki et~al.(2021)Kovachki, Lanthaler, and Mishra]{KLM21}
N.~Kovachki, S.~Lanthaler, and S.~Mishra.
\newblock On universal approximation and error bounds for fourier neural
  operators.
\newblock \emph{Machine Learning Research}, 22:\penalty0 1--76, 2021.

\bibitem[Kovachki et~al.(2023)Kovachki, Li, Liu, Azizzadenesheli, Bhattacharya,
  Stuart, and Anandkumar]{KLLABSA23}
N.~Kovachki, Z.~Li, B.~Liu, K.~Azizzadenesheli, K.~Bhattacharya, A.~Stuart, and
  A.~Anandkumar.
\newblock Neural operator: learning maps between function spaces with
  applications to {PDE}s.
\newblock \emph{Journal of Machine Learning Research}, 24:\penalty0 4061--4157,
  2023.

\bibitem[Kuo and Nuyens(2016)]{KN16}
F.~Y. Kuo and D.~Nuyens.
\newblock Application of quasi-{Monte Carlo} methods to elliptic {PDE}s with
  random diffusion coefficients: A survey of analysis and implementation.
\newblock \emph{Foundations of Computational Mathematics}, 16:\penalty0
  1631--1696, 2016.

\bibitem[Kuo et~al.(2021)Kuo, Migliorati, Nobile, and Nuyens]{KMNN21}
F.~Y. Kuo, G.~Migliorati, F.~Nobile, and D.~Nuyens.
\newblock Function integration, reconstruction and approximation using rank-1
  lattices.
\newblock \emph{Mathematics of Computation}, 90\penalty0 (330):\penalty0
  1861--1897, 2021.

\bibitem[Kuo et~al.(2025)Kuo, Mo, and Nuyens]{KMN25}
F.~Y. Kuo, W.~Mo, and D.~Nuyens.
\newblock Constructing embedded lattice-based algorithms for multivariate
  function approximation with a composite number of points.
\newblock \emph{Constructive Approximation}, 61\penalty0 (1):\penalty0 81--113,
  2025.

\bibitem[Lanthaler(2023)]{Lan23}
S.~Lanthaler.
\newblock Operator learning with {PCA-Net}: upper and lower complexity bounds.
\newblock \emph{Journal of Machine Learning Research}, 24\penalty0
  (318):\penalty0 1--67, 2023.

\bibitem[Lanthaler et~al.(2025{\natexlab{a}})Lanthaler, Li, and Stuart]{LLS25}
S.~Lanthaler, Z.~Li, and A.~M. Stuart.
\newblock Nonlocality and nonlinearity implies universality in operator
  learning.
\newblock \emph{Constructive Approximation}, 62:\penalty0 261--303,
  2025{\natexlab{a}}.

\bibitem[Lanthaler et~al.(2025{\natexlab{b}})Lanthaler, Stuart, and
  Trautner]{LST25}
S.~Lanthaler, A.~M. Stuart, and M.~Trautner.
\newblock Discretization error of {F}ourier neural operators,
  2025{\natexlab{b}}.
\newblock URL \url{https://arxiv.org/abs/2405.02221}.

\bibitem[Li et~al.(2020{\natexlab{a}})Li, Kovachki, Azizzadenesheli, Liu,
  Bhattacharya, Stuart, and Anandkumar]{LKAKBSA20a}
Z.~Li, N.~Kovachki, K.~Azizzadenesheli, B.~Liu, K.~Bhattacharya, A.~Stuart, and
  A.~Anandkumar.
\newblock Neural operator: Graph kernel network for partial differential
  equations, 2020{\natexlab{a}}.
\newblock URL \url{https://arxiv.org/abs/2003.03485}.

\bibitem[Li et~al.(2020{\natexlab{b}})Li, Kovachki, Azizzadenesheli, Liu,
  Bhattacharya, Stuart, and Anandkumar]{LKALBSA20b}
Z.~Li, N.~Kovachki, K.~Azizzadenesheli, B.~Liu, K.~Bhattacharya, A.~Stuart, and
  A.~Anandkumar.
\newblock Multipole graph neural operator for parametric partial differential
  equations.
\newblock In \emph{Proceedings of the 34th International Conference on Neural
  Information Processing Systems}, NeurIPS '20, Red Hook, NY, USA,
  2020{\natexlab{b}}. Curran Associates Inc.

\bibitem[Li et~al.(2021)Li, Kovachki, Azizzadenesheli, Liu, Bhattacharya,
  Stuart, and Anandkumar]{LKALBSA21}
Z.~Li, N.~Kovachki, K.~Azizzadenesheli, B.~Liu, K.~Bhattacharya, A.~Stuart, and
  A.~Anandkumar.
\newblock Fourier neural operator for parametric partial differential
  equations.
\newblock \emph{ICLR}, 2021.

\bibitem[Liu et~al.(2024)Liu, Mao, and Zhou]{LMZ24}
Y.~Liu, T.~Mao, and D.-X. Zhou.
\newblock Approximation of functions from korobov spaces by shallow neural
  networks.
\newblock \emph{Information Sciences}, 2024.

\bibitem[Longo et~al.(2021)Longo, Mishra, Rusch, and Schwab]{LMRS21}
M.~Longo, S.~Mishra, T.~K. Rusch, and C.~Schwab.
\newblock Higher-order quasi-{Monte Carlo} training of deep neural networks.
\newblock \emph{SIAM Journal on Scientific Computing}, 43:\penalty0
  A3938--A3966, 2021.

\bibitem[Lu et~al.(2021)Lu, Jin, Pang, Zhang, and Karniadakis]{LJPZK21}
L.~Lu, P.~Jin, G.~Pang, Z.~Zhang, and G.~E. Karniadakis.
\newblock Learning nonlinear operators via {DeepONet} based on the universal
  approximation theorem of operators.
\newblock \emph{Nature Machine Intelligence}, 3:\penalty0 218 -- 229, 2021.

\bibitem[Lye et~al.(2020)Lye, Mishra, and Ray]{LMR20}
K.~O. Lye, S.~Mishra, and D.~Ray.
\newblock Deep learning observables in computational fluid dynamics.
\newblock \emph{Journal of Computational Physics}, 410:\penalty0 109339, 2020.

\bibitem[Mishra and Rusch(2021)]{MR21}
S.~Mishra and T.~K. Rusch.
\newblock Enhancing accuracy of deep learning algorithms by training with
  low-discrepancy sequences.
\newblock \emph{SIAM Journal on Numerical Analysis}, 59:\penalty0 1811--1834,
  2021.

\bibitem[Niederreiter(1992)]{Nie92}
H.~Niederreiter.
\newblock \emph{Random Number Generation and Quasi-Monte Carlo Methods}.
\newblock CBMS-NSF Regional Conference Series in Applied Mathematics. Society
  for Industrial and Applied Mathematics, Philadelphia, PA, 1992.

\bibitem[Novak et~al.(2024)Novak, Sharma, and Shields]{NSS24}
L.~Novak, H.~Sharma, and M.~D. Shields.
\newblock Physics-informed polynomial chaos expansions.
\newblock \emph{Journal of Computational Physics}, 506:\penalty0 112926, 2024.

\bibitem[Nuyens and Cools(2006)]{NC06}
D.~Nuyens and R.~Cools.
\newblock Fast algorithms for component-by-component construction of rank-1
  lattice rules in shift-invariant reprodicing kernel {Hilbert} spaces.
\newblock \emph{Mathematics of Computation}, 75\penalty0 (254), 2006.

\bibitem[O'Leary-Roseberry et~al.(2024)O'Leary-Roseberry, Chen, Villa, and
  Ghattas]{OCVG24}
T.~O'Leary-Roseberry, P.~Chen, U.~Villa, and O.~Ghattas.
\newblock Derivative-informed neural operator: An efficient framework for
  high-dimensional parametric derivative learning.
\newblock \emph{Journal of Computational Physics}, 496:\penalty0 112555, 2024.

\bibitem[Pestourie et~al.(2023)Pestourie, Mroueh, Rackauckas, Das, and
  Johnson]{PMRDJ23}
R.~Pestourie, Y.~Mroueh, C.~Rackauckas, P.~Das, and S.~G. Johnson.
\newblock Physics-enhanced deep surrogates for partial differential equations.
\newblock \emph{Nature Machine Intelligence}, 5\penalty0 (12), 2023.

\bibitem[Raissi et~al.(2019)Raissi, Perdikaris, and Karniadakis]{RPK19}
M.~Raissi, P.~Perdikaris, and G.~Karniadakis.
\newblock Physics-informed neural networks: A deep learning framework for
  solving forward and inverse problems involving nonlinear partial differential
  equations.
\newblock \emph{Journal of Computational Physics}, 378:\penalty0 686--707,
  2019.

\bibitem[Savits(2006)]{Savits06}
T.~H. Savits.
\newblock Some statistical applications of {Faa di Bruno}.
\newblock \emph{Journal of Multivariate Analysis}, 97\penalty0 (10):\penalty0
  2131--2140, 2006.

\bibitem[Schwab and Zech(2019)]{SZ19}
C.~Schwab and J.~Zech.
\newblock Deep learning in high dimension: Neural network expression rates for
  generalized polynomial chaos expansions in {UQ}.
\newblock \emph{Analysis and Applications}, 17\penalty0 (01):\penalty0 19--55,
  2019.

\bibitem[Schwab and Zech(2023)]{SZ23}
C.~Schwab and J.~Zech.
\newblock Deep learning in high dimension: Neural network expression rates for
  analytic functions in {$L^2(\mathbb{R}^d, \gamma_d)$}.
\newblock \emph{SIAM/ASA Journal on Uncertainty Quantification}, 11\penalty0
  (1):\penalty0 199--234, 2023.

\bibitem[Shahane et~al.(2019)Shahane, Aluru, and Pratap~Vanka]{SAP19}
S.~Shahane, N.~R. Aluru, and S.~Pratap~Vanka.
\newblock Uncertainty quantification in three dimensional natural convection
  using polynomial chaos expansion and deep neural networks.
\newblock \emph{International Journal of Heat and Mass Transfer}, 139:\penalty0
  613--631, 2019.

\bibitem[Sloan and Joe(1994)]{SJ94}
I.~H. Sloan and S.~Joe.
\newblock \emph{Lattice rules for multiple integration}.
\newblock Oxford University Press, 1994.

\bibitem[Voigtlaender(2023)]{Voi23}
F.~Voigtlaender.
\newblock The universal approximation theorem for complex-valued neural
  networks.
\newblock \emph{Applied and Computational Harmonic Analysis}, 64:\penalty0
  33--61, 2023.

\end{thebibliography}

\end{document}